\documentclass[]{amsart}

\usepackage[headings]{fullpage}

 \usepackage{amsmath}       % I think this gives me some symbols
 \usepackage{amsthm}        % Does theorem stuff
 \usepackage{amssymb}       % more symbols and fonts
 \usepackage{empheq}        % Some more extensible arrows, like \xmapsto
  \usepackage{bbm}
\usepackage{enumerate}		% This allows customization of the enumerate environment.
 \usepackage{mathrsfs}      % Sheafy font \mathscr{}
\usepackage{tikz}
\usetikzlibrary[snakes]
\usetikzlibrary[shapes]
\usepgflibrary[arrows]
\usetikzlibrary{matrix}

 \usepackage[colorlinks,
             linkcolor=black,
             citecolor=blue,
             bookmarksnumbered=true]{hyperref}

 \theoremstyle{plain}               %%%%% Theorem-like commands
 \newtheorem{theorem}[equation]{Theorem}                            %%
 \newtheorem*{lemma*}{Lemma}                                        %%
 \newtheorem*{theorem*}{Theorem}                                    %%
 \newtheorem{lemma}[equation]{Lemma}                                %%
\newtheorem{corollary}[equation]{Corollary}                        %%
 \newtheorem{proposition}[equation]{Proposition}                    %%

 \theoremstyle{definition} %%%% Definition-like Commands  
\newtheorem{definition}[equation]{Definition}
\newtheorem{example}[equation]{Example}

\theoremstyle{remark} %%%%%% Remark-Like Commands
\newtheorem{remark}[equation]{Remark}

\DeclareMathOperator{\lefttriplearrows} {{\; \tikz{ \foreach \y in {0, 0.1, 0.2} { \draw [stealth-] (0, \y) -- +(0.5, 0);}} \; }}

 \newcommand{\bbar}[1]{\setbox0=\hbox{$#1$}\dimen0=.2\ht0 \kern\dimen0 \overline{\kern-\dimen0 #1}}

 \DeclareMathOperator{\End}{\ensuremath{\mathcal{E}\kern-.125em\mathpzc{nd}}}
 \DeclareMathOperator{\fun}{Fun}

 \let\hom\relax % kills the old hom, which is lowercase
 \DeclareMathOperator{\hom}{Hom}
 \DeclareMathOperator{\Hom}{\mathcal{H}\kern-.125em\mathpzc{om}}
  \DeclareMathOperator{\maps}{Maps}

 \DeclareMathOperator{\Spin}{Spin}
 \DeclareMathOperator{\String}{String}

 \DeclareMathOperator{\ab}	{{\sf Ab}}

 \DeclareMathOperator{\bibun}	{{\sf Bibun}}
 \DeclareMathOperator{\cat}	{\sf Cat}

  \DeclareMathOperator{\grp}	{{\sf Grp}} 
  \DeclareMathOperator{\gpd}	{{\sf Gpd}}

  \DeclareMathOperator{\Lie}	{{\sf Lie}}  
  \DeclareMathOperator{\man}	{\sf Man}

      \DeclareMathOperator{\stack} {{\sf Stack}}

 \DeclareMathOperator{\Top}	{\sf Top}

 \newcommand{\cE}{\mathcal{E}}

 \newcommand{\cO}{\mathcal{O}}

\newcommand{\cS}{\mathcal{S}}

 \newcommand{\A}{\mathbb{A}}

 \newcommand{\F}{\mathbb{F}}
\newcommand{\G}{\mathbb{G}}
 \newcommand{\N}{\mathbb{N}}

 \newcommand{\R}{\mathbb{R}}
\renewcommand{\S}{\mathbb{S}}

  \newcommand{\Z}{\mathbb{Z}}

 \newcommand{\sC}{\mathsf{C}}
 \newcommand{\sD}{\mathsf{D}}

 \newcommand{\sK}{\mathsf{K}}

\newcommand{\sS}{\mathsf{S}}

\definecolor{MyBlue}{rgb}{1.6,0.0,0.0}
 \write0{\today}
\begin{document}
\title{Central Extension of Smooth 2-Groups and a Finite-Dimensional String 2-Group}
%\title{A Finite-Dimensional String 2-Group}
\author{Christopher J. Schommer-Pries}
\email{schommerpries.chris.math@gmail.com}
%\address{Harvard University, Department of Mathematics, 1 Oxford St., Cambridge, MA 02138}
 \date{}

\begin{abstract}
%% Version with citations in the abstract. 
	We provide a model of the String group as a central extension of finite-dimensional 2-groups in the bicategory of Lie groupoids, left-principal bibundles, and bibundle maps. This bicategory is a geometric incarnation of the bicategory of smooth stacks and generalizes the more na\"ive 2-category of  Lie groupoids, smooth functors, and smooth natural transformations.  In particular this notion of smooth 2-group subsumes the notion of Lie 2-group introduced by Baez-Lauda in \cite{BL04}. More precisely we classify a large family of these central extensions in terms of the topological group cohomology introduced by G. Segal in \cite{Segal70}, and our String 2-group is a special case of such extensions. There is a nerve construction which can be applied to these 2-groups to obtain a simplicial manifold, allowing comparison with the model of A. Henriques \cite{Henriques06}. The geometric realization is an $A_\infty$-space, and in the case of our model, has the correct homotopy type of String(n). Unlike all previous models \cite{Stolz96, ST04, Jurco05, Henriques06, BCSS07} our construction takes place entirely within the framework of finite dimensional manifolds and Lie groupoids. Moreover within this context our model is characterized by a strong uniqueness result. It is a canonical central  extension of Spin(n). 
	%% Version without citations in the abstract. 
%	We provide a model of the String group as a central extension of finite-dimensional 2-groups in the bicategory of Lie groupoids, left-principal bibundles, and bibundle maps. This bicategory is a geometric incarnation of the bicategory of smooth stacks and generalizes the more na\"ive 2-category of  Lie groupoids, smooth functors, and smooth natural transformations.  In particular this notion of smooth 2-group subsumes the notion of Lie 2-group introduced by Baez-Lauda. More precisely we classify a large family of these central extensions in terms of the topological group cohomology introduced by G. Segal, and our String 2-group is a special case of such extensions. There is a nerve construction which can be applied to these 2-groups to obtain a simplicial manifold, allowing comparison with the model of A. Henriques. The geometric realization is an $A_\infty$-space, and in the case of our model, has the correct homotopy type of String(n). Unlike all previous models our construction takes place entirely within the framework of finite dimensional manifolds and Lie groupoids. Moreover within this context our model is characterized by a strong uniqueness result. It is a canonical central extension of Spin(n). 
\end{abstract}

 \maketitle
  \tableofcontents

\section{Introduction}

 The String group is a group (or $A_\infty$-space) which is a 3-connected cover of $\Spin(n)$. It has connections to string theory, the generalized cohomology theory {\em topological modular forms} ($tmf$), and to the geometry and topology of loop space. Many of these relationships can be explored homotopy theoretically, but a geometric model of the String group would help provide a better understanding of these subjects and their interconnections. Over the past decade there have been several attempts to provide geometric models of the String group  \cite{Stolz96, ST04, Jurco05, Henriques06, BCSS07}. The most recent of these use the language of higher categories, and consequently string differential  geometry also provides a test case for the emerging field of higher categorical differential geometry \cite{Waldorf09, SSS09, SSS08a}. 
% Add more citations. 

Nevertheless, progress towards the hard differential geometry questions, such as a geometric understanding of the connection to elliptic cohomology or 
 the H\"ohn-Stolz Conjecture \cite{Stolz96}, remains slow. Perhaps one reason is that all previous models of the string group, including the higher categorical ones, are fundamentally infinite-dimensional. In a certain sense, which will be made more precise below, it is impossible to find a finite-dimensional model of $\String(n)$ as a {\em group}. However, there remains the possibility that $\String(n)$ can be modeled as a finite dimensional, but higher categorical object, namely as a finite-dimensional 2-group. This idea is not new, and models for the string group as a Lie 2-group have been given in \cite{Henriques06, BCSS07}. However, these models are also infinite-dimensional.

In this paper we consider 2-groups in the bicategory of finite dimensional Lie groupoids, left principal bibundles, and bibundle maps. This bicategory, which is equivalent to the bicategory of smooth stacks, is an enhancement of the usual bicategory of Lie groupoids, smooth functors and smooth natural transformations. We call such 2-groups {\em smooth 2-groups}. We classify a large family of central extensions of smooth 2-groups in terms of easily computed cohomological data. Our model of the string group comes from such a finite-dimensional central extension. We begin this paper with a more detailed look at the string group and the ideas needed for constructing our model. The main ingredients are, of course, the above mentioned bicategory and also a certain notion of topological group cohomology introduced by Graeme Segal in the late 60's. 
 
\subsection*{What is the String Group?}

The String group is  best understood in relation to the Whitehead tower of the orthogonal group $O(n)$. The Whitehead tower of a space $X$ consists of a sequence of spaces $X\langle n+1\rangle \to X\langle n\rangle \to \cdots \to X$, which generalize the notion of universal cover. A (homotopy theorist's) universal cover of a connected space $X$ is a space $X\langle 2 \rangle$ with a map to $X$, which induces an isomorphism on all homotopy groups except $\pi_1$, and such that $\pi_1(X\langle 2\rangle) = 0$. For more highly connected spaces, there is an obvious generalization, and the Whitehead tower assembles these together. The
maps $X\langle n+1\rangle \to X\langle n\rangle \to \cdots \to X$, induce isomorphisms $\pi_i(X\langle n\rangle ) \cong \pi_i(X)$ for $i \geq n$, and each space satisfies $\pi_i(X\langle n\rangle) = 0$ for $i <  n$. 

For large $n$, the orthogonal group $O(n)$ has the following homotopy groups: 
\begin{equation*}
\begin{array}{ c | c c c c c c c c }
i 		& 0 & 1 & 2 & 3 & 4 & 5 & 6 & 7   	\\ \hline
\pi_i(O(n)) 	 & \Z/2 & \Z/2 & 0 & \Z & 0 & 0 & 0 & \Z	 \end{array}.
\end{equation*}
The first few spaces in the Whitehead tower of $O(n)$ are the familiar Lie groups $SO(n)$ and $\Spin(n)$. These are close cousins to the String group, $\String(n)$. The maps in the Whitehead tower are realized by Lie group homomorphisms,
\begin{align*}
  SO(n) \hookrightarrow O(n)& \\
 \Z / 2\Z \to   Spin(n) \to SO(n) & .
\end{align*}

This raises the question: can we realize the remaining spaces in the Whitehead tower of $O(n)$ as Lie groups? or as topological groups? The next in the sequence would be $O(n)\langle 4\rangle = \dots =  O(n)\langle 7 \rangle$, a space which now goes under the name $\String(n)$. It is the 3-connected\footnote{For $n \geq 7 $ $\String(n)$ is 6-connected.}
% we have pi_6(O(7)) = pi_6(String(7)) = 0 because the short exact sequence O(7) --> O(8) --> S^7 splits via octonionic multiplication.
 cover of $\Spin(n)$. Since $\pi_1$ and $\pi_3$ of this space are zero, it cannot be realized by a finite dimensional Lie group \footnote{This often cited fact follows from two classical results: A theorem of Malcev which states that any  connected Lie group deformation retracts onto a compact subgroup \cite{Malcev45, Malcev46} (See also \cite[Theorem 6]{Iwasawa49}), and the classification of finite dimensional compact simply connected Lie groups, which may be found in many standard text books on Lie groups. }. Moreover, since this hypothetical group is characterized homotopy-theoretically, it is not surprising that there are many models for this group.  

The easiest candidates arise from the machinery of homotopy theory. If we relax our assumption that $\String(n)$ be a topological {\em group} and allow it to be an $A_\infty$-space\footnote{Since $\String(n)$ is connected, if is an $A_\infty$-space it will automatically have the homotopy type of a loop space, i.e. as an ``$A_\infty$-group''.} then there is an obvious model. First we look at the classifying space $BO(n)$. We can mimic our discussion above and construct the Whitehead tower of $BO(n)$. The homotopy groups of $BO(n)$ are the same as those of $O(n)$, but shifted:
\begin{equation*}
\begin{array}{ c | c c c c c c c c c}
i 		& 0 & 1 & 2 & 3 & 4 & 5 & 6 & 7 & 8  	\\ \hline
\pi_i(BO) 	& 0 & \Z/2 & \Z/2 & 0 & \Z & 0 & 0 & 0 & \Z	 \end{array}
\end{equation*}
It is well known that the pointed loop space of a classifying space satisfies $\Omega(BG) \simeq G$ for topological groups $G$. It then follows that the space $\Omega(BO(n)\langle 8 \rangle)$ is an $A_\infty$-space with the right homotopy type. With more care, the homomorphism to Spin$(n)$ can also be constructed\footnote{For Lie groups $G$, the map $\Omega BG \to G$ may be constructed as the holonomy map of the universal connection on $EG$ over $BG$. Thus the composite $\Omega(B\Spin(n)\langle 8\rangle) \to \Omega B\Spin(n) \to \Spin(n)$ is one way to construct the desired map. }. 

If one insists on getting an actual {\em group}, then more sophisticated but similar homotopy theoretic techniques succeed. One replaces the space $BO(n)\langle 8 \rangle$ with its singular simplicial set, and applies Kan's simplicial loop group (see for example \cite[Ch. 5.5]{GJ99}). This produces a simplicial group, which models $\Omega(BO(n)\langle 8 \rangle)$. Taking the geometric realization gives an honest topological group with the correct homotopy type. Needless to say, this construction is not very geometric. 

In both of these approaches the homomorphism $\String(n) \to \Spin(n)$ realizes the String group as a fiber bundle whose fiber is an Eilenberg-Maclane space $K(\Z,2)$. This is a general feature of all approaches. Suppose that we are given a model of $\String(n)$ as a topological group equipped with a continuous homomorphism to Spin$(n)$, realizing it as the 3-connected cover. Let $K$ be the kernel of this map and suppose that this forms a fiber bundle
\begin{equation*}
K \to \String(n) \to \Spin(n).
\end{equation*}
% What conditions garentee this?
 By the long exact sequence of homotopy groups associated to this bundle, we have $K \simeq K(\Z, 2)$. %This bundle is then classified by the generator of $[\Spin(n), K(\Z, 3) ] = H^3(\Spin(n); \Z) \cong \Z$ (here we are assuming $n\geq 5$). 

The primary method of building models of the String group is consequently finding group extensions where the kernel is topologically an Eilenberg-MacLane $K(\Z, 2)$-space. The first geometric models, constructed by S. Stolz and P. Teichner, were of this kind \cite{Stolz96, ST04}. Any CW-complex with the homotopy type of a $K(\Z, 2)$ must have cells of arbitrarily high dimension, and is thus infinite-dimensional\footnote{In fact an easy Serre spectral sequence argument shows that $\String(n)$ itself has cohomology in arbitrarily high degrees and hence has no finite dimensional CW-model.}. Although the groups $K$ used in these models were not CW-complexes, they too were infinite-dimensional and hence resulted in infinite-dimensional models of the String group. 

While $K(\Z, 2)$ is infinite-dimensional, it still has a  well known finite-dimensional description, but at the cost of working higher categorically (or equivalently through the language of $S^1$-gerbes \cite{Giraud71, Murray96, BX06}). This suggests that there might be a finite-dimensional model of $\String(n)$, but as a higher categorical object. This idea is not new, and goes back to the work of Baez-Lauda \cite{BL04}, Henriques \cite{Henriques06}, and Baez-Crans-Schreiber-Stevenson \cite{BCSS07}. The latter were able to construct a Lie 2-group modeling String$(n)$ in a precise sense, but their model is also infinite-dimensional. 

Baez-Lauda \cite{BL04} considered (weak) group objects in the bicategory $\Lie\gpd$ of Lie groupoids\footnote{A {\em Lie groupoid} is the common name for a groupoid object internal to the category of smooth manifolds, in which the source and target maps are surjective submersions.}, smooth functors  and smooth natural transformations. These objects are now commonly called Lie 2-groups, and the finite dimensional incarnation of $K(\Z,2)$ in this context is the Lie 2-group we call $[pt/S^1]$. The Lie group $\Spin(n)$ also provides a basic example of a Lie 2-group. We will elaborate on this in due course. 

J. Baez and A. Lauda \cite{BL04} considered certain ``extensions" $[pt/S^1] \to E \to Spin(n)$, and under certain restrictive assumptions (which can be removed), they proved that such central extensions are in bijection with smooth group cohomology $H^3_{grp}(G; A)$. Herein lies the problem. Since the work of Hu, van Est, and Mostow,
\cite{Hu52-1, Hu52-2, vanEst53, vanEst55, Mostow62}
 we have that for all compact 1-connected simple Lie groups $G$,
 $H^3_{grp}( G; S^1) = 0$. Thus in $\Lie\gpd$, the only such central extension is the trivial one. This is why Baez-Crans-Schreiber-Stevenson were led to infinite-dimensional groups. Essentially, they replace $G = \Spin(n)$ with an infinite-dimensional 2-group for which the above central extension exists. This Lie 2-group is not equivalent to $\Spin(n)$, but nevertheless its geometric realization is homotopy equivalent to $\Spin(n)$, and the resulting central extension does model $\String(n)$.  
The model of A. Henriques \cite{Henriques06} uses different techniques but produces essentially the same object as BCSS, but cast in the language of simplicial spaces.  

In this paper we work entirely within the context of finite dimensional manifolds and Lie groupoids,  never passing into the infinite dimensional setting. As a result our model is fundamentally finite dimensional. The cost is that we must consider groups not in $\Lie\gpd$, but in the bicategory $\bibun$ of Lie groupoids, left-principal bibundles, and bibundle maps. This bicategory is a natural generalization of $\Lie\gpd$, in which the 1-morphisms have a simple geometric description.  Hence, this notion of 2-group, which we call {\em smooth 2-group}, subsumes the notion previously introduced by Baez-Lauda. The bicategory $\bibun$ has other familiar guises. It is equivalent to the bicategory of smooth stacks and also to the (derived) localization of $\Lie\gpd$ with respect to the {\em local equivalences} \cite{Pronk96, Lerman08}. This later has a description in terms of `smooth anafunctors' \cite{Bartels04}. 
 
\subsection*{The Structure and Results of this Paper}

 The bicategory $\bibun$, sadly, does not appear to be widely known, and so we provide a brief review of some key results about this bicategory that we will use.  We then review the notion of weak group object (and also weak abelian group object) in a general bicategory. These are commonly called {\em 2-groups}. More importantly we make precise the notion of extension and central extension of 2-groups, particularly in the context of the bicategory $\bibun$. This generalizes those central extensions of topological groups,
\begin{equation*}
	1 \to A \to E \to G \to 1
\end{equation*}
in which the $A$-action on $E$ realizes it as an $A$-principal bundle over $G$. We also show that the geometric realization of a group object in $\bibun$ is naturally a (group-like) $A_\infty$-space. 

To make the notion of central extension precise in the context of $\bibun$, we must consider certain pull-backs. However, just like the category of smooth manifolds, $\bibun$ does not admit all pull-backs. Nevertheless if two maps of smooth manifolds are transverse, then the fiber product exists. We extend this notion to $\bibun$, introducing {\em transversality for bibundles} in a way which generalizes the usual notion of transversality for smooth maps. We prove that for transverse bibundles the fiber product indeed exists. To our knowledge this is the first time such a result has appeared in the literature. We also introduce a notion of {\em surjective submersion} for bibundles, generalizing the usual notion. This permits us to make precise the central extensions of 2-groups we wish to consider. 

Given an abelian 2-group $\A$ and a 2-group $\G$ (both in $\bibun$) there is a bicategory of central extensions of $\G$ by $\A$, $Ext(\G; \A)$. This bicategory is contravariantly functorial in $\G$ and covariantly functorial in $\A$, and so the Baer sum equips $Ext(\G; \A)$ with the structure of a symmetric monoidal bicategory (See \cite{GPS95, KV94, KV94-2, BN96, DS97} and especially \cite[Chap. 3]{SchommerPries09}). In this paper we prove the following theorem:

\begin{theorem} \label{Thm:MainThm}
	Let $G$ be a Lie group and $A$ an abelian Lie group, viewed as a trivial $G$-module. Then we have an  (unnatural) equivalence of symmetric monoidal bicategories:
	\begin{equation*}
		Ext(G; [pt/A]) \simeq  H^3_\text{SM}( G; A) \times H^2_\text{SM}( G; A)[1] \times H^1_\text{SM}( G; A)[2].
	\end{equation*}
	where $ H^i_\text{SM}( G; A)$ denotes the smooth version Segal-Mitchison topological group cohomology \cite{Segal70}. Moreover, isomorphism classes of central extensions, 
		\begin{center}
	\begin{tikzpicture}[thick]
		\node (LL) at (0,0) 	{$1$ };

		\node (L) at (2,0) 	{$\downdownarrows$};

		\node (LA) at (2,.5) 	{$A$};
		\node (LB) at (2,-.5) 	{$pt$};
		\node (M) at (4, 0) {$\downdownarrows$};
		\node (MA) at (4,.5) 	{$\Gamma_1$};
		\node (MB) at (4,-.5) 	{$\Gamma_0$};

		\node (R) at (6,0)	{$\downdownarrows$};
		\node (RA) at (6,.5)	{$G$};
		\node (RB) at (6, -.5)	{$G$};

		\node (RR) at (8,0)	{$1$};

		\draw [->] (LL) --  node [left] {$$} (L);
		\draw [->] (L) -- node [above] {$$} (M);
		\draw [->] (M) -- node [right] {$$} (R);
		\draw [->] (R) -- node [below] {$$} (RR);
	\end{tikzpicture}
	\end{center}
	are in natural bijection with $H^3_\text{SM}(G;A)$.
\end{theorem}

In the above theorem an abelian group $M$ is regarded as a symmetric monoidal bicategory in three ways. It can be viewed as a symmetric monoidal bicategory $M$ with only identity 1-morphisms and 2-morphisms. It can be viewed as a symmetric monoidal bicategory $M[1]$ with one object, $M$ many 1-morphisms, and only identity 2-morphisms. Finally,  it may be viewed as $M[2]$, a symmetric monoidal bicategory with one object, one 1-morphism, and $M$ many 2-morphisms. Specializing to the case relevant to the String group we obtain the following Theorem: 

\begin{theorem} \label{String}
	If  $n \geq 5$, $A = S^1$, and $G=\Spin(n)$, we have 
	\begin{equation*}
		H^i_\text{SM}( \Spin(n); S^1) \cong H^{i+1}( B\Spin(n); \Z) \cong \begin{cases}
			\Z & i = 3 \\
			0 & i = 1,2 \\
		\end{cases}.
	\end{equation*}
		Thus for each class $[\lambda] \in H^3_\text{SM}( \Spin(n); S^1) \cong \Z$ the bicategory of central extensions with that class is contractible\footnote{i.e. equivalent to the terminal bicategory.}, hence such extension are coherently unique. Moreover, the central extension corresponding to a generator of $H^3_\text{SM}( \Spin(n); S^1)$ gives a finite dimensional model for String$(n)$.
\end{theorem}

The uniqueness in the above theorem is the strongest possible given the category number of the problem. It has the following interpretation. Given a class $[\lambda] \in H^3_\text{SM}( \Spin(n); S^1)$, there exists a central extension realizing that class. Any two such extensions are isomorphic, and moreover any two 1-morphisms realizing such an isomorphism are isomorphic by a unique 2-isomorphism.

%\subsection{Structure of this Paper}

%This paper is divided into four sections. In the remainder of this section we survey some background material on the string group and motivate the problem of finding manageable geometric models of String$(n)$. We highlight the key geometric questions related to string geometry. Following that we briefly review the previous models of String$(n)$ and summarize the main theorems of this paper. We also explain part of the role of Segal-Mitchison topological group cohomology.

%In the final section we review the relevant features Segal's topological group cohomology \cite{Segal70}, and prove our main theorem.

\subsection*{Importance of the String Group} 

The importance of the String group was first noticed in physics. It is well known that in order to define the 1-dimensional super-symmetric sigma model with target space a manifold $X$, one needs $X$ to be a spin manifold.  A similar problem for the 2-dimensional super-symmetric sigma model was studied by Killingback \cite{Killingback87} and later by Witten \cite{Witten88}. They realized that the 2-dimensional super-symmetric sigma model in a space $X$ requires a ``spin structure on the free loop space $LX$''. Witten's investigations eventually lead him to what is now called the Witten genus, which associates to an oriented manifold a formal power series whose coefficients are given by certain combinations of characteristic numbers. For string manifolds, this is the $q$-expansion of an integral modular form. 

One way to understand spin structures on a manifold $X$ is homotopy-theoretically.
The stable normal bundle induces a classifying map $X \to BO$, and a homotopy-theoretic spin structure is a lift of this map to $B\Spin$.  Classical obstruction theory arguments show such a lift exists only if the first and second second Stiefel-Whitney classes vanish. If both $w_1$ and $w_2$ vanish, then there is a new characteristic class  $\frac{p_1}{2}$, such that $2 \cdot (\frac {p_1}{2}) = p_1$ is the first Pontryagin class. A further lift to $BO\langle 8\rangle$ exists if and only if $\frac {p_1}{2}$ vanishes. Such a lift is the homotopy theoretic version of a string structure. A `spin structure on loop space' exists if the transgression of $\frac 1 2 p_1$ vanishes, and it satisfies a further locality property if $\frac 1 2 p_1$ itself vanishes \cite{ST}.

Standard techniques allow one to construct for each of the spaces $BO\langle n\rangle$ a corresponding bordism theory of $BO\langle n\rangle$-manifolds. These bordism theories gives rise to generalized cohomology theories, or more precisely $E_\infty$-ring spectra, $MO\langle n\rangle$. The Witten genus is an $BO\langle 8\rangle$-bordism invariant, and thus gives rise to a map, $MO\langle 8\rangle(pt) \to MF$, where $MF$ is the ring of integral modular forms. 

The Witten genus has a refinement as a map of cohomology theories \cite{AHS01, AHR, Henriques08}:
\begin{equation*}
	MO\langle 8\rangle \to tmf.
\end{equation*}
Here $tmf$ is the theory constructed by Hopkins and Miller of {\em topological modular forms} \cite{Hopkins95, Hopkins02}. There is a map of graded rings $tmf^*(pt) \to MF$, which factors the Witten-Genus. This map is rationally an isomorphism, but is not surjective or injective, integrally.
The ring $tmf^*(pt)$ contains a significant amount of torsion. The refinement of the Witten genus is similar to the refinement of the $\hat A$-genus, which can also be viewed as a map of cohomology theories, %\notetoself{Check this! is it $K$ or $KO$???}
\begin{equation*}
	MSpin \to KO.
\end{equation*}
Here $KO$ is real K-theory.
These refinements have the following consequences. If $E \to X$ is a family of string manifolds parametrized by $X$, then there is a family Witten genus which lives in $tmf^*(X)$. Similarly a family of spin manifolds has a family version of the $\hat A$-genus, which lives in $KO^*(X)$. While there are homotopy theoretic descriptions of both of these based on the Thom isomorphisms for string and spin vector bundles, respectively, the $\hat A$-genus also has an analytic/geometric interpretation derived from the concrete geometric model of Spin$(m)$.

Given a manifold with a geometric spin structure, we can form the associated bundle of spinors and the corresponding Clifford-linear Dirac operator. If we have a family of spin manifolds parametrized by a space $X$, we get a corresponding family of Fredholm operators. This represents the class in $KO^*(X)$. The Witten genus has no corresponding geometric definition\footnote{Note however that Witten's original argument is based on the construction of a Spin-structure on the free loop space. His heuristic derivation was to take the $S^1$-equivariant index of the \textquotedblleft Dirac operator on loop space
\textquotedblright. %Unfortunately this has not been rigorously constructed mathematically.
}, and nor does the cohomology theory $tmf$. A suitable geometric model for the String group will lead to a better {\em geometric} understanding of string structures and might provide insight into these problems. 

Finally, we should mention an as yet unresolved conjecture due independently to G. H\"ohn and  S. Stolz relating string structures and Riemannian geometry.  S. Stolz conjectures in \cite{Stolz96} that a $4k$-dimensional string manifold which admits a positive Ricci curvature metric necessarily has vanishing Witten genus. Some progress has been made towards this (and related conjectures \cite{StolzPriv}) in the Dissertation work of C. Redden \cite{Redden06}, but a clear answer remains out of reach. A better geometric understanding of string structures would doubtless shed light on this problem as well.

\subsection*{Acknowledgments}

I am grateful for the encouragement, inspiration and insightful comments continually provided by Peter Teichner. I would like to thank Alan Weinstein for introducing me to Lie groupoids and to the bicategory $\bibun$. I would also like to thank Christian Blohmann, Stephan Stolz, Urs Schreiber, Christopher Douglas, Calvin Moore, Christoph Wockel, Andr\'e Henriques, and Konrad Waldorf for their many useful conversations and remarks while this work was in progress.   

\section{Lie Groupoids and Smooth Stacks}

\subsection{Lie Groupoids}

\begin{definition}
A {\em Lie groupoid} is a groupoid object, $\Gamma = (\Gamma_1 \rightrightarrows \Gamma_0)$, in the category of (finite dimensional) smooth manifolds in which the source and target maps,
 \begin{equation*}
	s,t: \Gamma_1 \to \Gamma_0
\end{equation*}
are surjective submersions. (In particular the iterated fiber products $\Gamma_1 \times_{\Gamma_0} \Gamma_1$ and $\Gamma_1 \times_{\Gamma_0} \Gamma_1 \times_{\Gamma_0} \Gamma_1$ exist as smooth manifolds). 
{\em Functors} and {\em natural transformations} are defined as functors and natural transformations internal to the category of smooth manifolds. 
\end{definition}

Together Lie groupoids, functors and natural transformations form a 2-category, $\Lie\gpd$. There are many examples of Lie groupoids. The most common are special cases of the following two kinds:

\begin{example} [$G$-Spaces] \label{ExampleManAsLieGpd} \label{ExampleLieGroups}
Let $G$ be a Lie group acting smoothly (say, on the right) on a manifold $X$. Then we can form the {\em action groupoid} $\Gamma = [X/G]$. The objects are $\Gamma_0 = X$ and morphisms are $\Gamma_1 = X \times G$. The target map is projection, and the source is the action map. Composition,
\begin{equation*}
	m: (X \times G) \times_X (X \times G) = X \times G \times G \to X \times G
\end{equation*}
is given by multiplication in $G$. The identity map is $x \mapsto (x, e)$ and the inverse map is $(x, g) \mapsto (xg, g^{-1})$. %We will denote this Lie groupoid by $[X/G]$.

When the group is trivial, this allows any smooth manifold $X$ to be viewed as a Lie groupoid with $X_0 = X_1 = X$ and all maps identity maps. When the manifold $X = pt$ is trivial, this allows any Lie group $G$ to be viewed as a Lie groupoid with $G_0 = pt$ and $G_1 = G$. In this case the composition is the group multiplication, with the usual identities and  inverses.  
\end{example}

\begin{example} [\v{C}ech Groupoids] \label{ExampleCechGroupoids}
If $Y \to X$ is a submersion, 	then we can form the {\em \v{C}ech groupoid} $X_Y$. We have objects $(X_Y)_0 = Y$ and morphisms $(X_Y)_1 = Y^{[2]} : = Y \times_X Y$. The source and target maps are the canonical projections, the identities come from the diagonal. Inversion comes from the flip map and composition comes from forgetting the middle factor. We will only be interested in the case where $Y$ is a {\em surjective} submersion, and in particular when $Y = U \to X$ is an ordinary cover. 
The special case $Y = M \to pt = X$ yields a Lie groupoid known as the {\em pair groupoid} $EX$.

 %A more succinct way to express this is that $X_U$ is the pair groupoid for the object $U$, but considered in the category $\man_X$ of objects over $X$. 

%More generally, if $Y \to X$ is any surjective submersion, then we form the Lie groupoid $X_Y = (Y^{[2]} \rightrightarrows Y)$, just as above. We also refer to $X_Y$ as a \v{C}ech groupoid. 
\end{example}

%\begin{example}[Pair Groupoids]
%The  has $EX_0 = X$ and $EX_1 = X \times X$. The source and target maps are the two projections. The composition is given by $(x,y) \times (y, z) \mapsto (x, z)$, the inverse by $(x,y) \mapsto (y,x)$, and the identity $x \mapsto (x,x)$. 
%\end{example}

There are also many examples of functors and natural transformations:

\begin{example}[Smooth Maps]
Let $X, Y$ be manifolds viewed as Lie groupoids. Smooth functors from $X$ to $Y$ are the same as smooth maps $X \to Y$. Given two such functors $f,g$ there are no natural transformations unless we have equality $f = g$. In that case there is just the identity natural transformation. This gives a (fully-faithful) inclusion functor $\man \to \Lie\gpd$.
\end{example}
\begin{example}[Lie Homomorphisms]
	Let  $G$ and $H$ be Lie groups viewed as Lie groupoids. The functors from $G$ to $H$ are precisely the Lie group homomorphisms. A natural transformation between $f$ and $g$ is the same as an element of $H$, conjugating $f$ into $g$. 
\end{example}

\begin{example} \label{PairGroupoidEquiv}
	Let $X$ be a manifold and let $EX$ be the corresponding pair groupoid. There is a unique functor to the one-point groupoid $pt$. A choice of point $x_0 \in X$, determines a functor $x_0: pt \to EX$. The composition $pt \to EX \to pt$ is the identity functor. The other composition $EX \to pt \to EX$ sends every object to $x_0$ and every morphism to $\iota(x_0)$. This is naturally isomorphic to the identity functor via the natural transformation
	\begin{equation*}
		\eta: x \mapsto (x, x_0).
\end{equation*}
  	Thus $EX$ and $pt$ are equivalent as Lie groupoids. 
	\end{example}

	Let $U \to X$ be a cover, and let $X_U$ be the resulting \v{C}ech groupoid. Recall that the \v{C}ech groupoid can be thought of as the pair groupoid, but in the category of spaces over $X$. Again there is a canonical functor $X_U \to X$, and $X$ serves the same role as the point, but in the category of spaces over $X$. Thinking in this line, one is tempted to guess that $X_U \to X$ is an equivalence. However, usually this is false. The canonical functor $X_U \to X$ is an equivalence if and only if the cover admits a global section $s: X \to U$. More precisely, we have the following lemma, whose proof is a straightforward calculation left to the reader.
	
\begin{lemma} \label{mapsofcoversgiveequivofcech}
Let $Y \to X$ and $Z \to X$ be spaces over $X$. Then the corresponding \v{C}ech groupoids are equivalent if and only if there exist maps over $X$,  $f: Y \to Z$ and $g: Z \to Y$. In that case the equivalence is given by the canonically induced functors and the natural transformations are given by,
\begin{align*}
Y & \to Y^{[2]} \\
y & \mapsto (y, gf(y)) \\
Z & \to Z^{[2]} \\
z & \mapsto (z, fg(z))
\end{align*}
In particular, $X_U$ is equivalent to $X$ if and only if the cover $U \to X$ admits a global section.  
\end{lemma}

%\begin{proof}
%Trivial.
%\end{proof}	

The last example highlights one of the well-known deficiencies of the 2-category of Lie groupoids. The functor $X_U \to X$ is both fully faithful and essentially surjective (in fact, actually surjective), but it fails to be an equivalence.

\subsection{Bibundles and Smooth Stacks}

In the following let $\man_X$ denote the category of {\em manifolds over $X$}, that is  the category whose objects are manifolds $Y$ equipped with a smooth map $Y \to X$, and whose morphisms are smooth maps $Y \to Y'$ making the following triangle commute.
\begin{center}
\begin{tikzpicture}
	\node (LT) at (0, 1) {$Y$};
	\node (LB) at (1, 0) 	 {$X$};
	\node (RT) at (2, 1) {$Y'$};
	%\node (RB) at (2, 0)    {$$};
	\draw [->] (LT) -- node [left] {$$} (LB);
	\draw [->] (LT) -- node [above] {$$} (RT);
	\draw [->] (RT) -- node [right] {$$} (LB);
	%\draw [->] (LB) -- node [below] {$$} (RB);
\end{tikzpicture}
\end{center}

\begin{definition}
 Let $G= (G_1 \rightrightarrows G_0)$ and $H = (H_1 \rightrightarrows H_0)$ be Lie groupoids. A {\em (left principal) bibundle} from $H$ to $G$ is a smooth manifold $P$ together with
 \begin{enumerate}
\item A map $\tau: P \to G_0$, and a surjective submersion $\sigma: P \to H_0$, %thus 
%\begin{align*}
%		(\tau, \sigma): & \; P \to G_0 \times H_0 \\
%		(t \circ p_1, \sigma \circ p_2): & \; G_1 \times_{G_0}^{s, \tau} P \to G_0 \times H_0 \\
%		(  \tau \circ p_1, s \circ p_2):& \; P \times_{H_0}^{\sigma, t} H_1 \to G_0 \times H_0 
%\end{align*}
%are objects of $\man_{G_0 \times H_0}$.
\item Action maps in $\man_{G_0 \times H_0}$
\begin{align*}
	& G_1 \times_{G_0}^{s, \tau} P \to P \\
	&  P \times_{H_0}^{\sigma, t} H_1 \to P
\end{align*}
which we denote on elements as $(g, p) \mapsto g \cdot p$ and $(p, h) \mapsto p \cdot h$,
\end{enumerate}
such that
 \begin{enumerate}
 \item $g_1 \cdot ( g_2 \cdot p) = (g_1 g_2) \cdot p$ for all $(g_1, g_2, p) \in G_1 \times_{G_0}^{s, t}  G_1 \times_{G_0}^{s, \tau} P$
\item $(p \cdot h_1) \cdot h_2 = p \cdot (h_1h_2)$ for all $(p, h_1, h_2) \in P \times_{H_0}^{\sigma, t} H_1 \times_{H_0}^{s, t} H_1$
\item $p \cdot \iota_H(\sigma(p)) = p $ and $\iota_G(\tau(p)) \cdot p = p$ for all $p \in P$. 
\item $g \cdot (p \cdot h) = (g \cdot p) \cdot h$ for all $(g, p, h) \in  G_1 \times_{G_0}^{s, \tau} P  \times_{H_0}^{\sigma, t} H_1 $. 
\item The map 
\begin{align*}
	 G_1 \times_{G_0}^{s, \tau} P & \to P \times_{H_0}^{\sigma, \sigma} P  \\
	 (g, p) & \mapsto (g \cdot p, p)
\end{align*}
is an isomorphism. (The $G$-action is simply transitive)
\end{enumerate}
\end{definition}

Bibundles combine several widely used notions into a single useful concept, as these examples illustrate. 

\begin{example}[Smooth Maps]
Let $X$ and $Y$ be smooth manifolds, viewed as Lie groupoids. Let $P$ be a (left principal) bibundle from $X$ to $Y$. Then $\sigma: P \to X$ is an isomorphism. Thus $P$ is ``the same'' as a smooth map $\tau: X \to Y$. 
\end{example}

\begin{example}[Lie Homomorphisms] Let $G$ and $H$ be Lie groups, thought of as Lie groupoids as in Example \ref{ExampleLieGroups}. Let $P$ be a bibundle from $H$ to $G$. Then $P$ is (non-canonically) isomorphic to $G$ with its left $G$-action. After identifying $P$ with $G$, the right action of $H$ on $P$ is equivalent to a Lie group homomorphism $H \to G$. Thus Lie homomorphisms are ``the same'' as bibundles.  (More precisely, as we will see shortly,  conjugacy classes of Lie homomorphisms correspond to isomorphism classes of bibundles).  
\end{example}

This next example shows where bibundles derive their name.

\begin{example}[Principal Bundles]
	A (left principal) bibundle from a manifold $X$ to a Lie group $G $, viewed as Lie groupoids as in Example \ref{ExampleManAsLieGpd}, is the same as a (left) principal $G$-bundle over $X$.
\end{example} 

\begin{example}\label{ExampleACJBibuns}
	Generalizing this last example, let $Y$ be a manifold with a (left) action of a Lie group $G$. Denote the associated action Lie groupoid by $[Y/G]$. Let $X$ be a manifold. A (left principal) bibundle from $X$ to $[Y/G]$ consists of a (left) principal $G$-bundle $P$ over $X$ together with a $G$-equivariant map to $Y$. In particular, we may take the action of $G$ on $Y = Aut(G)$ to be by left multiplication by the conjugation automorphism, i.e. for $h \in Aut(G)$,  $g \cdot h = c_g \circ h$, where $c_g(g') = g g' g^{-1}$ is the conjugation automorphism. A (left principal) bibundle from $X$ to $[Aut(G)/G]$ is a ``$G$-bibundle'' in the sense of \cite{ACJ05}. 
\end{example}

\begin{example}[Identities] \label{ExampleIdentBibun}
	Let $G$ be a Lie groupoid. There is a $G$-$G$ bibundle given by $P = G_1$ with $\tau = t, \sigma = s$ and the obvious action maps. This is called the {\em identity bibundle} for reasons which will become obvious later. 
\end{example}

\begin{example} If $f: X \to G_0$ is a map, then we can form the pull back bibundle. $f^*G_1 = X \times_{G_0}^{f, s} G_1 \to X$. The induced action of $G_1$ on $f^*G_1$ makes this a bibundle from the trivial groupoid $X$ (with only identity morphisms) to the groupoid $G$.
\end{example}

\begin{example}
	Let $f: H \to G$ be a functor of Lie groupoids. Then we form the bibundle $\langle f\rangle$ as follows. As a space we have $\langle f\rangle = f_0^* G_1$, which we've already seen is a bibundle from $H_0$ to $G$. We need only supply the action of $H_1$. This is given by applying $f_1: H_1 \to G_1$ and using right action of $G_1$ on $f_0^* G_1$. Thus any functor gives rise to a  bibundle. The association $f \mapsto \langle f \rangle$ is known as {\em bundlization}.
\end{example}

\begin{example} \label{ExamplePullbackGroupoid}
If  $f: U \to G_0$ is a submersion, then we may form the pull-back groupoid $f^*G$. The objects consist of $U$, the morphisms consist of $(U \times U)  \times_{G_0 \times G_0} G_1$ with source and target the natural projections. Composition is defined in the obvious way, as a confluence of the composition the pair groupoid of $U$ and of the composition of $G$.
%by,
%\begin{align*}
%	[(U \times U) {}^{f \times t}\times_{G_0 \times G_0}^{(s,t)} G_1]  \times_U [(U \times U) {}^{f \times t} \times_{G_0 \times G_0}^{(s,t)} G_1]  & \\
%	= (U \times U \times U) {}^{f \times f \times f}\times_{G_0 \times G_0 \times G_0}^{[(sp_1, sp_2, tp_2)} [G_1 \times_{G_0}^{t,s} G_1] &\to (U \times U) {}^{f \times f}\times_{G_0 \times G_0}^{s,t} G_1 \\
%	(u_0, u_1, u_2, f, h) & \mapsto (u_0, u_2, f \circ h)
%\end{align*}
There is a functor from $f^*G$ to $G$ which on object is the original map $U \to G_0$ and on morphisms is the projection $f^*G \to G_1$. In particular, there is a canonical bibundle from $f^*G$  to $G$ given by the bundlization of this functor.  
\end{example}

\begin{remark}
Right principal bibundles can be defined in a similar manner. The only difference being that now $\tau$, instead of $\sigma$, is required to be a surjective submersion and the action of $H$ is simply transitive, i.e.
\begin{equation*}
	P \times_{H_0}^{\sigma, t} H_1 \cong P \times_{G_0}^{\tau, \tau} P.
\end{equation*}
In particular any left-principal bibundle $P$ from $H$ to $G$ gives rise to a right-principal bibundle $\overline P$ from $G$ to $H$, given by swapping the maps $\sigma$ and $\tau$, and pre-composing the action maps with the inverse maps.  
\end{remark}

\begin{definition}
	A {\em bibundle map} is a map $P \to P'$ over $H_0 \times G_0$ which commutes with the $G$- and $H$-actions, i.e. the following diagrams commute.
	\begin{center}
\begin{tikzpicture}[thick]
	\node (LT) at (0,1.5) 	{$G_1 \times_{G_0}^{s, \tau} P $ };
	\node (LB) at (0,0) 	{$G_1 \times_{G_0}^{s, \tau} P '$};
	\node (RT) at (2,1.5) 	{$P$};
	\node (RB) at (2,0)	{$P'$};
	\draw [->] (LT) --  node [left] {$$} (LB);
	\draw [->] (LT) -- node [above] {$$} (RT);
	\draw [->] (RT) -- node [right] {$$} (RB);
	\draw [->] (LB) -- node [below] {$$} (RB);
	
	\node (LT) at (4,1.5) 	{$ P \times_{H_0}^{\sigma, t} H_1 $ };
	\node (LB) at (4,0) 	{$P' \times_{H_0}^{\sigma, t} H_1 $};
	\node (RT) at (6,1.5) 	{$P$};
	\node (RB) at (6,0)	{$P'$};
	\draw [->] (LT) --  node [left] {$$} (LB);
	\draw [->] (LT) -- node [above] {$$} (RT);
	\draw [->] (RT) -- node [right] {$$} (RB);
	\draw [->] (LB) -- node [below] {$$} (RB);
\end{tikzpicture}
\end{center}
\end{definition}

Thus for each pair of groupoids we have a category $\bibun(H,G)$ of bibundles from $H$ to $G$. If $f,g: H \to G$ are two smooth functors between Lie groupoids, then the bibundle maps from $\langle f \rangle$ to $\langle g \rangle$ are in natural correspondence with the smooth natural transformations from $f$ to $g$. In this sense the category $\Lie\gpd(H,G)$ is a subcategory of $\bibun(H, G)$.

\begin{example}
	A left principal bibundle from a Lie group $H$ to a Lie group $G$ always arises as $\langle f \rangle$ for some functor $f: H \to G$.
	\end{example}

\begin{example}
	A left principal bibundle whose target is a space is also always of the form $\langle f \rangle$ for some functor $f: X \to Y$. Hence if $X$ and $Y$ are spaces the is the same as a map of spaces. If $X$ is an action groupoid, then this is just a $G$-invariant map. 
\end{example}

%\begin{example}
%	In a similar vein, a bibundle from an action groupoid to a space is just a $G$-invariant map, and a bibundle from a space to an action groupoid is a principal $G$-bundle over $X$ with a $G$-equivariant map to $Y$. 
%\end{example}

\begin{proposition}[\cite{Lerman08}] \label{PropSectionIsBundlization}
	A bibundle $P$ from $H$ to $G$ admits a section of $\sigma: P \to H_0$ if and only if $P \cong \langle f \rangle$ for some smooth functor $f$. 
\end{proposition}

Bibundles can be composed, and this gives us a bicategory $\bibun$. If $P$ is a bibundle from $H$ to $G$ and $Q$ is a bibundle from $K$ to $H$, then we define the bibundle $P \circ Q$ as the coequalizer,
\begin{equation*}
	P \times^{\sigma, t}_{H_0} H_1 \times^{s, \tau}_{H_0} Q \rightrightarrows P \times_{H_0}^{\sigma, \tau} Q \to P \circ Q
\end{equation*}
Since $\sigma$ is a surjective submersion, these pull-backs are manifolds and since our action on $Q$ is simply transitive  this coequalizer exists as a smooth manifold. In fact it is a bibundle from $K$ to $G$. See \cite{Lerman08} for details. Equivalence in this bicategory is sometimes referred to as {\em Morita equivalence}. They are characterized as those bibundles which are simultaneously left principal and right principal. The identity bibundle of Example \ref{ExampleIdentBibun} above serves as the identity 1-morphism. If the submersion in Example \ref{ExamplePullbackGroupoid} is surjective, then the pull-back groupoid is easily seen to be Morita equivalent to the original groupoid via the constructed bibundle. 

\begin{example}
Let $G$ and $H$ be Lie groupoids and $P: H \to G$ a left-principal bibundle. If $P$ is also a right-principal bibundle, then we may form a new left principal bibundle $P^{-1}: G \to H$. $P^{-1}$ is the space $P$ with $\tau$ and $\sigma$ switched, and with a right (resp. left) action of $G$ (resp. $H$) induced by the composition of the inversion map and the original action on $P$. In this case we have that $P \circ P^{-1}$ and $P^{-1} \circ P$ are isomorphic to identity bibundles. Is this case $P$ and $P^{-1}$ are Morita equivalences, and this characterizes Morita equivalences. 
\end{example}

\begin{example} As a special case of the above, suppose that  $G$ is a Lie group with a free and transitive action on the manifold $X$. Suppose further that the quotient space $Y$ is a manifold, with smooth quotient map, $q: X \to Y$. If the quotient map admits local sections (so that $X$ is a fiber bundle over $Y$), then we have a bibundle: 
	\begin{center}
\begin{tikzpicture}
\node (A) at (0,1) {$ Y$};
\node (B) at (0,0) { $Y$};
\draw [->] (A.255) -- (B.105);
\draw [->] (A.285) -- (B.75);

\node (C) at (4,1) {$X \times G$};
\node (D) at (4,0) { $X$};
\draw [->] (C.255) -- (D.105);
\draw [->] (C.285) -- (D.75);

\node (E) at (2,1) {$X$};
\draw [->] (E) -- node [above left] {$q$} (B);
\draw [->>] (E) -- node [above right] {$$}(D);
\end{tikzpicture}
\end{center}
with the obvious induced actions. This bibundle, which is the bundlization $\langle q \rangle$ of the induced quotient functor, is simultaneously a left- and right-principal bibundle and hence $Y$ and $[X/G]$ are equivalent in $\bibun$. Conversely, if $Y$ and $[X/G]$ are equivalent in $\bibun$, then the quotient map $q: X \to Y$ necessarily admits local sections. 
\end{example}

\begin{theorem}[Pronk \cite{Pronk96}]
	There are canonical equivalences of bicategories between
$\bibun$, $\stack_{pre}$, and $\Lie\gpd [ W^{-1}] $, where $\stack_{pre}$ is the 2-category of (presentable) smooth stacks (in the surjective submersion topology) and  $\Lie\gpd [ W^{-1}] $ is the (derived) localization of the 2-category of Lie groupoids, functors, and natural transformations with respect to the {\em essential equivalences}.\end{theorem}

\begin{remark} \label{RmkForgetTopology}
There is a forgetful 2-functor $\Lie \gpd \to \gpd$ which which forgets the topology of the Lie groupoid. This functor sends essential equivalences to equivalences and hence extends in an essentially unique way to a 2-functor $\bibun \to \gpd$. This 2-functor is product preserving. 
\end{remark}

\subsection{Transversality for Stacks}

\begin{definition}
	Let $X, Y, Z$ be Lie groupoids and let $G: X \to Y$ and $F: Z \to Y$ be two left-principal bibundles. $F$ and $G$ are {\em transverse} (written $F \pitchfork G)$ if the maps $F \to Y_0$ and $G \to Y_0$ are transverse.
\end{definition}

This extends the usual notion of transversality for maps of spaces. 
\begin{lemma}
	Let $X, Y, Z$ be Lie groupoids and let $G: X \to Y$ and $F: Z \to Y$ be left-principal bibundles. If $F \pitchfork G$ then each of the four pairs of maps
	\begin{enumerate}
	\item  $t \circ p_1 : Y_1 \times^{s, \tau}_{Y_0} F \to Y_0$ and $G \to Y_0$,
	\item $F \to Y_0$ and $ s \circ p_1: Y_1 \times^{t, \tau}_{Y_0} G \to Y_0$,
	\item $s \circ p_1 : Y_1 \times^{t, \tau}_{Y_0} F \to Y_0$ and $G \to Y_0$,   
	\item $F \to Y_0$ and $ t \circ p_1: Y_1 \times^{s, \tau}_{Y_0} G \to Y_0$,
	\end{enumerate}
	are transverse.
\end{lemma}

\begin{proof}
	By symmetry, it is enough to consider only the first two pairs of maps. Moreover, the transversality of the first pair is easily seen to be equivalent to the second pair, thus it is enough to prove that the first pair of maps are transverse. The map $t \circ p_1 : Y_1 \times^{s, \tau}_{Y_0} F \to Y_0$ factors through the action map 
\begin{equation*}
	Y_1 \times^{s, \tau}_{Y_0} F \to F
\end{equation*}	
	which is surjective and surjective on tangent spaces. Thus the images agree $d(t \circ p_1)(T(Y_1 \times^{s, \tau}_{Y_0} F)) = d\tau (TF)$, and the result follows.  
\end{proof}

Recall that given a left principal bibundle $G$ from the Lie groupoid $X$ to the Lie groupoid $Y$, we may form a right principal bibundle $\overline{G}$ from $Y$ to $X$ by flipping the structure maps $\tau$ and $\sigma$ and by using the inverse maps to switch left and right actions.

\begin{lemma}
If $F \pitchfork G$, then the following coequalizer is a smooth manifold,
\begin{equation*}
	\overline{G} \times^{\tau, t}_{Y_0} Y_1 \times_{Y_0}^{s, \tau} F \rightrightarrows \overline G \times_{Y_0} F \to \overline G \circ F.
\end{equation*}
\end{lemma}

\begin{proof}
	This is a local question. For each point $x \in X_0$ there exists an open neighborhood $U \subset X_0$ and a map $g_0: U \to Y_0$, so that over $U$ we have $G|_U \cong Y_1 \times_{Y_0}^{s, g_0} U$. Similarly, for each point in $Z_0$, there exists a open neighborhood $f_0: V \subset Z_0$, a map $V \to Y_0$ so that $F|_V \cong Y_1  \times_{Y_0}^{s, f_0} U$. The transversality conditions ensure that $f_0$ and $g_0$ are also transverse. Locally the above equalizer is isomorphic to
	\begin{equation*}
		U \times_{Y_0}^{g_0, t} Y_1 \times_{Y_0}^{s, f_0} V
\end{equation*}
which is again a manifold by our transversality assumptions. 
\end{proof}
A similar calculation shows that $X_1 \times_{X_0} ( \overline G \circ F) \times_{Z_0} Z_1$ is a smooth manifold. The primary reason for introducing the notion of transversality between maps of spaces is that it is a condition which ensures that pullbacks exist as smooth manifolds. The notion of transversality introduced here generalizes this property to the bicategory $\bibun$. 

\begin{proposition}\label{PropPullBackOfBibunExist}
Let $X, Z, Y$ be Lie groupoids and let $G: X \to Y$ and $F: Z \to Y$ be two left-principal bibundles. If $F$ and $G$ are transverse then the pullback exists in $\bibun$. 
\end{proposition}  

In the above proposition, {\em pull-back} is meant as a weak categorical limit (a.k.a. {\em bi-limit}) of the obvious diagram. See \cite{Street80, Street87} for details concerning such limits. In this case, such a pull-back consists of a Lie groupoid $W$, equipped with bibundles $P_1: W \to X$ and $P_2: W \to Y$, together with an isomorphism of bibundles $G \circ P_1 \cong F \circ P_2: W \to Y$, which is universal for such Lie groupoids. 

\begin{proof}[Proof of Proposition \ref{PropPullBackOfBibunExist}]
	We explicitly construct a pullback. The underlying Lie groupoid is given as follows:	
\begin{itemize}
\item objects: $\overline G \circ F$
\item morphisms: $ X_1 \times^{s, \sigma}_{X_0} (\overline G \circ F) \times^{\sigma, s}_{Z_0} Z_1$
\end{itemize}
with source map given by $p_2$ onto the middle factor and target given by the action. Composition is given by:
\begin{equation*}
	(\alpha,[g , f] , \beta) \circ ( \alpha', [g', f'], \beta') = ( \alpha \circ \alpha', [g,f], \beta \circ \beta').
\end{equation*}
The identities and inverses are given by the obvious maps. Call this Lie groupoid $\Gamma$. 
This Lie groupoid comes equipped with two smooth functors, which we regard as bibundles. The first is a functor $p_1:\Gamma \to X$ and is given on objects by the natural projection $(\overline G \circ F) \to X_0$. On morphisms it is also the projection
\begin{equation*}
X_1 \times^{s, \sigma}_{X_0} (\overline G \circ F) \times^{\sigma, s}_{Z_0} Z_1 \to X_1.
\end{equation*}
One can check that this indeed defines a functor. The functor $p_2: \Gamma \to Z$ is defined similarly. The bundlization of the first functor is a bibundle whose total space is $X_1 \times_{X_0} (\overline G \circ F)$.

Composing the first map with the bibundle $G$ we have,
\begin{align*}
	G \circ (X_1 \times_{X_0} (\overline G \circ F)) &\cong G \times_{X_0 } (\overline G \circ F)\\
	& \cong (G \times_{X_0} \overline G) \circ F \\
	& \cong (\overline G \times_{Y_0} Y_1) \circ F \\
	&\cong \overline G \times_{Y_0} F
\end{align*}
where the later isomorphism follows from the simple transitivity of the $Y$ action on $G$. 
A similar calculation shows that composing the second map with $F$ gives a canonically isomorphic bibundle. 

To prove that $\Gamma$ is the pullback, we must now check the universal property. In particular given a Lie groupoid $W$ and bibundles $f: W \to Z$ and $g: W \to X$, together with an isomorphism of bibundles $\phi: G \circ g\to F \circ f$, we must construct a bibundle $P:W \to \Gamma$ and isomorphisms $g \cong P_1 \circ P$, $f \cong P_2 \circ P$. The total space of $P$ is given by $P = g \times_{W_0} f$, with its canonical map to $W_0$, and diagonal action. We must construct the projection to $\overline G \circ F$. 

The isomorphism $\phi: G \circ g \to F \circ f$ is essential for this map. $\phi$ induces a map,
\begin{equation*}
	G \times_{X_0} g \times_{W_0} f \to (G \circ g) \times_{W_0} f \to (F \circ f) \times_{W_0} f \cong F \circ (Z_1 \times_{Z_0} f) \cong F \times_{Z_0} f
\end{equation*}
Let $(a,b) \in g \times_{W_0} f$. Consider the image of $a$ in $X_0$ under the projection $g \to X_0$. Choose a lift $\tilde a \in G \to X_0$, which always exists since $G \to X_0$ is a surjective map. The above map says that given $\tilde a, a, b$, we get an element in $\tilde b \in F$.

We define the image of $(a,b) \in g \times_{W_0} f$ in $\overline G \circ F$ to be the equivalence class $[\tilde a, \tilde b]$. The only ambiguity in this construction is the choice of the lift $\tilde a$. Since the action of $Y$ is simply transitive on $G$, the choices of $\tilde a$ differ precisely by the action of $Y$. Since $\phi$ is equivariant with respect to the $Y$-action, it follows that we have a well defined element in $\overline G \circ F$. Moreover since the lift $\tilde a$ is given by a section of $G \to X_0$, which can locally be chosen to be smooth, the resulting projection map is smooth. 

The left action on $g \times_{W_0} f$ is given by the usual action map via,
\begin{equation*}
[X_1 \times_{X_0} (\overline G \circ F) \times_{Z_0} Z_1] \times_{\overline G \circ F} [ g \times_{W_0} f] \cong (X_1 \times_{X_0} g) \times_{W_0} (Z_1 \times_{Z_0} f) \to g \times_{W_0} f.
\end{equation*}
One can check that there are canonical isomorphisms $g \cong P_1 \circ P$ as desired $f \cong P_2 \circ P$, and consequently that $\Gamma$ satisfies the universal property of a pullback.
\end{proof}

\begin{example}
 If $X, Y, Z$ are manifolds, then transversality is transversality in the usual sense and the pullback is the usual pullback.  More generally if $X$ and $Z$ are manifolds and $Y = [W/G]$ is a quotient Lie groupoid, then locally in $X$ a bibundle to $Y$ is given by a $G$-equivariant map $f:X \times G \to W$, or equivalently by a map $X \to W$. (This is only the {\em local} picture. Globally these maps $f$ are glued together by the action of $G$ on $W$. There is usually no global map.)  If $y \in W$ is a point which is in the image of the corresponding (local) maps $f:X \to W$ and $g:Z \to W$, then transversality at $y$ is equivalent to the identity: 
 \begin{equation*}
	df(T_xX) + dg(T_zZ) + T_y\cO_G(y) = T_yW
\end{equation*}
where $\cO_G(y)$ is the $G$-orbit through the point $y$. 
\end{example}

\begin{definition}
	A morphism $F \in \bibun(X,Y)$ is called {\em representable} if for all manifolds $M$ and all maps $G \in \bibun(M, Y)$, the pullback exists and is equivalent to a manifold.
\end{definition}

\begin{example}
	Let $X, Y, Z$ be Lie groups thought of as Lie groupoids with one object. Then $F$ and $G$ are equivalent to group homomorphisms and are always transverse. The pullback is the action groupoid of  $X \times Z$ acting on the space $Y$ by $(x,z) \cdot y = G(x) y F(z)^{-1}$. 
\end{example}

\begin{example}
An important special case of the previous example is when $X = pt$ corresponds to the trivial group, and  the homomorphism $Z \to Y$ corresponds to a closed embedding of Lie groups. In this case the action is free and the quotient is a manifold. Thus the groupoid of $Z$ acting on $Y$ is equivalent to the quotient space. In this case the map $Z \to Y$ is representable. 
\end{example}

\begin{example}
 Let $X=(X_1 \rightrightarrows X_0)$ be a groupoid. We may view $X_0$ as a Lie groupoid with only identity morphisms. Then there is the canonical bibundle $X_0 \to X$, which is the bundlization of the inclusion functor $X_0 \subset X$. If $M$ is any manifold with a bibundle $F \in \bibun(M, X)$, then the pullback is canonically isomorphic to the total space of $F$, viewed as a manifold. In particular, the pullback of $X_0 $ with itself over $X$ is the space $X_1$, thought of as a Lie groupoid with only identity morphisms. 
\end{example}

\begin{definition}
Let $F \in \bibun(X, Y)$. $F$ is a {\em covering bibundle} if it is representable and
 the map $\tau: F \to Y_0$ is a surjective submersion. 
\end{definition}

\begin{remark}
A bibundle  $F \in \bibun(X, Y)$ such that the map $\tau: F \to Y_0$ is a surjective submersion is transverse to every bibundle $G \in \bibun(Z, Y)$. 
\end{remark}

\begin{example}
For any groupoid $X=(X_1 \rightrightarrows X_0)$, the canonical bibundle from $X_0$ to $X$ is a covering bibundle. 
\end{example}

\section{2-Groups in Stacks}

\subsection{2-Groups}

Groups are pervasive in all subjects of mathematics and are an important and well studied subject. 2-Groups are a categorification of the notion of group, and have been playing an increasingly important role in many areas of mathematics and even physics. Recall the following slightly non-standard but equivalent definition of a group. 
Let $(G, 1, \cdot)$ be a monoid. We say $G$ is a {\em group} if the map,
\begin{align*}
	G \times G & \to G \times G \\
	(x,y) & \mapsto (x, x \cdot y)
\end{align*}
is a bijection. The inverse of this map allows one to find an element $g^{-1}$ for each element $g$ such that $g g^{-1} = g^{-1} g = 1$. Categorifying this definition yields the most succinct definition of 2-group of which I am aware. 

\begin{definition} \label{def:2-group}
	A monoidal category $(M, \otimes, 1, a, \ell, r)$ is a {\em 2-group} if the functor,
	\begin{equation*}
		(p_1, \otimes): M \times M \to M \times M
\end{equation*}
is an equivalence of categories, where $p_1$ is projection onto the first factor. The 2-category of 2-groups is the full sub-bicategory of the bicategory of monoidal categories whose objects consist of the 2-groups. 
\end{definition}

There are many equivalent descriptions of 2-groups which have arisen in various branches of mathematics. While the precise history of 2-groups is too intricate and convoluted to be done justice in this article, a few key highlights are in order. 
One of the earliest appearances of 2-groups arose in topology, without the aid of (higher) category theory. Since a 2-group is automatically a groupoid, its simplicial nerve will be a Kan simplicial set. Hence the geometric realization of a 2-group is automatically a homotopy 1-type (i.e. $\pi_i = 0$ at all base points for all $i > 1$). The geometric realization of a monoidal category is well known to be an $A_\infty$-space, and for 2-groups it is group-like. Thus it may be de-looped once to obtain a pointed connected homotopy 2-type $B|M|$. The (pointed) mapping spaces between pointed connected homotopy 2-types are automatically homotopy 1-types, and so by replacing the mapping space with its fundamental groupoid we obtain a bicategory which captures essentially all the homotopical information of homotopy 2-types. This bicategory is equivalent to the bicategory of 2-groups, and so the study of pointed connected 2-types (going back to the work of Whitehead and Mac Lane in the 1940's and 1950's) can be regarded as one of the earliest studies of 2-groups. 

It is well known that small monoidal categories can be strictified, that is replaced with equivalent monoidal categories where associativity and unit identities are satisfied on the nose. Doing this to a 2-group yields a so-called ``categorical group'', i.e. a (strict) group object in categories. A construction, known in the 1960s, shows that such categorical groups are essentially the same thing as {\em crossed modules}, a concept introduced by J. H. C. Whitehead in 1946 and later used by Whitehead and Mac Lane to classify pointed connected homotopy 2-types. 

Finally, another method of studying 2-groups is via {\em skeletalization} (introduced for 2-groups in \cite{BL04}) in which the 2-group is replaced by an equivalent 2-group which is skeletal\footnote{A category $C$ is skeletal if for all objects $x, y \in C$, the property $x \cong y$ implies $x=y$}. This yields a particularly simple description of each 2-group in terms of invariants: two ordinary groups $\pi_1$, $\pi_2$, and certain other data known collectively as the k-invariant. This classification is in direct correspondence with the usual classification of connected pointed 2-types in terms of Postnikov data.

\subsection{2-Groups in General Bicategories}

Monoid and group objects can be defined in any category with finite products, and a similar statement holds true for 2-groups. Following Baez and Lauda \cite{BL04} we introduce 2-group objects in arbitrary bicategories with finite products. Such a ``2-group'' consists of an object, $G$, together with a multiplication 1-morphism $m: G \times G \to G$, a unit 1-morphism $e: 1 \to G$, and several coherence 2-isomorphisms. Additionally it must satisfy a property which ensures that a coherent inverse map may be chosen. This is essentially a mild generalization of the definition of ``coherent  2-group objects'' defined in \cite{BL04}, modified to make sense in an arbitrary bicategory.

\begin{definition}[\cite{BL04}] \label{DefnWeak2group}
 Let $\sC$ be a bicategory with finite products. A {\em  2-group} in $\sC$ consists of an object $G$ together with 1-morphisms $e:1 \to G$,  $m: G \times G \to G$, %$I: G \to G$ 
 and invertible 2-morphisms: %$a,\ell, r, i, e$
\begin{align*}
	a : & \; m \circ (m \times id) \to m \circ (id \times m) \\
	\ell : &\;  m \circ (e \times id)  \to id \\
	r: & \; m \circ (id \times e) \to id \\
%	b: & \; e \circ \epsilon \to [m \circ (id \times I)] \circ \Delta \\
%	c: & \; [m \circ (I \times id)] \circ \Delta \to e \circ \epsilon
\end{align*}
  such that 
\begin{equation*}
(p_1, m) : G \times G \to G \times G
\end{equation*}
  is an equivalence in $\sC$ and the diagrams in Figures \ref{fig:pentagon} and \ref{fig:triangle} commute.

\begin{figure}[h]
	\caption{The `Pentagon' Identity}
	  \begin{center}
	\begin{tikzpicture}[thick]
		%\node at (0, 5.5) {\underline{The ``Pentagon'' Identity}};
		\node (LT) at (0,2) 	{$[m \circ (m \times id)] \circ (m \times id \times id) $ };
		\node (LTa) at (0,1) 	{$m \circ [(m \times id) \circ (m \times id \times id)] $ };
		\node (LB) at (1,-1) 	{$m \circ [(m \times id) \circ (id \times m \times id)]  $};
		\node (LBa) at (1,-2) 	{$[m \circ (m \times id)] \circ (id \times m \times id)  $};
		\node (RT) at (8,2) 	{$[m  \circ (id \times m)] \circ (id \times id \times m)$};
		\node (RTa) at (8,1) 	{$m  \circ [(id \times m) \circ (id \times id \times m)]$};

		\node (MT) at (1,4) {$[m \circ (id \times m)] \circ (m \times id \times id) $};
		\node (MTa) at (7,4) {$[m \circ (m \times id)] \circ (id \times id \times m)  $};

		\node (RB) at (7,-1)	{$m \circ [(id \times m) \circ (id \times m \times id)] $};
		\node (RBa) at (7,-2)	{$[m \circ (id \times m)] \circ (id \times m \times id) $};
		\draw [->] (LT) -- (LTa);
		\draw [->] (LB) -- (LBa);
		\draw [->] (RBa) -- (RB);
		\draw [->] (LTa) --  node [left] {$m \circ (a \times id) $} (LB);
		\draw [->] (LT) -- node [above left] {$a \circ (m \times id \times id) $} (MT);
		\draw [->] (MT) -- (MTa);
		\draw [->] (MTa) to node [above right] {$a \circ (id \times id \times m) $} (RT);
		\draw [->] (RT) -- (RTa);
		\draw [->] (RB) -- node [right] {$m \circ (id \times a) $} (RTa);
		\draw [->, in = 200, out = 340] (LBa) to node [below] {$a \circ (id \times m \times id) $} (RBa);
	\end{tikzpicture}
	\end{center}
	 \label{fig:pentagon}	
\end{figure}

\begin{figure}[h] 
	\caption{The `Triangle' Identity}
	\begin{center}
	\begin{tikzpicture}[thick]
	%	\node at (0,3.5) 	{\underline{The ``Triangle'' Identity}};
		\node (LT) at (0,1.5) 	{$ m\circ [(m  \circ (id \times e)) \times id]  $ };
		\node (LTa) at (0,2.5) 	{$ [m\circ (m \times id)] \circ (id \times e \times id)  $ };

		\node (RT) at (8,1.5) 	{$m  \circ [id  \times (m \circ ( e \times id)) ]$};
		\node (RTa) at (8,2.5) 	{$[m  \circ (id  \times m)] \circ (id \times e \times id) $};
		\node (B) at (4,0)	{$m \circ (id \times id)$};

		\draw [->] (LT) -- (LTa);
		\draw [->] (RTa) -- (RT); 
		\draw [->] (LT) --  node [below left] {$m \circ (r \times id) $} (B);
		\draw [->] (LTa) -- node [above] {$a \circ (id \times e \times id) $} (RTa);
		\draw [->] (RT) -- node [below right] {$m \circ (id \times \ell) $} (B);
	\end{tikzpicture}
	\end{center}
	\label{fig:triangle}
\end{figure}

The names of these diagrams have been chosen so as to correspond to the names of diagrams when $\sC$ is a strict 2-category. Thus the ``pentagon'' identity is no longer pentagonal in shape. 
The unlabeled arrows are the canonical isomorphisms given from associativity and products in $\sC$. 
\end{definition}

Just as the notion of 2-group presented in \cite{BL04} extends to an arbitrary bicategory $\sC$, so too do the notions of homomorphism and 2-homomorphism. Homomorphisms and 2-homomorphisms compose making a bicategory of 2-groups in $\sC$. A direct calculation shows that all 2-homomorphisms are invertible\footnote{Thus the category of 2-groups is an example of an $(\infty,1)$-category.}. 

\begin{definition}
Let $G$ and $G'$ be 2-groups in $\sC$. A {\em homomorphism} of 2-groups $G \to G'$ consists of:
\begin{itemize}
\item A 1-morphism $F: G \to G'$, 
\item 2-isomorphisms: $F_2: m' \circ (F \times F) \to F \circ m$ and  $F_0:   e' \to F \circ e$,
\end{itemize}
such that the three diagrams in Figures \ref{fig:2grpHom1}, \ref{fig:2grpHom2}, and \ref{fig:2grpHom3} commute. In these diagrams the unlabeled arrows are the canonical isomorphisms given from associativity and products in $\sC$.
\begin{figure}[ht]
	\caption{Axiom 1 for 2-Group homomorphisms.}
	\begin{center}		
	\begin{tikzpicture}[thick]
		\node (LT) at (0,5) 	{$[m' \circ (m' \times id)] \circ (F \times F \times F)$};
		\node (LMZ) at (0,3) 	{$[m' \circ (id \times m')] \circ (F \times F \times F)$};
		\node (LMY) at (0,2) 	{$m' \circ [ F  \times (m' \circ (F \times F ) ]$};
		\node (LB) at (0,0)	{$m' \circ [ F \times (F \circ m)]$};
		\node (LMB) at (1,-1) 	{$[m' \circ (F \times F)] \circ (id \times m)$ };
		\node (RMB) at (7,-1) {$[F \circ m] \circ (id \times m)$};
		\node (RB) at (8,0) 	{$F \circ [m \circ (id \times m)]$};

		\node (LMT) at (1,6)	{$m' \circ [( m' \circ (F \times F)) \times F]$};
		\node (RMT) at (7,6) 	{$ m' \circ [ (F \circ m) \times F ]$};
		\node (RT) at (8,5) 	{$[m' \circ (F \times F)] \circ (m \times id)$};
		\node (RMZ) at (8,3) 	{$[F \circ m] \circ (m \times id)$};
		\node (RMY) at (8,2)	{$F \circ [ m \circ (m \times id)]$};

		\draw [->] (LT) --  node [left] {$a' \circ (F \times F \times F)$} (LMZ);
		\draw [->] (LMZ) --  node [left] {$$} (LMY);
		\draw [->] (LMY) --  node [left] {$m' \circ (F \times F_2)$} (LB);
		\draw [->] (LB) -- node [below] {$$} (LMB);
		\draw [->] (LMB) -- node [below] {$F_2 \circ (id \times m)$} (RMB);
		\draw [->] (RMB) -- node [below] {$$} (RB);

		\draw [->] (LT) -- node [above] {$$} (LMT);
		\draw [->] (LMT) -- node [above] {$m' \circ (F_2 \times F)$} (RMT);
		\draw [->] (RMT) -- node [above] {$$} (RT);
		\draw [->] (RT) -- node [right] {$F_2 \circ (m \times id)$} (RMZ);
		\draw [->] (RMZ) -- node [right] {$$} (RMY);
		\draw [->] (RMY) -- node [right] {$F \circ a$} (RB);
		\end{tikzpicture}
	\end{center}	
	\label{fig:2grpHom1}
\end{figure}

\begin{figure}[ht]
	\caption{Axiom 2 for 2-Group homomorphisms.}
	\begin{center}
	\begin{tikzpicture}[thick]
		\node (LMZ) at (0,3) 	{$[m' \circ (e' \times id)] \circ (id_1 \times F)$};
		\node (LMY) at (0,2) 	{$[m' \circ (id \times F)] \circ (e' \times id)$};
		\node (LB) at (0,0)	{$[ m' \circ (id \times F) ] \circ [ (F \circ e) \times id]$};
		\node (LMB) at (1,-1) {$[m' \circ (F \times F)] \circ (e \times id)$ };
		\node (RMB) at (7,-1) {$[F \circ m] \circ (e \times id)$};
		\node (RB) at (8,0) 	{$F \circ [ m \circ (e \times id)]$};
		\node (RMY) at (8,2)	{$F \circ id$};
		\node (RMZ) at (8,3) 	{$(id_1 \times F)$};

		\draw [->, out = 30, in = 150] (LMZ)   to node [above] {$\ell' \circ (id_1 \times F)$} (RMZ);
		\draw [->] (LMZ) --  node [left] {$$} (LMY);
		\draw [->] (LMY) --  node [right] {$[m' \circ (id \times F)] \circ (F_0 \times id)$} (LB);
		\draw [->] (LB) -- node [below] {$$} (LMB);
		\draw [->] (LMB) -- node [below] {$F_2 \circ (e \times id)$} (RMB);
		\draw [->] (RMB) -- node [below] {$$} (RB);
		\draw [<-] (RMZ) -- node [right] {$$} (RMY);
		\draw [<-] (RMY) -- node [right] {$F \circ \ell$} (RB);
	\end{tikzpicture}
	\end{center}
	\label{fig:2grpHom2}
\end{figure}

\begin{figure}[ht]
	\caption{Axiom 3 for 2-Group homomorphisms.}
	\begin{center}
	\begin{tikzpicture}[thick]
		\node (LMZ) at (0,3) 	{$[m' \circ (id \times e')] \circ (F \times id_1)$};
		\node (LMY) at (0,2) 	{$[m' \circ (F \times id)] \circ (id \times e')$};
		\node (LB) at (0,0)	{$[ m' \circ (F\times id) ] \circ [ id \times (F \circ e)]$};
		\node (LMB) at (1,-1) {$[m' \circ (F \times F)] \circ (id \times e)$ };
		\node (RMB) at (7,-1) {$[F \circ m] \circ (id \times e)$};
		\node (RB) at (8,0) 	{$F \circ [ m \circ (id \times e)]$};
		\node (RMY) at (8,2)	{$F \circ id$};
		\node (RMZ) at (8,3) 	{$id \circ (F \times id_1)$};

		\draw [->, out = 30, in = 150] (LMZ)   to node [above] {$r' \circ ( F \times id_1)$} (RMZ);
		\draw [->] (LMZ) --  node [left] {$$} (LMY);
		\draw [->] (LMY) --  node [right] {$[m' \circ (F \times id)] \circ (id \times F_0)$} (LB);
		\draw [->] (LB) -- node [below] {$$} (LMB);
		\draw [->] (LMB) -- node [below] {$F_2 \circ (id \times e)$} (RMB);
		\draw [->] (RMB) -- node [below] {$$} (RB);
		\draw [<-] (RMZ) -- node [right] {$$} (RMY);
		\draw [<-] (RMY) -- node [right] {$F \circ r$} (RB);
	\end{tikzpicture}
	\end{center}
	\label{fig:2grpHom3}
\end{figure}

\end{definition}

\begin{definition}
Given two homomorphisms $F, K: G \to G'$ between 2-groups in $\sC$, a 2-homomorphism $\theta: F  \Rightarrow K$ is a 2-morphism such that the diagrams in Figure \ref{fig:2grp2homs} commute.
\begin{figure}[ht]
	\caption{Axioms of 2-Homomorphisms}
	\begin{center}
	\begin{tikzpicture}[thick]
		\node (LT) at (0,1.5) 	{$m' \circ (F \times F)$ };
		\node (LB) at (0,0) 	{$F \circ m$};
		\node (RT) at (5,1.5) 	{$m' \circ (K \times K)$};
		\node (RB) at (5,0)	{$K \circ m$};
		\draw [->] (LT) --  node [left] {$F_2$} (LB);
		\draw [->] (LT) -- node [above] {$m' \circ ( \theta \times \theta)$} (RT);
		\draw [->] (RT) -- node [right] {$K_2$} (RB);
		\draw [->] (LB) -- node [below] {$\theta \circ m$} (RB);
	\end{tikzpicture}
	\end{center}
	\begin{center}
	\begin{tikzpicture}[thick]
		\node (LT) at (1.5,1.5) 	{$e'$ };
		\node (LB) at (0,0) 	{$F \circ e$};
		\node (RT) at (3,1.5) 	{$$};
		\node (RB) at (3,0)	{$K \circ e$};
		\draw [->] (LT) --  node [above left] {$F_0$} (LB);
		%\draw [->>] (LT) -- node [above] {$$} (RT);
		\draw [->] (LT) -- node [above right] {$K_0$} (RB);
		\draw [->] (LB) -- node [below] {$\theta \circ e$} (RB);
	\end{tikzpicture}
	\end{center}
	\label{fig:2grp2homs}
\end{figure}

\end{definition}

\begin{definition}
	Let $\sC$ be a bicategory with finite products and $G$ be a 2-group in $\sC$. Let $X$ be an object in $\sC$. A {\em (left) $G$-action} on  $X$ consists of:
	\begin{itemize}
\item A 1-morphism $f: G \times X \to X$,
\item invertible 2-morphisms:
\begin{align*}
	a_f : & \; f \circ (m \times id) \to f \circ (id \times f) \\
	\ell_f : &\;  f \circ (e \times id)  \to id %\\
	%r: & \; f \circ (id \times e) \to id \\
\end{align*}
\end{itemize}
such that the diagrams in Figures \ref{fig:Pent2GrpAct} and \ref{fig:Tri2GrpAct} commute. %\notetoself{Fix these diagrams to have the correct labels.}
\begin{figure}[ht]
	\caption{The `Pentagon' Identity for 2-Group Actions.}
	\begin{center}
	\begin{tikzpicture}[thick]
		%\node at (0, 5.5) {\underline{The Pentagon Identity}};
		\node (LT) at (0,2) 	{$[f \circ (m \times id)] \circ (m \times id \times id) $ };
		\node (LTa) at (0,1) 	{$f \circ [(m\circ (m \times id) ) \times id] $ };
		\node (LB) at (1,-1) 	{$f \circ [(m \circ (id \times m)) \times id]  $};
		\node (LBa) at (1,-2) 	{$[f \circ (m \times id)] \circ (id \times m \times id)  $};
		\node (RT) at (8,2) 	{$[f  \circ (id \times f)] \circ (id \times id \times f)$};
		\node (RTa) at (8,1) 	{$f  \circ [id \times (f \circ (id \times f))]$};

		\node (MT) at (1,4) {$[f \circ (id \times f)] \circ (m \times id \times id) $};
		\node (MTa) at (7,4) {$[f \circ (m \times id)] \circ (id \times id \times f)  $};

		\node (RB) at (7,-1)	{$f \circ [id \times (f \circ (m \times id))] $};
		\node (RBa) at (7,-2)	{$[f \circ (id \times f)] \circ (id \times m \times id) $};
		\draw [->] (LT) -- (LTa);
		\draw [->] (LB) -- (LBa);
		\draw [->] (RBa) -- (RB);
		\draw [->] (LTa) --  node [left] {$f \circ (a \times id) $} (LB);
		\draw [->] (LT) -- node [above left] {$a_f \circ (m \times id \times id) $} (MT);
		\draw [->] (MT) -- (MTa);
		\draw [->] (MTa) to node [above right] {$a_f \circ (id \times id \times f) $} (RT);
		\draw [->] (RT) -- (RTa);
		\draw [->] (RB) -- node [right] {$f \circ (id \times a_f) $} (RTa);
		\draw [->, in = 200, out = 340] (LBa) to 
			node [below] {$a_f \circ (id \times m \times id) $} (RBa);
	\end{tikzpicture}
	\end{center}
	\label{fig:Pent2GrpAct}
\end{figure}

\begin{figure}[ht]
	\caption{The `Triangle' Identity for 2-Group Actions.}
	\begin{center}
	\begin{tikzpicture}[thick]
		%\node at (0,3.5) 	{\underline{The Triangle Identity}};
		\node (LT) at (0,1.5) 	{$ f \circ [(m  \circ (id \times e)) \times id]  $ };
		\node (LTa) at (0,2.5) 	{$ [f\circ (m \times id)] \circ (id \times e \times id)  $ };

		\node (RT) at (8,1.5) 	{$f  \circ [id  \times (f \circ ( e \times id)) ]$};
		\node (RTa) at (8,2.5) 	{$[f  \circ (id  \times f)] \circ (id \times e \times id) $};
		\node (B) at (4,0)	{$f \circ (id \times id)$};

		\draw [->] (LT) -- (LTa);
		\draw [->] (RTa) -- (RT); 
		\draw [->] (LT) --  node [below left] {$f \circ (r \times id) $} (B);
		\draw [->] (LTa) -- node [above] {$a_f \circ (id \times e \times id) $} (RTa);
		\draw [->] (RT) -- node [below right] {$f \circ (id \times \ell_f) $} (B);
	\end{tikzpicture}
	\end{center}
	\label{fig:Tri2GrpAct}
\end{figure}

\end{definition}

\begin{definition}
	Let $\sC$ be a bicategory with finite products, $G$ a 2-group in $\sC$ and $X$ and $Y$ two objects in $\sC$ equipped with $G$-actions. A {\em $G$-equivariant 1-morphism} from $X$ to $Y$ consists of:
	\begin{itemize}
\item a 1-morphism $g: X \to Y$ in $\sC$,  
\item invertible 2-isomorphisms in $\sC$: $\phi : f_Y \circ (1 \times g) \to g \circ f_X$
\end{itemize}
Such that the diagrams in Figures \ref{fig:Equiv1Mor1} and \ref{fig:Equiv1Mor2} commute.
\begin{figure}[ht]
	\caption{Axiom 1 of Equivariant 1-Morphism.}
	\begin{center}
	\begin{tikzpicture}[thick]
		\node (LT) at (0,5) 	{$[f_Y \circ (m \times id)] \circ (1 \times 1 \times g)$};
		\node (LMZ) at (0,3) 	{$[f_Y \circ (id \times f_Y)] \circ (1 \times 1 \times g)$};
		\node (LMY) at (0,2) 	{$f_Y \circ [ 1  \times (f_Y \circ (1 \times g )) ]$};
		\node (LB) at (0,0)	{$f_Y \circ [ 1 \times (g \circ f_X)]$};

		\node (LMB) at (1,-1) 	{$[f_Y \circ (1 \times g)] \circ (id \times f_X)$ };
		\node (RMB) at (7,-1) {$[g \circ f_X] \circ (id \times f_X)$};
		\node (RB) at (8,0) 	{$g \circ [f_X \circ (id \times f_X)]$};
		\node (LMT) at (1,6)	{$f_Y \circ [( m \circ (1 \times 1)) \times g]$};
		\node (RMT) at (7,6) 	{$ f_Y \circ [ (1 \circ m) \times g ]$};
		\node (RT) at (8,5) 	{$[f_Y \circ (1 \times g)] \circ (m \times id)$};
		\node (RMZ) at (8,3) 	{$[g \circ f_X] \circ (m \times id)$};
		\node (RMY) at (8,2)	{$g \circ [ f_X \circ (m \times id)]$};

		\draw [->] (LT) --  node [left] {$a_{f_Y} \circ (F \times F \times F)$} (LMZ);
		\draw [->] (LMZ) --  node [left] {$$} (LMY);
		\draw [->] (LMY) --  node [left] {$f_Y \circ (1 \times \phi)$} (LB);
		\draw [->] (LB) -- node [below] {$$} (LMB);
		\draw [->] (LMB) -- node [below] {$\phi \circ (id \times f_X)$} (RMB);
		\draw [->] (RMB) -- node [below] {$$} (RB);
		\draw [->] (LT) -- node [above] {$$} (LMT);
		\draw [->] (LMT) -- node [above] {$$} (RMT);
		\draw [->] (RMT) -- node [above] {$$} (RT);

		\draw [->] (RT) -- node [right] {$\phi \circ (m \times id)$} (RMZ);
		\draw [->] (RMZ) -- node [right] {$$} (RMY);
		\draw [->] (RMY) -- node [right] {$g \circ a_{f_X}$} (RB);
		\end{tikzpicture}
	\end{center}
	\label{fig:Equiv1Mor1}
\end{figure}

\begin{figure}[ht]
	\caption{Axiom 2 of Equivariant 1-Morphism.}
	\begin{center}
	\begin{tikzpicture}[thick]
		\node (LMZ) at (0,3) 	{$[f_Y \circ (e \times id)] \circ (id \times g)$};
		\node (LMY) at (0,2) 	{$[f_Y \circ (id \times g)] \circ (e \times id)$};
		\node (LB) at (0,0)	{$[ f_Y \circ (id \times g) ] \circ [ (1 \circ e) \times id]$};
		\node (LMB) at (1,-1) {$[f_Y \circ (1 \times g)] \circ (e \times id)$ };
		\node (RMB) at (7,-1) {$[g \circ f_X] \circ (e \times id)$};
		\node (RB) at (8,0) 	{$g \circ [ f_X \circ (e \times id)]$};
		\node (RMY) at (8,2)	{$g \circ id$};
		\node (RMZ) at (8,3) 	{$(id \times g)$};

		\draw [->, out = 30, in = 150] (LMZ)   to node [above] 
			{$\ell_{f_Y} \circ (id \times g)$} (RMZ);
		\draw [->] (LMZ) --  node [left] {$$} (LMY);
		\draw [->] (LMY) --  node [left] {$$} (LB);
		\draw [->] (LB) -- node [below] {$$} (LMB);
		\draw [->] (LMB) -- node [below] {$\phi \circ (e \times id)$} (RMB);
		\draw [->] (RMB) -- node [below] {$$} (RB);
		\draw [<-] (RMZ) -- node [right] {$$} (RMY);
		\draw [<-] (RMY) -- node [right] {$F \circ \ell_{f_X}$} (RB);
	\end{tikzpicture}
	\end{center}
	\label{fig:Equiv1Mor2}
\end{figure}
\end{definition}

\begin{definition}
	Given two $G$-equivariant 1-morphisms $g, g': X \to Y$, an {\em equivariant 2-morphism} $\rho: g \to g'$ is a 2-isomorphism such that the following square commutes:	
	\begin{center}
\begin{tikzpicture}[thick]
	\node (LT) at (0,1.5) 	{$f_Y \circ (1 \times g)$ };
	\node (LB) at (0,0) 	{$g \circ f_X$};
	\node (RT) at (5,1.5) 	{$f_Y \circ (1 \times g')$};
	\node (RB) at (5,0)	{$g' \circ f_X$};
	\draw [->] (LT) --  node [left] {$\phi$} (LB);
	\draw [->] (LT) -- node [above] {$f_Y \circ ( 1 \times \rho)$} (RT);
	\draw [->] (RT) -- node [right] {$\phi'$} (RB);
	\draw [->] (LB) -- node [below] {$\rho \circ f_X$} (RB);
\end{tikzpicture}
\end{center}
\end{definition}

\subsection{Abelian Groups in Bicategories}

The description in terms of monoidal categories makes it clear that there are two related notions of 2-group which generalize the notion of abelian group: braided 2-groups and symmetric 2-groups. In this work we will only be interested in the later, most highly commutative structure. A braided monoidal category is a monoidal category $G$, equipped with natural isomorphisms $\beta_{x,y}:x \otimes y \to y \otimes x$, which satisfy the requirement that  two different hexagonal diagrams commute. A symmetric monoidal category further satisfies the condition: $
\beta_{x, y} \beta_{y, x} = id$. 

 Following the discussion in \cite{JS93}, if the above symmetry equation is satisfied, then the hexagonal diagrams become redundant: only one is necessary, the other is a consequence.  Thus if one were interested only in defining symmetric monoidal categories and not braided monoidal categories, one could omit one of the hexagonal diagrams from the definition. This is the approach we take here.

\begin{definition}
Let $\sC$ be a bicategory with finite products. An {\em abelian 2-group} in $\sC$ consists of a group $(G, e, m, I, a, \ell, r)$ in $\sC$, together with a 2-isomorphism $\beta: m \to m \circ \tau$, where $\tau: G \times G \to G \times G$ is the ``flip'' 1-morphism in $\sC$, such that the following diagrams commutes (here the unlabeled arrows are canonical 2-morphisms from $\sC$):
\begin{center}
\begin{tikzpicture}[thick]
	\node (LT) at (0,1.5) 	{$m \circ \tau$ };
	\node (LB) at (0,0) 	{$m$};
	\node (RT) at (3,1.5) 	{$( m \circ \tau) \circ \tau$};
	\node (RB) at (3,0)	{$m$};
	\draw [<-] (LT) --  node [left] {$\beta$} (LB);
	\draw [->] (LT) -- node [above] {$\beta \circ \tau$} (RT);
	\draw [->] (RT) -- node [right] {$$} (RB);
	\draw [->] (LB) -- node [above] {$1$} (RB);
\end{tikzpicture}
\end{center}
 \begin{center}
\begin{tikzpicture}[thick]
	\node (LT) at (0,2) 	{$m \circ (1 \times m) $ };
	\node (LTa) at (0,0.5) 	{$m \circ (m \times 1)  $ };
	\node (LB) at (0,-1) 	{$m \circ ( [m \circ \tau] \times  1) $};
	\node (LBa) at (.5,-2.5) 	{$[m \circ (m \times 1)] \circ [\tau \times 1]$};
	
	\node (RT) at (8,2) 	{$[m \circ (m \times 1)]  \circ [(1 \times \tau) \circ (\tau \times 1)] $};
	\node (RTa) at (8,0.5) 	{$ [m  \circ  (1 \times m)]  \circ [(1 \times \tau) \circ (\tau \times 1)] $};
	\node (MT) at (4,3.5) {$[m \circ \tau] \circ (1 \times m)$};
	\node (RB) at (8,-1)	{$[m \circ (1 \times (m \circ \tau)] \circ [\tau \times 1]$};
	\node (RBa) at (7.5,-2.5)	{$[m \circ (1 \times m) ] \circ [\tau \times 1]$};
	
	\draw [->] (LTa) -- node [left] {$a$} (LT);
	
	\draw [->] (LB) -- (LBa);
	\draw [->] (RBa) -- node [left] {$m \circ (1 \times \beta)] \circ [ \tau \times 1]$} (RB);
	\draw [->] (RT) -- node [left] {$a \circ [(1 \times \tau) \circ (\tau \times 1)]$} (RTa);   
	\draw [->] (LTa) --  node [right] {$m \circ (\beta \times 1) $} (LB);
	\draw [->] (LT) -- node [above left] {$\beta \circ (1 \times m)$} (MT);
	\draw [->] (MT) -- node [above right] {$$} (RT);
	\draw [->] (RB) -- node [right] {$$} (RTa);
	\draw [->, in = 200, out = 340] (LBa) to node [below] {$a \circ [\tau \times 1]$} (RBa);
\end{tikzpicture}
\end{center}

\end{definition}

\begin{definition}
Let $G$ and $G'$ be abelian 2-groups in $\sC$. A {\em homomorphism} of abelian 2-groups consists of a homomorphism $(F, F_2, F_0): G \to G'$ of underlying groups, such that the following diagram commutes:
\begin{center}
\begin{tikzpicture}[thick]
	\node (LT) at (0,1.5) 	{$m' \circ (F \times F) $ };
	\node (LB) at (0,0) 	{$F \circ m$};
	\node (MT) at (3, 3) {$[m' \circ \tau] \circ (F \times F) $};
	\node (RT) at (6,1.5) 	{$[m' \circ (F \times F)] \circ \tau$};	
	\node (RB) at (6,0)	{$[F \circ m] \circ \tau$};
	\node (MB) at (3, -1.5) {$F \circ [m \circ \tau]$};
	\draw [->] (LT) --  node [left] {$F_2$} (LB);
	\draw [->] (LT) -- node [above left] {$\beta' \circ (F \times F)$} (MT);
	\draw [->] (MT) -- node [above] {$$} (RT);
	\draw [->] (RT) -- node [right] {$F_2 \circ \tau$} (RB);
	\draw [->] (LB) -- node [below left] {$F \circ \beta$} (MB);
	\draw [->] (MB) -- node [below] {$$} (RB);
\end{tikzpicture}
\end{center}
 
 A {\em 2-homomorphism} between homomorphisms of abelian 2-groups in $\sC$ consists of a 2-homomorphism of underlying homomorphisms of groups in $\sC$. 
\end{definition}

\begin{remark}
2-Groups and abelian 2-groups in bicategories are defined diagrammatically. Thus if $h:\sC \to \sC'$ is a product preserving 2-functor and $G$ is an (abelian) 2-group object in $\sC$, then $h(G)$ is canonically an (abelian) 2-group object in $\sC'$. Similarly, $h$ sends $G$-objects  in $\sC$ to $h(G)$-objects in $\sC'$. 
\end{remark}

%\begin{remark}
%The requirement that $\sC$ have all finite products is not necessary to define 2-group objects and abelian 2-group objects in $\sC$. The only necessary products are those used in the corresponding definitions. 
%\end{remark}

\subsection{2-Groups as a Localization}

The bicategory of 2-groups (in $\cat$) admits a succinct description as a localization of the bicategory of monoidal groupoids (and hence also as a localization of monoidal categories). This description will play a small technical role in this paper and so we offer a brief summary of this approach and a proof of its equivalence to the one already introduced (Definition \ref{def:2-group}). Along the way we will encounter several equivalent descriptions of 2-groups. 

\begin{definition}
	A monoidal category $M$ {\em admits functorial inverses} if there exists a functor $i: M \to M$ and a natural isomorphism $x \otimes i(x) \cong 1$. A {\em choice of functorial inverses} shall refer to a specific choice of functor $i$ and corresponding natural isomorphism. 
\end{definition}

\begin{lemma}
	If $M$ is a monoidal category which admits functorial inverses, then for any choice of functorial inverses $i$ there is a natural isomorphism $i^2(x) \cong x$.
\end{lemma}

\begin{proof}
	$i^2(x) \cong 1 \otimes i^2(x)   \cong  (x \otimes i(x)) \otimes i^2(x) \cong x \otimes (i(x) \otimes i^2(x)) \cong x \otimes 1 \cong x$. 
\end{proof}

\begin{lemma}
	If $M$ is a monoidal category which admits functorial inverses, then the underlying category of $M$ is a groupoid.
\end{lemma}

\begin{proof}
	Let $f: x \to y$ be a morphism in $M$. Its inverse $f^{-1}: y \to x$ is given by the following composition of natural morphisms,
	\begin{equation*}
		f^{-1}: y \cong i^2(y) \cong   ( x \otimes i(x)) \otimes i^2(y) \stackrel{ (1 \otimes i(f)) \otimes 1}{\longrightarrow}   ( x \otimes i(y)) \otimes i^2(y) \cong  x \otimes (i(y) \otimes i^2(y)) \cong x.
	\end{equation*}
\end{proof}

\begin{lemma} \label{Lma:functinversegivebijectonhoms}
	If $M$ is a monoidal category which admits functorial inverses and $f: x \to x'$ is a morphism in $M$, then for all objects $y,z \in M$ the following maps are bijections,
	\begin{align*}
		 (-) \otimes f: C(y,z) &\to C(  y \otimes x,  z \otimes x'), \\
		f \otimes (-): C(y,z) &\to C( x \otimes  y ,  x' \otimes z). 
	\end{align*}
\end{lemma}

\begin{proof}
	The inverse to the first bijection is obtained by choosing a functorial inverse $i$, applying the functor $(-) \otimes i(f)$, and using the natural isomorphisms $y \cong (y \otimes x) \otimes i(x)$ and $z \cong (z \otimes x') \otimes i(x')$. The inverse to the second is obtained similarly, making use of the isomorphism $i^2(x) \cong x$. 
\end{proof}

The bicategory of monoidal categories admits all small weak colimits. This can be seen, for example, by observing that the bicategory of monoidal categories is equivalent to the bicategory of algebras for a 2-monad on the bicategory of categories \cite[Section 6]{BKP89}. Similarly the bicategory $\gpd^\otimes$ of monoidal groupoids (i.e. of those monoidal category whose underlying categories are groupoids) is equivalent to the bicategory of algebras for the same 2-monad restricted to the bicategory of groupoids.  All the 2-morphisms of $\gpd^\otimes$ are invertible, hence it fits into the formalism of $(\infty,1)$-categories as considered in \cite{Lurie09} and \cite{Lurie07}, and we may therefore bring to bear the sophisticated machinery developed in those sources in our study of 2-groups. In particular $\gpd^\otimes$ is {\em presentable} in the sense of \cite[Def 5.5.0.1]{Lurie09}. We will see the relevance of this shortly. 

Any monoid may be regarded as a monoidal groupoid in which all morphisms are identity morphism and where the monoidal structure is given by multiplication in the monoid. The natural numbers $\N$, viewed as a monoidal category in this way, are free in the sense that the we have a natural equivalence of categories $\hom(\N, C) \simeq C$ for any monoidal category $C$, where $\hom(A,B)$ denotes the category of monoidal functors from $A$ to $B$. 

\begin{definition}
	Define the monoidal category $\F$ as the weak pushout in the following diagram of monoidal categories.
\begin{equation}\label{eqn:diagramdefiningfree2gp}
	\begin{tikzpicture}[baseline = 0.75 cm]
		\node (LT) at (0, 1.5) {$\N$};
		\node (LB) at (0, 0) {$0$};
		\node (RT) at (2, 1.5) {$\N \times \N$};
		\node (RB) at (2, 0) {$\F$};
		\draw [->] (LT) -- node [left] {$$} (LB);
		\draw [->] (LT) -- node [above] {$\Delta$} (RT);
		\draw [->] (RT) -- node [right] {$$} (RB);
		\draw [->] (LB) -- node [below] {$$} (RB);
		%\node at (0.5, 1) {$\ulcorner$};
		\node at (1.5, 0.5) {$\lrcorner$};
		\node at (1, 0.75) {$\Downarrow$};
	\end{tikzpicture}
\end{equation}
%	\begin{center}
%	\end{center}
	The inclusion $\N \times 0 \to \N \times \N$ induces a map of monoidal categories $s: \N \to \F$.
\end{definition}

\begin{definition}
	A monoidal category $M$ is {\em $s$-local} if the induced functor,
	\begin{equation*}
		s^*: M^\F = \hom(\F, M) \to \hom(\N, M) \simeq M,
	\end{equation*}
	is an equivalence of categories. 
\end{definition}

\begin{theorem} \label{thm:charof2groups}
	For a monoidal category $M$ the following conditions are equivalent
	\begin{enumerate}
		\item $M$ is a 2-group (i.e. $(p_1, \otimes): M \times M \to M \times M$ is an equivalence),
		\item $M$ is $s$-local, 
		\item $M$ admits functorial inverses.
	\end{enumerate}
\end{theorem}

\begin{proof}
	Let $M$ be a monoidal category. Applying $\hom(-, M)$ to the comutative diagram in Equation (\ref{eqn:diagramdefiningfree2gp}) we obtain the first of the following pair of weak pull-back squares of categories,
	\begin{center}
	\begin{tikzpicture}
		\node (LT) at (0, 1.5) {$M \simeq M^\N$};
		\node (LB) at (0, 0) {$0$};
		\node (RT) at (3, 1.5) {$M \times M \simeq M^{\N \times \N}$};
		\node (RB) at (3, 0) {$M^\F$};
		\draw [<-] (LT) -- node [left] {$$} (LB);
		\draw [<-] (LT) -- node [above] {$\otimes$} (RT);
		\draw [<-] (RT) -- node [right] {$$} (RB);
		\draw [<-] (LB) -- node [below] {$$} (RB);
		%\node at (0.5, 1) {$\ulcorner$};
		\node at (2.5, 0.5) {$\lrcorner$};
		\node at (1.5, 0.75) {$\Downarrow$};
		
		\node (LLT) at (6, 1.5) {$M$};
		\node (LLB) at (6, 0) {$0$};
		\node (LRT) at (8, 1.5) {$M \times M$};
		\node (LRB) at (8, 0) {$M$};
		\draw [<-] (LLT) -- node [left] {$$} (LLB);
		\draw [<-] (LLT) -- node [above] {$p_2$} (LRT);
		\draw [<-] (LRT) -- node [right] {$i_1$} (LRB);
		\draw [<-] (LLB) -- node [below] {$$} (LRB);
		%\node at (0.5, 1) {$\ulcorner$};
		\node at (7.5, 0.5) {$\lrcorner$};
		\node at (7, 0.75) {$\Downarrow$};
	\end{tikzpicture}.
	\end{center}
If $M$ is a 2-group, then these pull-back squares are equivalent via the equivalence $(p_1, \otimes): M \times M \to M \times M$. Thus the natural map $s^*:M^\F \to M$ is an equivalence, and so 2-groups are $s$-local. 

By construction, a functor $\F \to M$ is equivalent to a pair of objects $x, \overline{x} \in M$ together with an equivalence $\alpha: x \otimes \overline{x} \cong 1$. The functor $s^*: M^\F \to M$ sends the triple $(x, \overline{x}, \alpha)$ to the object $x$. If $M$ is $s$-local then we have an inverse equivalence $M \to M^\F$, and hence a {\em functorial} choice of inverse $\overline{x}$ for every object $x \in M$. In other words, $M$ admits functorial inverses. 
	
Finally, suppose that $M$ admits functorial inverses. We wish to show that $M$ is a 2-group, i.e. that the natural functor $(p_1, \otimes): M \times M \to M \times M$ is an equivalence of categories. Given a monoiadal category $C$, the collection of isomorphism classes of objects, $\pi_0C$, is a monoid. For a category which admits functorial inverses, $M$, the monoid $\pi_0 M$ is a group, and hence $(p_1, \otimes)$ is a bijection on isomorphism classes of objects. It remains to show that $(p_1, \otimes)$ is fully-faithful, i.e. that for all objects $x,y,x', y' \in M$, the natural map,
\begin{equation*}
	M(x,y) \times M(x', y') \to M(x,y) \times M(x \otimes x', y \otimes y') 
\end{equation*}
	is a bijection. This in turn is equivalent to the statement that for each $f: x \to y$, the map $f \otimes (-): M(x', y') \to M(x \otimes x', y \otimes y')$ is a bijection, which is part of the statement of Lemma \ref{Lma:functinversegivebijectonhoms}.

%	$\pi_0$ of an $s$-local category is a group. Choice of inverses can be made functorial, by choosing an inverse functor to $C^F \to C$, therefore $C$ is a grouoid. By the lemma below $(p_1, \otimes)$ is a bijection on hom sets (fully-faithful), therefore an equivalence. 

\end{proof}

\begin{remark}
Examining the last part of the above proof and the proof of Lemma \ref{Lma:functinversegivebijectonhoms}, one observes that if the underlying category of $M$ is a groupoid, then $(p_1, \otimes)$ is an equivalence (and hence $M$ is a 2-group) precisely if $\pi_0M$ is a group. Thus our definition agrees with the notion of ``weak 2-group'' given in \cite[Definition 2]{BL04}. This characterization allows one to deduce that $\F$ itself is a 2-group: $\F$ is a monoidal groupoid, being a colimit of such, and moreover $\pi_0 \F \cong \Z$ is a group (this last follow from the definition of $\F$ and from the fact that $\pi_0$ sends colimits of monoidal groupoids to colimits of monoids). Finally, using the skeletal classification of 2-groups \cite[Section 8.3]{BL04} and the universal property that $M^\F \simeq M$ for all 2-groups $M$, one may deduce the monoidal equivalence $\F \simeq \Z$. Alternatively, one may simply compute the push-out defining $\F$ and deduce this equivalence. We will not make use of this in what follows.   
\end{remark}

%Choosing an adjoint inverse equivalence allows one to construct, among other things, a {\em functorial} assignment $x \mapsto \overline x$ together with functorial isomorphisms $x \otimes \overline x \cong 1$, for every object $x \in M$. This has the further consequence that the underlying category of $M$ is a groupoid, and hence agrees with the notion of ``weak 2-group'' given in \cite{BL04}. Moreover we see from this description that the category of such choices is a {\em contractible category} (i.e. equivalent to the terminal category) and so for all intents and purposes it is enough to know merely that such a choice exists.  Any two such choices are uniquely equivalent. 

\begin{corollary} \label{cor:2grptogrpdismonadic}
	The bicategory of 2-groups is cocomplete, the inclusion of 2-groups into monoidal groupoids admits a weak left adjoint, and this adjunction is 2-monadic.  
\end{corollary}

\begin{proof}
	By the above theorem, 2-groups are precisely the $s$-local objects of $\gpd^\otimes$. Since this later is a presentable $(\infty,1)$-category, the first two claims are direct statements from \cite[Proposition 5.5.4.15]{Lurie09}, from which it also follows that 2-groups form a {\em strongly reflective sub-bicategory} of $\gpd^\otimes$ \cite[pg 482]{Lurie09}. 
	
The final statement follows from general principles as the localizing adjunction from any presentable $(\infty,1)$-category to a strongly reflective sub-$(\infty,1)$-category is monadic. This is classical for ordinary categories and an identical argument applies to the higher categorical setting, as follows. Let $i: 2\grp \to \gpd^\otimes$ be the inclusion functor, $L$ its left adjoint and $T = iL$ the corresponding 2-monad. Then for every 2-group $X$, the adjunction induces the structure of a $T$-algebra on $iX$. Moreover, since $i$ is fully-faithful, any $T$-algebra structure on $iX$ is equivalent to this canonical one. Thus it is sufficient to show that if $Y$ is an arbitrary $T$-algebra, then $Y$ is in fact a 2-group. This follows since for any $T$-algebra $Y$ the structure morphism $h:TY \to Y$ is an equivalence, with inverse $\eta_Y: Y \to TY$.  
%	[I really should spell this out.]
%	The sub-bicategory of 2-groups is closed under equivalence. 
%	$T = i \circ L$ is a 2-monad. For every 2-group $G$, the underlying monoidal category $i(G)$ is a $T$-algebra via the canonical map $i \varepsilon: i \circ L \circ i (G) \to i(G)$. 
%	Since $i$ is fully-faithful (i.e. induces equivalences on hom categories) the counit of the adjunction $\varepsilon L\circ i --> id$ is an equivalence. 
%	If $h: T(X) \to X$ is a $T$-algebra, then $h$ is an equivalence, and so $X$ was already a 2-group. 
%\begin{eqnarray}
%	\eta_X \circ h &\simeq& (Th) \circ \eta_{TX} \\
%	&\simeq& (Th) \circ (i \varepsilon^{-1}_{LX}) \circ  (i \varepsilon_{LX}) \circ \eta_{i L X} \\
%	& \simeq & (Th) \circ (i \varepsilon^{-1}_{LX}) \\
%	& \simeq & (Th) \circ (i \varepsilon^{-1}_{LX}) \circ (i \varepsilon_{LX}) \circ (i L \eta_X) \\
%	& \simeq & (Th)  \circ (i L \eta_X) \\
%	& \simeq & T (h \circ \eta_X) \simeq id_{TX}
%\end{eqnarray}	
\end{proof}

\begin{corollary} \label{cor:2grptogrpdismonadic2}
	The forgetful functor from 2-groups to groupoids admits a weak left adjoint and the resulting adjunction is 2-monadic. 
\end{corollary}

\begin{proof}
	The composition of 2-monadic adjunctions remains 2-monadic, and so the statement follows from the previous corollary and from the fact that the forgetful functor from monoidal groupoids to groupoids is part of a 2-monadic adjunction \cite[Section 6]{BKP89}.
\end{proof}

The results of this section can be applied to 2-groups in a bicategory $\sC$ much more general then $\sC = \cat$. Let $\sS$ be an essentially small category\footnote{The following construction works when $\sS$ is an essentially small bicategory, but we will only need the case where $\sS$ is an ordinary category.}. Let $\sC$ be a localization (i.e. reflexive sub-bicategory) of the functor category $\fun(\sS^{op}, \gpd)$. For each object $U \in \sS$, let $L_U: \gpd \to \sC$ be the left-adjoint to evaluation at $U$. Let $\sC^\otimes$ denote to bicategory of monoidal objects in $\sC$.\footnote{These are defined identically to 2-groups in $\sC$, except for the requirement that $(p_1, m): M \times M \to M \times M$ is an equivalence.} The functor $L_U$ induces a functor $L_U: \gpd^\otimes \to \sC^\otimes$. Let $s_U: L_U(\N) \to L_U(\Z)$ denote the canonical map of objects in $\sC^\otimes$. An object of $\sC^\otimes$ is {\em $S$-local} if it is local with respect to all $s_U$.\footnote{Since $\sS$ is essentially small, the $S$-local objects of $\sC$ are a further reflexive sub-bicategory of $\fun(\sS^{op}, \gpd)$. } We have the following theorem:

\begin{theorem}
	Let $\sC$ be a localization of $\fun(\sS^{op}, \gpd)$, as above, and let $M \in \sC^\otimes$. Then the following statements are equivalent:
	\begin{enumerate}
		\item $M$ is a 2-group in $\sC$.
		\item $M(U)$ is a 2-group (in $\gpd$) for every object $U \in \sS$. 
		\item $M(U)$ admits functorial inverses for every object in $U\in \sS$.
		\item $M(U)$ is $s$-local for every object in $U\in \sS$.
		\item $M$ is an $S$-local object of $\sC^\otimes$. 
	\end{enumerate}
	Moreover the adjunction $F:\sC \leftrightarrows 2\grp(\sC):U$ (induced by the forgetful functor from $2\grp(\sC)$ to $\sC$) is monadic.
\end{theorem}

\begin{proof}
	The equivalences (1) $\Leftrightarrow$ (2) and (4) $\Leftrightarrow$ (5) follow from the bicategorical Yonneda lemma, and the fact that $\sC$ is a full subbicategory of $\fun(\sS^{op}, \gpd)$. The equivalences (2) $\Leftrightarrow$ (3) $\Leftrightarrow$ (4) follow from Theorem \ref{thm:charof2groups}. The proof of Corollaries  \ref{cor:2grptogrpdismonadic} and \ref{cor:2grptogrpdismonadic2} carry over immediately to show the final statement. 
\end{proof}

\begin{example} \label{Example:2GrpInStacksIsMonadic}
	The category $\sC = \stack$ of all stacks (not necessarily presentable) on the site $\sS = \man$ of smooth manifolds with the surjective submersion topology is a localization of $\fun(\man^{op}, \gpd)$. Hence the adjunction $F:\stack \leftrightarrows 2\grp(\stack):U$ between stacks and 2-groups in stacks is monadic. 
\end{example}

%%%%%%%%%%%%%%%%%%%
% For the general argument, we look at presheaves of groupoids on a category C. We can look at monoidal objects which are the same as pointwise monoidal objects. For each object U in C there is an evaluation to monoidal groupoids and this has a (weak) left adjoint $L_U$. Thus for each U we get a $s_U: L_U(N) \to L_U(\F)$ map and an S-local object is local with respect to each $s_U$. 
% The proof of Theorem ... immeadiately generalizes to show that objectwise 2-groups are the same as S-local are the same as admitting objectwise inverses. Moreover an objectwise 2-group is equivalent to a 2-group in presheaves. The same holds for other localizations of presheaves (we simply apply the further localization functor L). For example we can apply this to 2-groups in stacks. We obtain that the adjunction from 2-groups in stacks to stacks is 2-monadic. 

%%%%%%%%%%%%%

\subsection{Smooth 2-Groups and Gerbes}

We now specialize to the case $\sC = \bibun$, the bicategory of Lie groupoids, bibundles, and bibundle morphisms. We will also refer to the objects of $\bibun$ as {\em smooth} stacks.

\begin{definition}
A {\em smooth 2-group} (resp. {\em smooth abelian 2-group}) is a 2-group object (resp. abelian 2-group object) in $\bibun$. Let $G$ be a smooth 2-group. Then a smooth $G$-stack $X$ is a $G$-object in $\bibun$. Similarly, if $X$ is a smooth stack, a {\em smooth 2-group over $X$} is a group object of $\bibun/X$. Let $G$ be a smooth 2-group over $X$, then a {\em smooth $G$-stack over $X$} is a $G$-object in $\bibun/X$. 
\end{definition}

\begin{remark}\label{RmkDiscreteInSmooth}
	If $X$ is a discrete space, then any surjective submersion $P \to X$ admits a {\em global} section. Hence a bibundle whose source is a discrete groupoid is equivalent to one arising from a functor (see Proposition \ref{PropSectionIsBundlization}). In particular we have that the composite 2-functor $\gpd \hookrightarrow \Lie \gpd \to \bibun$ is fully-faithful. Thus any smooth 2-group whose underlying Lie groupoid is discrete arises from a discrete 2-group and we may regard the theory of discrete 2-groups as a special case of smooth 2-groups.  
\end{remark}

\begin{example}[Lie Groups] \label{ExampleLieGrp}
Let $G$ be a Lie group, viewed as a Lie groupoid with only identity morphisms. Then $G$ is a smooth 2-group with monoidal structure coming from the multiplication in $G$.
\end{example}

\begin{example}[Abelian Lie Groups] \label{ExampleAbLieGrp}
Let $A$ be an abelian Lie group. Let $[pt/A]$ denote the Lie groupoid with a single object and with automorphism of this object equal to $A$. Since $A$ is a abelian, addition is a group homomorphism $A \times A \to A$. Thus addition equips $[pt/A]$ with a monoidal structure. Moreover, there is a (trivial) braiding making this into a smooth abelian 2-group. 
\end{example}

\begin{example}[Crossed Modules] \label{ExampleCrossedMod2grp}
A crossed module of Lie groups, $\beta:H \to G$, is well known to be equivalent to a group object in the category of Lie groups. Thus a crossed module gives rise to a Lie 2-group (and hence a smooth 2-group) in which the associator and unitor structures are trivial. The translation from a crossed module to a Lie 2-group is as follows, see  \cite{BL04}. The objects consist of the manifold $G$. The morphisms consist of the manifold $G \times H$. The source map is projection onto the $G$-factor. The target map is given by
\begin{equation*}
t(g,h) = g \cdot \beta(h).
\end{equation*}
Composition is given by $(g_0, h_0) \circ (g \beta(h_0), h_1) = (g, h_0 h_1)$. 
Viewing the morphisms as the group $G \rtimes H$, both the source and target maps are group homomorphisms. Group multiplication in $G$ and $G \rtimes H$ equip this Lie groupoid with a strict monoidal structure, with strict inverses. The underlying stack of this smooth 2-group is the quotient stack $[G/H]$.

A particularly important example of such a crossed module is $H \to Aut(H)$, sending an element to the conjugation automorphism. The corresponding smooth 2-group $[Aut(H) / H]$ plays a key role in the theory of non-abelian bundle gerbes described in \cite{ACJ05}, to which we will turn shortly.  

\end{example}

\begin{example}[Smooth Cocycles]\label{ExampleCocycle2Grp}
Let $G$ be a Lie group and $A$ an abelian Lie group equipped with an action of $G$. Let $a \in Z^3(G, A)$ be a smooth normalized group cocycle. Following \cite{BL04} we may form the following smooth 2-group $\Gamma = (G, A, a)$. The objects of $\Gamma$ consist of the manifold $G$, the morphisms are the space $G \times A$. There are no morphism from an element $g$ to $g'$, unless $g = g'$. In this case the morphisms are identified with $A$, which is the fiber of the projection map $G \times A \to A$. The monoidal structure is given by group multiplication at the level of objects and by the multiplication in the group $G \rtimes A$ at the level of morphisms. This is a strictly associative multiplication. Nevertheless, we equip $\Gamma$ with the nontrivial associator determined by $a$. This determines the unitor and inversion structures up to natural isomorphism. 
\end{example}

Given a object $X \in \bibun$ we can consider the product 2-functor:
\begin{equation*}
	\times X: \bibun  \to \bibun / X.
\end{equation*}
This is product preserving, and sends any smooth 2-group to a smooth 2-group over $X$. If $G$ is a smooth 2-group, then by a smooth $G$-stack over $X$ we will mean a smooth $G \times X$-stack over $X$. 
 If $U \to X$ is any surjective submersion, then the pullback functor
\begin{equation*}
	\bibun/X \to \bibun/U
\end{equation*}
 is product preserving as well, hence sends smooth $G$-stacks over $X$ to smooth $G$-stacks over $U$. 

\begin{example}
Let $A$, $B$, and $C$ be smooth 2-group, let $f: A \to B$ and $g: B \to C$ be homomorphisms. Then $f$ gives $B$ the structure of a smooth $A$-stack. Moreover, if  $\phi: gf \to 0$ is a 2-homomorphism, then the action of $A$ on $B$ induced by $f$ may be canonically augmented via $\phi$ to an action over $C$. Hence in this case $B$ is an $A$-object over $C$. 
\end{example}

\begin{definition}[Principal Bundles]
Let $G$ be a smooth 2-group. A smooth $G$-stack $Y$ over $X$ is a {\em $G$-principal bundle} if it is locally trivial as a $G$-stack, i.e. there exists a covering bibundle $f:U \to X$ such that $f^*Y$ is $G$-equivariantly equivalent to $U \times G$ as a smooth $G$-stack over $U$. 
\end{definition}

\begin{example}[Ordinary Principal Bundles]
Let $G$ be a Lie group, thought of as a smooth 2-group with only identity morphisms. Let $X$ be a manifold thought of as a Lie groupoid with only identity morphisms. Then a $G$-principal bundle over $X$ is the same as a $G$-principal bundle over $X$ in the usual sense.
\end{example}

\begin{example}[Abelian Cocycle Data] \label{ExampleCocycleGerbe}
	Let $A$ be an abelian Lie group, $X$ a manifold. Let $Y \to X$ be a surjective submersion, and fix a \v{C}ech 2-cocycle $\lambda: Y^{[3]} = Y \times_X Y \times_X Y \to A$. We construct a $[pt/A]$-principal bundle $E^\lambda$ over $X$ as follows. The objects of $E^\lambda$ consist of the manifold $Y$. The morphisms of $E^\lambda$ consist of the manifold $Y^{[2]} \times A$. Composition is given by the formula,
\begin{align*}
	E^\lambda_1 \times_Y E^\lambda_1 &\to E^\lambda_1 \\
	(y_0, y_1, a) \times (y_1, y_2, b) & \mapsto (y_0, y_2, a + b + \lambda(y_0, y_1, y_2)).
\end{align*}
This is readily seen to be a Lie groupoid equipped with a map to $X$ and an $[pt/A]$-action over $X$. The pullback of $E^\lambda$ along $Y \to X$ is a trivial $[pt/A]$-space over $Y$, so that $E^\lambda$ is indeed a principal bundle.
\end{example}

\begin{example}[Abelian Bundle Gerbes]
		More generally, an {\em $A$-bundle gerbe over $X$} in the sense of Murray \cite{Murray96} is an $[pt/A]$-principal bundle over $X$. Let $A$ be an abelian Lie group, $X$ a manifold. An $A$-bundle gerbe over $X$ consists of $Y \to X$ a surjective submersion, $L \to Y^{[2]}$ an $A$-principal bundle, together with an isomorphism of $A$-principal bundles over $Y^{[3]}$,
	\begin{equation*}
		\lambda: d_2^*L \otimes_A d_0^* L \to d_1^*L
\end{equation*}
such that the induced map $d\lambda: Y^{[4]} \to A$ is trivial. From this we construct an $[pt/A]$-principal bundle $E^\lambda$ over $X$ %in the sense of this paper 
as follows.  The objects of $E^\lambda$ are the manifold $Y$. The morphisms are the elements of the manifold $L$, with source and target maps induced from the maps,
\begin{equation*}
	L \to Y^{[2]} \rightrightarrows Y.
\end{equation*}
Composition is induced from the map $\lambda$, and is associative because $d \lambda = 0$. There exists a covering $c:U \to Y$ such that the induced $A$-bundle $c^*L \to U^{[2]}$ is trivial, and pulling $E^\lambda$ back to $U$ consequently yields a trivial $[pt/A]$-stack over $U$. 
	\end{example}

The last example can be given a more conceptual description. Just as maps from a space to a topological group again form a group, maps from a manifold to a smooth 2-group form a 2-group. A map from a manifold $Z$ to the 2-group $[pt/A]$ consists precisely of an $A$-bundle $L$ over $Z$. The (2-)group structure is given by tensoring $A$-bundles. Thus the above abelian bundle gerbe data consists of a map of stacks $L:Y^{[2]} \to [pt/A]$, together with an isomorphism of bibundles $\lambda: dL \to 0$ from $Y^{[3]}$ to $[pt/A]$, such that $d\lambda$ is the canonical isomorphism $d^2 L \cong 0$ of bibundles from $Y^{[4]}$ to $[pt/A]$. 

\begin{example}[Classical Nonabelian Bundle Gerbes]
The classical data of a non-abelian bundle gerbe as described in \cite{ACJ05} consists of a non-abelian Lie group $H$, a manifold $X$, a surjective submersion $Y$, a bibundle $\cE$ from $Y^{[2]}$ to the smooth crossed module 2-group $[Aut(H)/H]$ from Example \ref{ExampleCrossedMod2grp}, and an isomorphism of bibundles $\lambda: d \cE \to 1$ from $Y^{[3]}$ to  $[Aut(H)/H]$, such that $d\lambda$ is the canonical isomorphism $d^2 \cE \cong 0$ of bibundles from $Y^{[4]}$ to $[Aut(H)/A]$.

This can be made into an $[Aut(H)/H]$-principal bundle over $X$ in a manner analogous to the above constructions. The objects of this Lie groupoid are the manifold $Y \times Aut(H)$, and the morphisms are the manifold $\cE \times Aut(H)$. Composition is defined using the above structures and there is an induced action of $[Aut(H)/H]$. By choosing a covering of $Y$ appropriately (so that the pullback of $\cE$ can be trivialized) we see that this yields a principal bundle over $X$. 
\end{example}

One of the advantages of using the bicategory $\bibun$ to define the notion of gerbe is that it automatically produces the correct notion of {\em equivalence} of gerbe over $X$. To see this, consider a covering $U \to X$ and the corresponding \v{C}ech groupoid $X_U = (U^{[2]} \rightrightarrows U)$. There is a canonical functor $X_U \to X$ given by projection. This functor is almost never an equivalence in $\Lie\gpd$, see Example \ref{PairGroupoidEquiv}. However, the bundlization is {\em always} an invertible bibundle. It is an equivalence is in $\bibun$. For this reason the stable equivalences which need to be formally inverted in the approaches given in \cite{Murray96, ACJ05}, correspond to honest equivalences of $G$-stacks over $X$. This approach is similar to the ones presented in \cite{Bartels04} and \cite{Bakovic07}.

In the last section of this paper we will provide a model of the String group as a smooth 2-group which is not of the form of the Examples  \ref{ExampleLieGrp},  \ref{ExampleAbLieGrp}, \ref{ExampleCrossedMod2grp}, and \ref{ExampleCocycle2Grp}. Nevertheless, the above material allows us to discuss principal String$(n)$-bundles over a given manifold $X$. The notion of string structure introduced in \cite{Waldorf09} yields such a principal String$(n)$-bundle for the model of String$(n)$ constructed in the final section of this paper.

\subsection{Extensions of 2-Groups}

\begin{definition}
An {\em extension} of a smooth 2-group $G$ by a smooth 2-group $A$ consists of a smooth 2-group $E$, a homomorphisms $f: A \to E$ , a homomorphism $g: E \to G$, and a 2-homomorphism $\phi: gf \to 0$, such that $E$ is an $A$-principal bundle over $G$. 
\end{definition}

\begin{lemma}
	For any extension of smooth 2-groups as above, the following diagram is a (homotopy) pull-back in the bicategory of smooth 2-groups.
	%\begin{center}
	\begin{equation}\label{EqnExtensionAsPullBackPushOut}
	\begin{tikzpicture}[baseline = 0.75 cm]
		\node (LT) at (0, 1.5) {$A$};
		\node (LB) at (0, 0) 	 {$0$};
		\node (RT) at (2, 1.5) {$E$};
		\node (RB) at (2, 0)    {$G$};
		\draw [->] (LT) -- node [left] {$$} (LB);
		\draw [->] (LT) -- node [above] {$f$} (RT);
		\draw [->] (RT) -- node [right] {$g$} (RB);
		\draw [->] (LB) -- node [below] {$$} (RB);
		\node at (1, 0.75) {$\Downarrow \phi$};
		\node at (0.3, 1.2) {$\ulcorner$};
	%	\node at (1.7, 0.3) {$\lrcorner$};
	\end{tikzpicture}
	\end{equation}
	%\end{center}
\end{lemma}
\begin{proof}
	The forgetful functor from smooth 2-groups to stacks reflects equivalences and is a (weak) right adjoint, hence preserves (homotopy) pull-backs. Since $A$ is the pull-back in stacks, it follows that $A$ is also the pull-back in smooth 2-groups. 
\end{proof}

	There is an obvious generalization to extensions of smooth abelian 2-groups. In that context, the above square is also a push-out square. In the non-abelian setting this fails, a fact which was graciously pointed out to us by an anonymous reviewer. Nevertheless, the above square does satisfy a closely related universal property, which we now formulate. 

\begin{definition}
	A {\em kernel square} of smooth 2-groups consists of a pull-back diagram 
	\begin{equation*}
	\begin{tikzpicture}[baseline = 0.75 cm]
		\node (LT) at (0, 1.5) {$X$};
		\node (LB) at (0, 0) 	 {$0$};
		\node (RT) at (2, 1.5) {$Y$};
		\node (RB) at (2, 0)    {$Z$};
		\draw [->] (LT) -- node [left] {$$} (LB);
		\draw [->] (LT) -- node [above] {$f$} (RT);
		\draw [->] (RT) -- node [right] {$g$} (RB);
		\draw [->] (LB) -- node [below] {$$} (RB);
		\node at (1, 0.75) {$\Downarrow \phi$};
		\node at (0.3, 1.2) {$\ulcorner$};
	%	\node at (1.7, 0.3) {$\lrcorner$};
	\end{tikzpicture}
	\end{equation*}	
	i.e. homomorphisms $f: X \to Y$, $g: Y \to Z$, and a 2-homomorphism $\phi: gf \cong 0$, in which $X$ is the pull-back. 
\end{definition}

The 1-morphisms and 2-morphisms of kernel squares are the obvious ones for diagrams in bicategories. Given an extension of smooth 2-groups $A \stackrel{f}{\to} E \stackrel{g}{\to} G$, we may consider the sub-bicategory $\sK\sS(f)$ of kernel squares,
\begin{equation*}
\begin{tikzpicture}[baseline = 0.75 cm]
	\node (LT) at (0, 1.5) {$A$};
	\node (LB) at (0, 0) 	 {$0$};
	\node (RT) at (2, 1.5) {$E$};
	\node (RB) at (2, 0)    {$Z$};
	\draw [->] (LT) -- node [left] {$$} (LB);
	\draw [->] (LT) -- node [above] {$f$} (RT);
	\draw [->] (RT) -- node [right] {$h$} (RB);
	\draw [->] (LB) -- node [below] {$$} (RB);
	\node at (1, 0.75) {$\Downarrow \phi$};
	\node at (0.3, 1.2) {$\ulcorner$};
%	\node at (1.7, 0.3) {$\lrcorner$};
\end{tikzpicture}
\end{equation*}
and those 1-morphisms and 2-morphism which restrict to the identity of $f: A \to E$. 

\begin{proposition} \label{Prop:UniversalPropOfExten}
	Given an extension $A \to E \to G$ of smooth 2-groups as above, the kernel square $A \to E \to G$ is the initial object of $\sK\sS(f)$.
\end{proposition}

Before proving the above proposition, we must first introduce a technical lemma.
\begin{lemma} \label{Lma:TechnicalSimplicialColimit}
	If $A \to E \to G$ is an extension of smooth 2-groups, then the induced augmented simplicial object in smooth 2-groups
	\begin{equation*}
		G \leftarrow E \leftleftarrows E \times_G E \lefttriplearrows E \times_G E \times_G E \cdots
	\end{equation*}
	is a (homotopy) colimit diagram. 
\end{lemma}

\begin{proof}[Proof of Proposition \ref{Prop:UniversalPropOfExten}, assuming Lemma \ref{Lma:TechnicalSimplicialColimit}]
	Let $E^\bullet$ denote the simplicial 2-group in stacks, 
\begin{equation*}
		E \leftleftarrows E \times_G E \lefttriplearrows E \times_G E \times_G E \cdots.
\end{equation*}	
Given a kernel square $A \to E \to X$, there exists a 2-homomorphism from $E^\bullet$ to the constant simplicial object $X$, which agrees with the given homomorphism $E \to X$ on the $\text{zero}^\text{th}$  objects. Moreover the category of such 2-homomorphisms of simplicial objects is contractible. This can be seen, for example, by identifying the $A$-stack $E \times_G \cdots \times_G E$ with the $A$-stack $E \times A \times \cdots \times A$. 
	
Thus, by Lemma \ref{Lma:TechnicalSimplicialColimit}, the homomorophism $E \to X$ factors uniquely through the homomorphism $E \to G$. Because both $A \to E \to G$ and $A \to E \to X$ are pull-back squares, this extends to an essentially unique morphism of kernel squares.
\end{proof}

It remains to prove Lemma \ref{Lma:TechnicalSimplicialColimit}. Up to this point we have been primarily concerned with 2-groups in the bicategory of presentable stacks in the surjective submersion topology on smooth manifolds. It is surely possible to provide a direct proof of Lemma \ref{Lma:TechnicalSimplicialColimit} entirely within this bicategory. However, for the sake of brevity, we will now also contemplate the bicategory $\stack$ of {\em all stacks} for the surjective submersion topology on manifolds. In other words, $\stack$ is the bicategory of all fibered categories over $\man$ (fibered in groupoids) which satisfy stack descent with respect to the surjective submersion topology. 

Unlike $\bibun$ this larger bicategory is complete and cocomplete (in the higher categorical sense), and there exists a monadic adjunction
\begin{equation*}
	F:\stack \rightleftarrows 2\grp(\stack): U
\end{equation*}
  between stacks and the bicategory of 2-groups in stacks (see Example \ref{Example:2GrpInStacksIsMonadic}). Thus we may now apply the higher categorical Barr-Beck Theorem \cite[Theorem 3.4.5]{Lurie07}, which has the following corollary. 

\begin{corollary} \label{Cor:BarrBeckCor}
	Let $F: \sC \rightleftarrows \sD: U$ be a monadic adjunction of bicategories (for example $\sC= \stack$ and $\sD =  2\grp(\stack)$). Then for any object $X \in \sD$, the colimit of the simplicial object $FU^{\bullet}(X)$ exists in $\sD$ and agrees with the object $X$. 
\end{corollary}

\begin{proof}
	This follows immediately from Theorem 3.4.5 in \cite{Lurie07}, as the augmented simplicial object $X \leftarrow FU^\bullet(X)$ is $U$-split. 
\end{proof}

\begin{proof}[Proof of Lemma \ref{Lma:TechnicalSimplicialColimit}]
	Consider the augmented simplicial object
	\begin{equation*}
			G \leftarrow E \leftleftarrows E \times_G E \lefttriplearrows E \times_G E \times_G E \cdots
	\end{equation*} 
	which we write $G \leftarrow E^\bullet$. Because $E$ is an $A$-principal bundle over $G$, the diagram $U(G) \leftarrow U(E^\bullet)$ is a colimit diagram in $\stack$ (and hence in presentable stacks). Since $F$ is a left adjoint, it preserves colimits. Hence the diagram $FU(G) \leftarrow FU(E^\bullet)$ is a colimit diagram in 2-groups in stacks. 
	
	We may now consider the bisimplicial diagram of 2-groups in stacks $\{FU^p(E^{[q]})\}$, where $E^{[q]} = E \times_G \cdots \times_G E$, $q$-times. We may compute the colimit of this diagram in two ways, each consisting of two steps. We may first take the colimit in the $p$-direction, after which we obtain the simplicial diagram $E^\bullet$. Hence the colimit of this bisimplicial 2-group is precisely the colimit of $E^\bullet$. 
	
	On the other hand, we may instead take the colimit first in the $q$-direction. For each fixed $p$, this is the colimit of $FU^p(E^\bullet)$, which we have already observed is $FU^p(G)$. Thus the colimit of this bisimplicial 2-group is also the colimit of the simplical digram $FU^\bullet(G)$. By Corollary \ref{Cor:BarrBeckCor}, this is precisely $G$. 
\end{proof}

%%%%%%

%\begin{remark}
%For an extension of smooth 2-groups as above, the following diagram is (homotopy) cartesian and cocartesian.
%%\begin{center}
%\begin{equation}\label{EqnExtensionAsPullBackPushOut}
%\begin{tikzpicture}[baseline = 0.75 cm]
%	\node (LT) at (0, 1.5) {$A$};
%	\node (LB) at (0, 0) 	 {$0$};
%	\node (RT) at (2, 1.5) {$E$};
%	\node (RB) at (2, 0)    {$G$};
%	\draw [->] (LT) -- node [left] {$$} (LB);
%	\draw [->] (LT) -- node [above] {$f$} (RT);
%	\draw [->] (RT) -- node [right] {$g$} (RB);
%	\draw [->] (LB) -- node [below] {$$} (RB);
%	\node at (1, 0.75) {$\Downarrow \phi$};
%	\node at (0.3, 1.2) {$\ulcorner$};
%	\node at (1.7, 0.3) {$\lrcorner$};
%\end{tikzpicture}
%\end{equation}
%%\end{center}

Given a smooth abelian 2-group $A$ and a smooth 2-group $G$ the central extensions of $G$ by $A$ form a bicategory  $Ext(G; A)$. A 1-morphism of extensions $(E,f, g, \phi) \to (E', f', g', \phi')$, consists of an equivalence of 2-groups $h: E \to E'$, together with 2-isomorphisms $\alpha: hf \cong f'$ and $\beta: g'h \cong g$, such that the following pasting diagram is the identity:
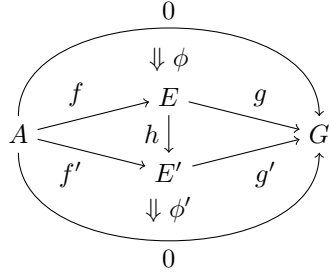
\begin{center}
\begin{tikzpicture}%[baseline = 0.75 cm]
	\node (L) at (0, 0) {$A$};
%	\node (LB) at (0, 0) 	 {$0$};
	\node (MT) at (2, 0.5) {$E$};
	\node (MB) at (2, -0.5) {$E'$};
	\node (R) at (4, 0)    {$G$};
	\draw [->] (L) -- node [above left] {$f$} (MT);
	\draw [->] (L) -- node [below left] {$f'$} (MB);
	\draw [->] (MT) -- node [above right] {$g$} (R);
	\draw [->] (MB) -- node [below right] {$g'$} (R);
	\draw [->] (MT) -- node [left] {$h$} (MB);
	\draw [->] (L) to [out = 90, in = 90] node [above] {$0$}(R) ;
	\draw [->] (L) to [out = -90, in = -90] node [below] {$0$}(R) ;
	\node at (2, 1) {$\Downarrow \phi$};
	\node at (2, -1) {$\Downarrow \phi'$};
	
%	\node at (0.3, 1.2) {$\ulcorner$};
%	\node at (1.7, 0.3) {$\lrcorner$};
\end{tikzpicture}
%\end{equation}
\end{center}
The 2-morphisms in $Ext(G;[pt/A])$ are given by 2-isomorphisms $\psi: h \to h'$, such that $\alpha = \alpha' \circ (\psi * id_f) $ and $\beta = \beta' \circ (if_g * \psi)$. By construction the bicategory $Ext(G; A)$ is contravariantly 2-functorial in $G$, covariantly 2-functorial in $A$, and commutes with products.  Thus the usual Baer sum operation equips $Ext(\G; \A)$ with the structure of a symmetric monoidal bicategory (See \cite{GPS95, KV94, KV94-2, BN96, DS97} and \cite[Chap. 3]{SchommerPries09}). 

Since the 2-category of discrete 2-groups embeds into the bicategory of smooth 2-groups (see Remark \ref{RmkDiscreteInSmooth}) this gives a notion of extension of discrete 2-groups. 
%Every gerbe over a discrete space is automatically trivial, so in the discrete setting this notion simplifies. An extension of discrete 2-groups consists of a pair of homomorphisms $f: A \to E$ and $g: E \to G$ together with a 2-homomorphism $\phi: gf \to 0$ such that the square (\ref{EqnExtensionAsPullBackPushOut}) is homotopy cartesian and cocartesian. 
%\end{remark}
Just as the theory of discrete groups is more elementary then the theory of topological groups, so too the theory of discrete 2-groups is easier then the theory of topological (or smooth) 2-groups. Nevertheless, we can learn many things by comparing these two settings. For example, extensions of groups $G$ by abelian groups $A$ are categorized according to the induced action of $G$ on $A$. In the topological setting, such actions are more problematic because such an action should be required to be a {\em continuous} homomorphism $G \to Aut(A)$, where this latter group is the group of continuous automorphisms of $A$. This is defined using the internal hom for topological spaces. 

While this doesn't pose a significant problem for finite dimensional Lie groups, there are further issues to contend with in the case of smooth stacks. Noohi \cite{Noohi08} and Carchedi \cite{Carchedi09} have independently gone to some pains to understand conditions under which internal homs exist in the setting of topological stacks.\footnote{The internal hom always exists as a fibered category and automatically satisfies stack descent. The main problem is in proving that such fibered categories are {\em presentable} by smooth or topological groupoids.} The smooth situation is likely more difficult. 

These difficulties disappear in the discrete setting, where internal homs for the 2-category of groupoids are well known to exist. Even though the general theory of smooth actions of groups presents problems, it is still possible to define {\em central} extensions merely by comparing with the discrete case. An extension of topological groups $A \to E \to G$ is a central extension precisely when it is a central extension of discrete groups, after forgetting the topology. We will employ the same strategy to define central extensions of smooth 2-groups. 

\begin{lemma}
Given an extension of discrete 2-groups $A \to B \to C$ with $A$ an abelian 2-group, there exists a homomorphism of 2-groups $C \to Aut(A)$, unique up to unique 2-homomorphism, where $Aut(A)$ is the automorphism 2-group of $A$. 
\end{lemma}

\begin{proof}
Let $f: A \to B$,  $g: B \to C$, and $\phi: gf \to 0$ be the homomorphisms and 2-homomorphism of the extension of 2-groups. Choose a functorial assignment $b \mapsto \overline b$ of weak inverses for the elements of $B$ together with functorial isomorphisms $b \otimes \overline b \cong 1$. Such an assignment is unique up to unique isomorphism. 

For each object $b \in B$ we may form an automorphism of $B$ given by conjugation. Specifically we consider the functor defined on object $x \in B$ by $x \mapsto b \otimes [ x \otimes \overline b]$. This can be made a self homomorphism of $B$ by using the structure maps of $B$. It is compatible with composition and induces a homomorphism of 2-groups $B \to Aut(B)$.

Pre-composing with $f$ yields, for each $b$,  a new homomorphism $f^b: A \to B$ together with a 2-homomorphism $\phi^b: g f^b \to 0$. On objects we have $f^b(a) = b \otimes [f(a) \otimes \overline{b}]$, where $\overline{b}$ is the functorial weak inverse of $b$. The structure morphisms of $f^b$ and $\phi^b$ are canonically induced by those of $f$, $\phi$, $g$, and the 2-group structure of $B$. A morphism $b \to b'$ in $B$ induces a natural isomorphism $f^b \to f^{b'}$.
Thus by the universal property of the pull-back, for each object of $b$ we have a homomorphism as in the following diagram.
\begin{center}
\begin{tikzpicture}
	\node (A) at (-1, 3) {$A$};
	\node (LT) at (0, 1.5) {$A$};
	\node (LB) at (0, 0) 	 {$0$};
	\node (RT) at (2, 1.5) {$B$};
	\node (RB) at (2, 0)    {$C$};
	\draw [->] (LT) -- node [left] {$$} (LB);
	\draw [->] (LT) -- node [above] {$f$} (RT);
	\draw [->] (RT) -- node [right] {$g$} (RB);
	\draw [->] (LB) -- node [below] {$$} (RB);
	\node at (1, 0.75) {$\Downarrow \phi$};
	\node at (0.3, 1.2) {$\ulcorner$};
	\node at (1.7, 0.3) {$\lrcorner$};
	\draw [->] (A) to [out = 0, in = 120] node [above] {$f^b$}(RT);
	\draw [->, dashed] (A) to node [above right] {$h(b)$}(LT);
	\draw [->] (A) to [out = 270, in = 120] node [above] {$$}(LB);
\end{tikzpicture}
\end{center}
Morphisms $b \to b'$  in $B$ induce 2-homorphisms $h(b) \to h(b')$. The assignment $h: b \mapsto f^b$, $h: ( b \to b') \mapsto (f^b \to f^{b'})$ is not strictly canonical, but depends  upon a contractible category of choices. 

The assignment $b \mapsto h(b)$ is compatible with the multiplication in $B$ in the sense that $h(b \otimes b') \cong h(b) \circ h(b')$. These isomorphisms may be chosen to be functorial and yield a homomorphism of 2-groups $h: B \to Aut(A)$. Again this choice is unique up to unique isomorphism. Pre-composition yields a homomorphism $A \to Aut(A)$. However, if $A$ is abelian, then the braiding allows us to canonically trivialize this composite. 

We now use the universal property of the extension $A \to B \to C$ to factor $h$ by an essentially unique homomorphism $B \to C \to Aut(A)$. The trivialization of $A \to Aut(A)$ permits us to form the following square:
\begin{equation*}
\begin{tikzpicture}[baseline = 0.75 cm]
	\node (LT) at (0, 1.5) {$A$};
	\node (LB) at (0, 0) 	 {$0$};
	\node (RT) at (2, 1.5) {$B$};
	\node (RB) at (2, 0)    {$C \times Aut(A)$};
	\draw [->] (LT) -- node [left] {$$} (LB);
	\draw [->] (LT) -- node [above] {$f$} (RT);
	\draw [->] (RT) -- node [right] {$(g, h)$} (RB);
	\draw [->] (LB) -- node [below] {$$} (RB);
	\node at (1, 0.75) {$\Downarrow \phi$};
%	\node at (0.3, 1.2) {$\ulcorner$};
%	\node at (1.7, 0.3) {$\lrcorner$};
\end{tikzpicture}
\end{equation*}
This square is readily checked to be a kernel square. Since $A \to B \to C$ is the initial such kernel square, there exists an essentially unique morphism of kernel squares from $A \to B \to C$ to the above kernel square. In particular this consists of a homomorphism $C \to C \times Aut(A)$, and projecting to the second factor yields the desired homomorphism $C \to Aut(A)$.
%Hence by the universal property of the push-out, there exists a homomorphism $\tilde{h}: C \to Aut(A)$, unique up to unique 2-homomorphism. 
\end{proof}

Given a smooth 2-group, we obtain a discrete 2-group by applying the forgetful 2-functor $U:\bibun \to \gpd$, which forgets the topology. Thus to every  extension of smooth 2-group we get a corresponding extension of discrete 2-groups. 

\begin{definition}
	An extension of discrete 2-groups $A \to B \to C$, with $A$ abelian, is {\em central} if the induced homomorphism $C \to Aut(A)$ is isomorphic to the trivial homomorphism. An extension of smooth 2-groups is {\em central} if the corresponding extension of discrete 2-groups is central.
\end{definition}
By the work of \cite{BL04}, every discrete 2-group is equivalent to a {\em skeletal} 2-group\footnote{A discrete 2-group $\Gamma$ is skeletal if the for all objects $x, x' \in \Gamma$ the condition $x \cong x'$ implies $x = x'$.} and these are classified by the following invariants:
\begin{enumerate}
\item A group $\pi_0 \Gamma$ = isomorphism classes of objects,
\item An abelian $\pi_1\Gamma = Hom_\Gamma( 1,1)$,
\item An action $\rho$ of $\pi_0\Gamma$ on $\pi_1\Gamma$, (induced by conjugating an automorphism of $1$ by an object in $G$),
\item The k-invariant $[a] \in H^3(\pi_0\Gamma; \pi_1 \Gamma, \rho)$, which is determined by the associator of $\Gamma$,
\end{enumerate}
see \cite{BL04} for details. Part of this classification is the construction of a 2-group from a this data. This construction will play a role in what follows, so we review it. Consider the data $(G, A, \rho, \alpha)$ where $G$ is a group, $A$ is an abelian group, $\rho$ is an action of $G$ on $A$ and $\alpha \in Z^3_\text{grp}(G; A, \rho)$ is a group cocycle. Then we may form the following skeletal 2-group $\Gamma(G, A, \rho, \alpha)$: The objects of $\Gamma$ are the elements $G$, the morphisms of $\Gamma$ are the product space $G \times A$, with both source and target maps the projection to $G$. The automorphisms of each object are identified with the fiber $A$, as a group. The monoidal structure is given by the group multiplication in $G$ on objects and the group multiplication of $A \rtimes_\rho G$ on morphisms. It is strictly associative, nevertheless we equip it with a nontrivial associator determined by $\alpha$. The associator is given by $a_{g_0, g_1, g_2} = \alpha ({g_0, g_1, g_2} ) \in A $, using the identification of $A$ with the automorphisms of $g_0 g_1 g_2$. That $\alpha$ is a cocycle ensures that the pentagon identity is satisfied. The left and right unitors are uniquely determined by $\alpha$ and the requirement that the triangle identity hold. %We may take $I$ to be inversion on the objects of $G$, and $b$ and $c$ to be identities, although any other choice yields an equivalent 2-group.  

There exist canonical homomorphisms of 2-groups $f: [pt/A] \to \Gamma(G, A, \rho, \alpha)$ and $g: \Gamma(G, A, \rho, \alpha) \to G$ given by the obvious inclusion and projection. The composition $gf$ is equal to the zero map $[pt/A] \to G$, which in this case has no non-identity automorphisms.

\begin{lemma}
Consider an extension of discrete 2-groups of the form 
	\begin{center}
\begin{tikzpicture}[thick]
	%\node (LL) at (0,0) 	{$1$ };
	
	\node (L) at (2,0) 	{$\downdownarrows$};

	\node (LA) at (2,.5) 	{$A$};
	\node (LB) at (2,-.5) 	{$pt$};
	\node (M) at (4, 0) {$\downdownarrows$};
	\node (MA) at (4,.5) 	{$E_1$};
	\node (MB) at (4,-.5) 	{$E_0$};

	\node (R) at (6,0)	{$\downdownarrows$};
	\node (RA) at (6,.5)	{$G$};
	\node (RB) at (6, -.5)	{$G$};
	
	%\node (RR) at (8,0)	{$1$};

	%\draw [->] (LL) --  node [left] {$$} (L);
	\draw [->] (L) -- node [above] {$f$} (M);
	\draw [->] (M) -- node [above] {$g$} (R);
	%\draw [->] (R) -- node [below] {$$} (RR);
\end{tikzpicture}
\end{center}
(the 2-homomorphism $\phi: gf \to 0$ is unique if it exists, hence is determined by $f$ and $g$). 
Then this extension is equivalent to an extension such that $E = \Gamma(G, A, \rho, \alpha)$, with $f$ and $g$ the canonical inclusion an projection homomorphisms. Moreover, all such inclusion-projection sequences are extensions, and such an extension is central if and only if the action $\rho$ is trivial. 
\end{lemma}

\begin{proof}
By the work of \cite{BL04} we know that $E$ is equivalent to {\em some} skeletal 2-group $E \simeq \Gamma( H, B, \rho, \alpha)$, so it suffices to consider that case. The homomorphisms $f$ and $g$ are determined by their component homomorphisms $f_1: A \to B$ and $g_0: H \to G$. 

Consider the 2-group $[pt/ B]$. This has a canonical inclusion functor $\ell: [pt/B] \to E$ whose composition with $g$ is zero. Thus by the universal property of the pullback, there exists a (unique) homomorphism $[pt/B] \to [pt/A]$ which factors this inclusion. In particular $f_1$ must be a split surjection. 

Now consider the kernel $\ker f_1$, with its inclusion $j: \ker f_1 \to A$. This yields a homomorphism of 2-groups $j: [pt/ \ker f_1] \to [pt/ A]$, such that the following two diagrams commute (strictly):
\begin{center}
\begin{tikzpicture}
	\node (A) at (-1.5, 3) {$[pt/ \ker f_1]$};
	\node (LT) at (0, 1.5) {$[pt/A]$};
	\node (LB) at (0, 0) 	 {$0$};
	\node (RT) at (2, 1.5) {$E$};
	\node (RB) at (2, 0)    {$G$};
	\draw [->] (LT) -- node [left] {$$} (LB);
	\draw [->] (LT) -- node [above] {$f$} (RT);
	\draw [->] (RT) -- node [right] {$g$} (RB);
	\draw [->] (LB) -- node [below] {$$} (RB);
	\node at (1, 0.75) {$$};
	\node at (0.3, 1.1) {$\ulcorner$};
	\node at (1.7, 0.3) {$\lrcorner$};
	\draw [->] (A) to [out = 0, in = 120] node [above] {$0$}(RT);
	\draw [->] (A) to node [above right] {$j$}(LT);
	\draw [->] (A) to [out = 270, in = 140] node [left] {$!$}(LB);
\end{tikzpicture}
\qquad
\begin{tikzpicture}
	\node (A) at (-1.5, 3) {$[pt/ \ker f_1]$};
	\node (LT) at (0, 1.5) {$[pt/A]$};
	\node (LB) at (0, 0) 	 {$0$};
	\node (RT) at (2, 1.5) {$E$};
	\node (RB) at (2, 0)    {$G$};
	\draw [->] (LT) -- node [left] {$$} (LB);
	\draw [->] (LT) -- node [above] {$f$} (RT);
	\draw [->] (RT) -- node [right] {$g$} (RB);
	\draw [->] (LB) -- node [below] {$$} (RB);
	\node at (1, 0.75) {$$};
	\node at (0.3, 1.1) {$\ulcorner$};
	\node at (1.7, 0.3) {$\lrcorner$};
	\draw [->] (A) to [out = 0, in = 120] node [above] {$0$}(RT);
	\draw [->] (A) to node [above right] {$0$}(LT);
	\draw [->] (A) to [out = 270, in = 140] node [left] {$!$}(LB);
\end{tikzpicture}
\end{center}
By the universal property of the pullback this implies that $j \cong 0$ and hence $f_1: A \to B$ is an isomorphism.

Dually, consider the group $H$ viewed as a 2-group with only identity morphisms. The canonical projection $E = \Gamma(H, B, \rho, \alpha) \to H$ is a homomorphism such that the composite 
\begin{equation*}
[pt/A] \to E = \Gamma(H, B, \rho, \alpha) \to H
\end{equation*}
is isomorphic to the zero homomorphism (if such an isomorphism exists it is unique). Thus there exists an essentially unique morphism $[pt/A] \to K$, where $K = \ker (E \to H)$. Conversely, since the map $E \to G$ factors as $E \to H \to G$, we obtain a unique map $K \to [pt/A]$. These are easily checked to be inverses so that $K = [pt/A]$, and hence $[pt/A] \to E \to H$ is a kernel square. 

Thus by the universal property of the extension $[pt/A] \to E \to G$, there exists a (unique) group homomorphism $k:G \to H$ making the following diagram commute.
\begin{center}
\begin{tikzpicture}
%	\node (LT) at (0, 1.5) {$[pt/A]$};
%	\node (LB) at (0, 0) 	 {$0$};
	\node (RT) at (2, 1.5) {$H$};
	\node (RB) at (2, 0)    {$G$};
	\node (H) at (3.5, 0) {$H$};
	%\draw [->] (LT) -- node [left] {$$} (LB);
	%\draw [->] (LT) -- node [above] {$$} (RT);
	\draw [->] (RT) -- node [left] {$g_0$} (RB);
	%\draw [->] (LB) -- node [below] {$$} (RB);
	\draw [->] (RT) -- node [above right] {$id$} (H);
	\draw [->, dashed] (RB) -- node [below left] {$k$} (H);
\end{tikzpicture}
\end{center}
In particular the kernel of $g_0: H \to G$ is zero. 

We may compose the map $E \to H$ with the homomorphism $g_0: H \to G$, and thereby obtain a map of kernel squares from $[pt/A] \to E \to G$ to itself which restricts to the map $g_0 \circ k: G \to G$. Since this kernel square is initial, this composite must be the identity on $G$. Thus $g_0$ is an isomorphism.

%We can compose the projection $E \to H$ with the homomorphism $g_0: H \to G$, and by the universal property of a pushout, the composition $g_0 \circ k: G \to G$ must be the identity. Thus $g_0$ is surjective. 
%
%Let $i:  \ker g_0 \hookrightarrow H$ be the natural inclusion, and consider the 
%2-group $K = \Gamma( \ker g_0, B, \rho \circ i, i^* \alpha)$. There is a canonical functor $K \to E$ which on objects is given by the inclusion $i: \ker g_0 \to H$ and on morphisms is given by the identity. The composition with $g: E \to G$ yields the zero homomorphism, and thus this homomorphism factors as a composition $K \to [pt/A] \to E$. Any such composition induces that zero map on objects, so in particular the inclusion $i : \ker g_0 \to H$ is the zero map. Hence $\ker g_0 = 0$, so that $g_0: H \to G$ is an isomorphism.  

A similar argument shows that the inclusion-projection sequence $[pt/A] \to E \to G$ is always an extension. The automorphism 2-group of the 2-group $[pt/A]$ is equivalent to the group $Aut(A)$ viewed as a 2-group with only identity morphisms. Following the previous construction, we see that the induced map $G \to Aut( [pt/A]) = Aut(A)$ is precisely the action $\rho$, and thus the extension is central precisely when $\rho$ is trivial.
\end{proof}

%\begin{remark}
	The above notion of central extension of smooth 2-group is more general then the notion introduced in \cite{Wockel08}. In particular it is invariant under equivalence of smooth 2-group, and includes the following examples not covered by Wockel's treatment. 
%\end{remark}

\begin{example}
	Let  $\Gamma = (G, A, a)$ be the 2-group from Example \ref{ExampleCocycle2Grp}. There exists a canonical central homomorphism (inclusion) $i: [pt/A] \to \Gamma$ and a homomorphism  (projection) $\pi: \Gamma \to G$. The composite is equal to the zero homomorphism $[pt/A] \to G$. This has a unique automorphism $\phi$, and with this choice the triple $(i, \pi, \phi)$ is a central extension. $\Gamma$ is a trivial principal bundle over $G$ in the sense that it is equivalent to $G \times [pt/A]$ as a $[pt/A]$-stack over $G$. 
\end{example}

\begin{example} \label{Example:FavExtension}
Let $A$ be an abelian Lie group. There is a unique homomorphism from $A$ to the abelian 2-group $[pt/A]$. This homomorphism factors as the composite
\begin{equation*}
A \stackrel{f}{\to} 0 \stackrel{g}{\to} [pt/A]
\end{equation*}
The automorphisms of $0 = gf : [pt/A] \to A$ are in canonical bijection with $\phi \in \hom(A, A)$. The triple $(f,g, \phi)$ is a central extension precisely when $\phi$ is an automorphism.  
\end{example}

\subsection{Smooth 2-Groups and $A_\infty$-Spaces}

Given a smooth 2-group $\Gamma$ we may obtain a space by taking the geometric realization $|\Gamma|$, see \cite{Segal68}. In the trivial case, when $\Gamma = G$, the resulting space is $|G| \cong G$ and hence is a topological group. This is too much to expect in general, and indeed the geometric realization functor, viewed as an assignment in the category of space, is not precisely functorial with respect to bibndles. It does however lead to a functor $\bibun \to $h-$\Top$ with values in the {\em homotopy category} of spaces, as we shall see.  

In particular any smooth 2-group gives rise to a group-like H-space, and this assignment is functorial.  In this section we will show that this can be improved upon to give an infinitely coherent multiplication ($A_\infty$-structure) on the geometric realization of any smooth 2-group. Everything in this section holds equally well in the topological setting, provided the source and target maps admit local sections. %and our spaces are locally contractible.  %%%%%%
% We really just need geometric realization to be stable under fiber products. 
% Theorem: Finite colimits in Top are stable under pull-back. 
% Here is a short proof:
% In set finite colimits distribute over fiber products (Set is a topos). And the forgetful functor preserves both limits and colimits since it is both a left and right adjoint (this fails in Hausdorff Topological spaces). So we are just comparing two topologies on the same set. 
% Coproducts clearly commute with fiber products in Top. 
% So it is enough to show that equalizers commute with fiber products. But equalizers are certain quotients and it is easy to check that the two topologies agree. In particular

%% Hmmm seems to be a problem. Come back to this....

%% Don't we also need local contractibility to get that E{(t)}G is a bundle over |G|? cite Segal for why we need local contractibility...

Given a Lie groupoid $\Gamma = (\Gamma_1 \rightrightarrows \Gamma_0)$ we may construct an associated groupoid. The target $t: \Gamma_1 \to \Gamma_0$ is a surjective submersion so we may construct the corresponding \v{C}ech groupoid $E^{(t)}\Gamma = (\Gamma_1 \times_{\Gamma_0}^{t,t} \Gamma_1 \rightrightarrows \Gamma_1)$, see Example \ref{ExampleCechGroupoids}. As always with \v{C}ech groupoids, there is a functor to the space $\Gamma_0$, viewed as a Lie groupoid. In this case that functor is given by the target map $t: E^{(t)}\Gamma \to \Gamma_0$. This map is an equivalence of Lie groupoids: the identity map $\iota: \Gamma_0 \to \Gamma_1$ induces a functor the other direction which is an inverse to $t$. This equivalence can be taken to be over the Lie groupoid $\Gamma_0$.

Moreover, there is a functor $\sigma: E^{(t)}\Gamma \to \Gamma$ which on objects is $s: \Gamma_1 \to \Gamma_0$, and on morphisms is given by $(f,g) \mapsto g \circ f^{-1}$. Passing to the nerve, we see on each level that we have a space $E^{(t)}\Gamma_n$, which consists of $n$-tuples of morphisms of $\Gamma$, all with the same target. There is an action by $\Gamma$ in the sense that post-composition gives a map,
\begin{equation*}
	\Gamma_1 \times_{\Gamma_0}^{s, t} E^{(t)}\Gamma_n \to E^{(t)}\Gamma_n
\end{equation*}
This action map is over $\Gamma_0 \times \Gamma_n$, and $E^{(t)}\Gamma_n$ becomes a bibundle from the space $\Gamma_n$ to $\Gamma$. 

Geometric realization of these simplicial spaces is stable under fiber products \cite[Cor. 11.6]{May72} (See also \cite{Lewis78} and \cite{MO34805})
and thus upon geometric realization we find that $|E^{(t)}\Gamma|$ is a (left principal) bibundle from the classifying space $|\Gamma|$ to the groupoid $\Gamma$, now viewed in the topological category rather then the smooth category\footnote{It is not clear from our description that $|E^{(t)}\Gamma|$ will be locally trivial over $|\Gamma|$, i.e. admit local sections. Indeed this fails for general topological groupoids. However, in the case that the spaces involved are locally contractible, local triviality follows from an argument identical to the proof of \cite[Prop. A.1]{Segal70}. Since local triviality is not used in our argument we will omit these details.}.  
 The bibundle $|E^{(t)}\Gamma|$ is the analog of the classifying bundle of a group: when $\Gamma = (G \rightrightarrows pt)$ is a Lie group, $|E^{(t)}G| = |EG|$ is exactly the classifying bundle in the usual sense. 

Moreover, the equivalence between $E^{(t)}\Gamma$ and $\Gamma_0$ induces a homotopy equivalence between $|E^{(t)} \Gamma|$ and  $\Gamma_0$ in the category of spaces over $\Gamma_0$, i.e. $|E^{(t)}\Gamma|$ is homotopy equivalent to the terminal object of $\Top_{\Gamma_0}$. This is the appropriate analog of {\em contractible} in the relative category $\Top_{\Gamma_0}$. When $\Gamma = G$ is a group, then $\Gamma_0= pt$ and hence $EG$ is contractible in the usual sense. Starting with the source map $s: \Gamma_1 \to \Gamma_0$ yields an analogous story, but the outcome is a {\em right} principal bibundle $|E^{(s)}\Gamma|$ from $\Gamma$ to $|\Gamma|$. Again $|E^{(s)}\Gamma| \simeq \Gamma_0$ as spaces over $\Gamma_0$. The inversion isomorphism allows us to canonically identify $\overline{|E^{(t)}\Gamma|} \cong |E^{(s)}\Gamma| $, thus we obtain an isomorphism of spaces $|E^{(t)}\Gamma| \times_{|\Gamma|} |E^{(s)}\Gamma| {\cong} \Gamma_1 \times_{\Gamma_0} |E^{(t)}\Gamma|$. Projection gives rise to a map, 
\begin{equation} \label{EqnCompose}
	|E^{(t)}\Gamma| \times_{|\Gamma|} |E^{(s)}\Gamma| {\to} \Gamma_1 , 
\end{equation}
which commutes with both the left and right $\Gamma$-actions.

\begin{lemma}
	Let $P$ be a (left principal) bibundle from the space $X$ to the Lie groupoid $\Gamma$. Then the composition $|E^{(s)}\Gamma| \circ P$ with the {\em right} principal bundle $|E^{(s)}\Gamma|$ is a space over $X$ homotopy equivalent to $X$ over $X$. In particular the space of sections is a contractible space. 
\end{lemma}

\begin{proof}
	This is a local statement and so it is enough to consider the case when we have a map $f: X \to \Gamma_0$ and $P \cong \Gamma_1 \times_{\Gamma_0}^{s, f} X$. The composition with $|E^{(s)}\Gamma|$ then becomes $|E^{(s)}\Gamma| \circ P \cong |E^{(s)}\Gamma| \times_{\Gamma_0} X$. But since $|E^{(s)}\Gamma| \simeq \Gamma_0$ over $\Gamma_0$ the space $P \circ |E^{(s)}\Gamma|$ is homotopy equivalent to $X$ over $X$. 
\end{proof}

\begin{corollary}\label{Fin}
 A bibundle $P$ from the Lie groupoid $G$ to the Lie groupoid $\Gamma$ gives rise to a contractible family of morphisms from $|G|$ to $|\Gamma|$ (described in the proof below). This association is compatible with composition, and hence yields a functor $\bibun \to$ h-$\Top$.
\end{corollary}

\begin{proof}
Consider the following chain of bibundles:
	\begin{center}
\begin{tikzpicture}
\node (Z) at (-4,0) {$|\Gamma|$};

\node (A) at (0,1) {$\Gamma_1$};
\node (B) at (0,0) { $\Gamma_0$};
\draw [->] (A.255) -- (B.105);
\draw [->] (A.285) -- (B.75);

\node (Y) at (-2,1) {$|E^{(s)} \Gamma|$};
\draw [->>] (Y) -- node [above left] {$$} (Z);
\draw [->] (Y) -- node [above right] {$$}(B);

\node (C) at (4,1) {$G_1$};
\node (D) at (4,0) { $G_0$};
\draw [->] (C.255) -- (D.105);
\draw [->] (C.285) -- (D.75);

\node (E) at (2,1) {$P$};
\draw [->>] (E) -- node [above left] {$$} (B);
\draw [->>] (E) -- node [above right] {$$}(D);

\node (K) at (8,0) {$|G|$};

\node (J) at (6,1) {$|E^{(t)} G|$};
\draw [->>] (J) -- node [above left] {$$} (K);
\draw [->] (J) -- node [above right] {$$}(D);
\end{tikzpicture}.
\end{center}
Note that $P$ and $|E^{(t)}G|$ are left-principal bibundles, while $|E^{(s)} \Gamma|$ is a right-principal bibundle. 
Composing these three yields a space $K = |E^{(s)} \Gamma| \circ P \circ |E^{(t)} G|$, with two maps, one to $|\Gamma|$ and one to $|G|$. Moreover, by the previous lemma, $K$ is homotopy equivalent to $|G|$ over $|G|$. Hence the space $\cS$ of sections of $K$ over $|G|$ is contractible. Composing a section $|G| \to K$ with the projection $K \to |\Gamma|$ induce the desired family, $\cS \to \maps(|G|, |\Gamma|)$.  

For a composable pair of bibundles $P: G \to \Gamma$, and $Q: H \to G$, a pair of sections in $\cS_P$ and $\cS_Q$ gives rise to a section over $|H|$ of the space
\begin{equation*}
	|E^{(s)}\Gamma| \circ P \circ |E^{(t)G}| \times_{|G|} |E^{(s)}G| \circ Q \circ |E^{(t)H}|.
\end{equation*}
From this, we get an section over $|H|$ of $|E^{(s)}\Gamma| \circ P \circ Q \circ |E^{(t)H}|$, and hence a map $\cS_P \times \cS_Q \to \cS_{P \circ Q}$, by composing with the map of spaces
\begin{equation*}
	|E^{(t)}G| \times_{|G|} |E^{(s)}G| \to G_1
\end{equation*}
described in Equation \ref{EqnCompose}. The compatibility of this map with composition follows from the bi-equivariance of the map in Equation \ref{EqnCompose}.
 %%%%%%
%Both of these maps are homotopy equivalences by the previous proposition.
\end{proof}

\begin{remark}
A more sophisticated approach is to consider the bicategory $\bibun$ as an $(\infty, 1)$-category. Then the above corollaries and proposition may be summarized by saying that there is an  $\infty$-functor from $\bibun$ to the $\infty$-category of topological spaces. 
\end{remark}

\begin{corollary}
	The geometric realizations of Morita equivalent Lie groupoids are homotopy equivalent.
\end{corollary}

\begin{proof}
	In the setting of the previous proof, a bibundle $P$ between Lie groupoids gives rise to a space $K$ with maps to $|G|$ and $|\Gamma|$. In the case that $P$ is a Morita equivalence, both these maps are homotopy equivalences. 
\end{proof}

\begin{definition}
 A {\em topological operad} consists of a collection of spaces $S_n$ for each $n \geq 0$, together with composition maps:
 \begin{equation*}
		S_n \times S_{i_1} \times \cdots S_{i_n} \to S_{i_1 + \cdots + i_n}
\end{equation*}
 which are associative in the obvious way. An {\em algebra} for a topological operad $S$ is a space $X$ together with actions maps,
  \begin{equation*}
		S_n \times \underbrace{X \times \cdots X}_{n \text{ times}}\to X
\end{equation*}
which again are associative and compatible with the maps from $S$, in the obvious way.
\end{definition}

\begin{definition}[\cite{May72}]
	An {\em $A_\infty$-operad} is any topological operad with contractible spaces. An {\em $A_\infty$-space} is a space $X$ which is an algebra for an $A_\infty$-operad. 
\end{definition}

\begin{theorem}
 The geometric realization of a smooth 2-group in $\bibun$ is naturally an $A_\infty$-space. 
\end{theorem}

\begin{proof} We must construct an $A_\infty$-operad and an action of this operad on $|\Gamma|$. Consider the composition of  bibundles, $|E^{(s)}\Gamma| \circ m \circ (|E^{(t)}\Gamma| \times |E^{(t)}\Gamma|)$. This is a space with a map to $| \Gamma|$ coming from $|E^{(s)}\Gamma|$ and a map to $|\Gamma| \times |\Gamma|$ coming from $ |E^{(t)}\Gamma| \times |E^{(t)}\Gamma|$. Thus if we choose a section over $|\Gamma| \times |\Gamma|$ we get a map of spaces. Let $S_2$ denote the space of sections over $|\Gamma| \times |\Gamma|$. Putting these maps together gives us a map,
\begin{equation*}
	S_2 \times |\Gamma| \times |\Gamma| \to |\Gamma|
\end{equation*}
which is continuous. Recall, however, that the space $S_2$ is contractible. Thus the contractible space $S_2$ parametrizes several multiplications for the space $|\Gamma|$. We mimic this and define contractible spaces of sections $S_n$ for all $n$. For $n \geq 2$ $S_n$ is the space of sections (over $|\Gamma|^n$) of:
\begin{equation*}
|E^{(s)}\Gamma| \circ \underbrace{m \circ (m \times 1) \circ (m \times 1 \times 1) \circ \cdots   }_{n \text{ times}} \circ
\underbrace{(|E^{(t)}\Gamma| \times |E^{(t)}\Gamma| \times \cdots \times  |E^{(t)}\Gamma|    )}_{n \text{ times}}
\end{equation*}
This is again a contractible space. We set $S_0 = S_1 = pt$. Since we started with a Lie 2-group we have a specified isomorphism of bibundles $m \circ (m \times 1) \cong m \circ (1 \times m)$, given by the associator. Mac Lane's coherence theorem ensures us that this extends to a canonical isomorphism between any two possible bracketings. For example the composition
\begin{equation*}
		m \circ (m \times 1) \circ (m \times m \times m) \circ ( 1 \times 1 \times m \times 1 \times m \times 1)
\end{equation*}
is canonically isomorphic to the composition, 
\begin{equation*}
\underbrace{ m \circ  (m \times 1) \circ (m \times 1 \times 1) \circ \cdots }_{ 7 \text{ times}}.
\end{equation*}
We turn the collection of spaces $S_n$ into an $A_\infty$-operad as follows. A point in the space, $S_{i_1} \times \cdots S_{i_n}$ is a section of the corresponding product of bibundles (over $|\Gamma|^{( i_1 + \cdots + i_n)}$). These bundles project to $|\Gamma|^n$ and so we get a map,
\begin{equation*}
|\Gamma|^{( i_1 + \cdots + i_n)} \to |\Gamma|^n
\end{equation*}
A point in $S_n$ then gives us a section (over $|\Gamma|^n$) of its corresponding bundle. When we compose these bundles, we get a corresponding composition of sections. This is a section of a certain bundle over $|\Gamma|^{( i_1 + \cdots + i_n)} $, similar in construction to $S_n$, but with a different bracketing. The canonical identification from the associator allows us to identity this with a point of $S_{i_1 + \cdots + i_n}$. Hence we have assembled maps,
\begin{equation*}
	S_n \times S_{i_1} \times \cdots S_{i_n} \to S_{i_1 + \cdots + i_n}.
\end{equation*}
It can readily be checked that this is an operad (the compositions involving $S_1$ and $S_0$ are similar, where $S_1$ corresponds to sections of the identity bibundle $|\Gamma| \to |\Gamma$ and $S_0$ to sections of the unit bibundle $\iota: 1 \to |\Gamma|$). Moreover since the spaces are contractible, this is an $A_\infty$-operad. 

We have also seen how $|\Gamma|$ is naturally an algebra for this operad. $S_n$ is a space of sections of a bundle over $|\Gamma|^n$ and this bundle has a map to $|\Gamma|$, hence there is an induced action map,
\begin{equation*}
	S_n \times | \Gamma|^n \to |\Gamma|
\end{equation*}
which makes $|\Gamma|$ into an $A_\infty$-space.
\end{proof}

\begin{remark}
With more work one sees that a homomorphism of smooth 2-groups yields a morphism of $A_\infty$-spaces. 
\end{remark}

\section{A Finite Dimensional String 2-Group}

In this section we prove a theorem which interprets Segal-Mitchison topological group cohomology in terms of certain central extensions of smooth 2-groups. The model of the String group presented in this paper is a special case of such an extension.

\subsection{Segal's Topological Group Cohomology}

In \cite{Segal70, Segal75} G. Segal introduced a version of cohomology for (locally contractible) topological groups, which mimics the derived functor definition of ordinary group cohomology. A few years later, Quillen introduced the notion of {\em exact category}  in his work on algebraic K-theory \cite{Quillen72}. Roughly speaking, an exact category is an additive category equipped with a distinguished class of {\em short exact sequences}. Such a category is not required to be an abelian category, and there are many examples, among them the category of topological groups considered by Segal. It is now realized that essentially all the constructions and machinery of homological algebra carry over to the setting of exact categories, see \cite{Buhler08} for a fairly comprehensive introduction and overview. In particular resolutions and derived functors can often be defined in this setting and Segal's cohomology is an example. 

Segal's cohomology was rediscovered by Brylinski in \cite{Brylinski00} in the smooth setting. This group cohomology solves many of the defects of the naive \textquotedblleft group cohomology with continuous/smooth cochains\textquotedblright, and certain cocycle representatives will serve as our basic input in constructing the String$(n)$ 2-group. Let us summarize some of the special features of this cohomology theory. Proofs of these facts can be found in \cite{Segal70, Segal75}. If $G$ is a topological group and $A$ is a topological $G$-module\footnote{As mentioned in the introduction, and action of a topological group $G$ on a topological abelian group $A$ is an action in the usual sense such that the map
\begin{equation*}
 G \times A \to A
\end{equation*}
is continuous, where $G \times A$ is given the compactly generated topology.
} then we can form the {\em Segal-Mitchison group cohomology} $H_\text{SM}^n(G; A)$.  
\begin{enumerate}
\item In low dimensions, $q = 0,1,2$, $H^q_\text{SM}(G; A)$ may be interpreted in the usual manner.
\begin{enumerate}
\item  $H^0_\text{SM}(G; A) = A^G$, the $G$-invariant subgroup,
\item  $H^1_\text{SM}(G; A)$ is the group of continuous crossed homomorphisms $G \to A$, modulo the principal crossed homomorphisms, and
\item $H^2_\text{SM}(G; A)$ is the group of isomorphism classes of group extensions, \begin{equation*}
A \to E \to G
\end{equation*}
inducing the action of $G$ on $A$, where $E \to G$ is topologically a locally trivial fibration, i.e. a fiber bundle. 
\end{enumerate}
\item If $A$ is contractible, then the Segal-Mitchison cohomology coincides with the continuous group cohomology.
\item If $A$ is discrete, then the Segal-Mitchison cohomology is isomorphic to the twisted cohomology of the space $BG$ with coefficients in $A$.\footnote{If $A$ is discrete then the action of $G$ factors as $G \to \pi_0G \to Aut(A)$ and since $\pi_1 BG \cong \pi_0 G$, we have a canonical locally constant sheaf over $BG$. This can also be obtained as the sheaf associated to the fiber bundle $EG \times_G A \to BG$. This sheaf is used to define the twisted cohomology of $BG$.} In particular if the $G$ action is trivial we have $H_\text{SM}^n(G; A) \cong H^n(BG; A)$, the ordinary cohomology of the space $BG$ with coefficients in $A$. 

\item A sequence of topological $G$-modules  $A' \to A \to A''$ is a {\em short exact sequence} if it is a short exact sequence of underlying abelian groups and the action of $A'$ on $A$ realizes $A$ as an $A'$-principal bundle over $A''$. If $A' \to A \to A''$ is such a short exact sequence then there is a long exact sequence of cohomology groups,
\begin{align*}
0 & \to H^0_\text{SM}(G, A') \to  H^0_\text{SM}(G, A) \to  H^0_\text{SM}(G, A'') \to \\ 
	&\to  H^1_\text{SM}(G, A') \to H^1_\text{SM}(G, A) \to \cdots
\end{align*}

\item If $G$ is a topological group and $A$ is a topological $G$-module, then $A$ determines a simplicial sheaf $\cO_A$ on the simplicial space $BG_\bullet$. When the action of $G$ on $A$ is trivial, then $\cO_A^n$ is simply the sheaf of continuous functions with values in $A$. In general we have \ $H^q_\text{SM}(G; A) \cong H^q( BG_\bullet ; \cO_A)$,  where $BG_\bullet$ is the simplicial nerve of $G$, $\cO_A$ is the simplicial sheaf corresponding to $A$, and this latter group denotes the hypercohomology (see \cite{Friedlander82} for details about simplicial hypercohomology). 

\end{enumerate}

The category of (locally contractible) topological abelian groups becomes an exact category with the short exact sequences introduced above. Segal's cohomology is then defined to be the derived functor of the invariant subgroup functor,
\begin{equation*}
	\Gamma^G: A \mapsto A^G.
\end{equation*}
In \cite{Segal70} Segal proves that this functor is derivable by demonstrating a class of objects adapted to this functor (his so-called ``soft'' modules). He also proves this cohomology has  the above properties.

In the finite dimensional smooth setting, there is an analogous exact structure. More precisely fix a Lie group $G$ and consider the category of abelian Lie groups equipped with smooth actions of $G$. A sequence of such $G$-modules  $A' \to A \to A''$ will be called a {\em short exact sequence} if it is a short exact sequence of underlying abelian groups and the action of $A'$ on $A$ realizes $A$ as an $A'$-principal bundle over $A''$. Unfortunately this category will not contain enough adapted objects in order to derive the invariant subgroup functor. 

This can be overcome by embedding abelian Lie groups into a lager category of ``smooth'' abelian group objects. For example abelian group objects in one of the ``convenient categories of smooth spaces'' discussed in \cite{BH08} provide such an enhancement. Alternatively one could use the sheaf cohomology of the resulting simplicial sheaf on $BG_\bullet$, see \cite{Friedlander82} for the relevant definitions. Both of these approaches result in the same cohomology theory which Brylinski \cite{Brylinski00} shows may be computed as the cohomology of the total complex of a certain double complex, which we now describe. 

%Let $\Delta$ be a skeleton of the category of finite non-empty ordered sets and order preserving maps. A simplicial object in a category $\sC$ is a functor $\Delta^\text{op} \to \sC$, \cite{Segal68}. We will be interested in the case $\$  
%A {\em semi-simplicial space} is a functor $\Delta^\text{op}_\text{semi} \to \Top$, where $\Delta_\text{semi} \subseteq \Delta$ has the same objects but with morphisms which are {\em strictly} order preserving maps. Every simplicial space gives rise to a semi-simplicial space via pull-back of functors. 

\begin{definition}
	A {\em simplicial cover} (or just {\em cover}) of a simplicial manifold $X_\bullet$ is a simplicial manifold $U_\bullet$ and a map $U_\bullet \to X_\bullet$, such that each $U_n \to X_n$ is a surjective submersion. A cover is {\em good} if each of the spaces 
	\begin{equation*}
		U_n^{[p]} = \underbrace{U_n \times_{X_n} \cdots \times_{X_n} U_n}_{ p \text{ times}}
\end{equation*}
	is the union of paracompact contractible spaces, where $p,n \geq 0$. 
%Given a semi-simplicial cover $U_\bullet \to X_\bullet$, for each $d_i$ we may form the pullback cover,
%	\begin{center}
%\begin{tikzpicture}[thick]
%	\node (LT) at (0,1.5) 	{$d^*_i U_{n-1}$ };
%	\node (LB) at (0,0) 	{$X_n$};
%	\node (RT) at (2,1.5) 	{$U_{n-1}$};
%	\node (RB) at (2,0)	{$X_{n-1}$};
%	\node at (.35, 1) {$\ulcorner$};
%	\draw [->] (LT) --  node [left] {$$} (LB);
%	\draw [->] (LT) -- node [above] {$$} (RT);
%	\draw [->] (RT) -- node [right] {$$} (RB);
%	\draw [->] (LB) -- node [below] {$d_i$} (RB);
%\end{tikzpicture}
%\end{center}
% and their mutual fiber product also forms a cover of $X_n$,
% \begin{equation*}
%	d_0^* U_{n-1} \times_{X_n} d_1^* U_{n-1} \times_{X_n} \cdots \times_{X_n} d_n^*U_{n-1} \to X_n.
%\end{equation*}
%We say a semi-simplicial covering is {\em self-covering} if the canonical map,
%\begin{equation*}
%	U_n \to d_0^* U_{n-1} \times_{X_n} d_1^* U_{n-1} \times_{X_n} \cdots \times_{X_n} d_n^*U_{n-1}
%\end{equation*}
%is a covering map (and in particular a surjective submersion).
\end{definition}

Consider the simplicial manifold $BG_\bullet$. In \cite{Brylinski00} Brylinski provides an inductive construction of a good simplicial cover of $BG_\bullet$\footnote{For general simplicial spaces it is not possible to construct {\em good} simplicial covers. In that case one must instead use good hypercovers. Brylinski's construction allows use to avoid this subtlety entirely.}. 
% Does he? or is it merely semi-simplicial.
 For compact $G$, using techniques developed in \cite{Meinrenken03} we may construct a canonical such cover. For our purposes, however, any good simplicial cover will do and so we will not dwell on this aspect. For $A \in \Top\ab_G$, we get an induced double complex, where 
	\begin{equation*}
		{C}^{p,q} = C^\infty(U_q^{[p+1]}, A)
\end{equation*}
and the differentials are induced by the two simplicial directions.
The cohomology of the total complex computes the smooth version of Segal's group cohomology. Let us fix some notation. 
Let $G$ be a Lie group and $A$ an abelian Lie group with a $G$-action.
\begin{itemize}
\item $H^k_{SM}(G;A)$ is the smooth version of Segal's cohomology which we take to be the total cohomology of the double complex ${C}^{p,q} = C^\infty(U_q^{[p+1]}, A)$ computed from a good simplicial cover of $BG_\bullet$. 
\item $H^k_\text{smooth}(G;A)$ denotes the cohomology of $G$ computed with smooth group cocycles. 
\item We will primarily be interested in the case where the action of $G$ on $A$ is trivial. In this case $\check{H}^k(G;A)$ is the \v{C}ech cohomology of the space $G$ with coefficients in  the sheaf of smooth functions with values in $A$. 
\end{itemize}

\begin{corollary}\label{CorCptG}
If $G$ is a compact Lie group and $A = S^1$ then we have the following isomorphism of smooth Segal-Mitchison cohomology
\begin{equation*}
	H^i_{SM}(G; S^1) \cong H^{i+1}_{SM}(G; \Z) \cong H^{i+1}(BG)
\end{equation*}
for all $i \geq 1$, where $H^k(BG)$ is integral cohomology of the space $BG$, and moreover in low degrees we have an exact sequence
\begin{equation*}
	0 \to H^0_{SM}(G; \Z) \to H^0_{SM}(G; \R) \to H^0_{SM}(G; S^1) \to H^1_{SM}(G; \Z) \to 0
\end{equation*}
where $S^1$, $\R$, and $\Z$ are considered $G$-modules with trivial action.
\end{corollary}
\begin{proof}
	The short exact sequence of Lie groups $\Z \to \R \to S^1$ induces a long exact sequence in Segal cohomology. However since $\R$ is contractible, Segal cohomology agrees with cohomology computed with smooth cochains. Since $G$ is compact, these vanish in degrees larger then zero. 
\end{proof}

\subsection{Classifying Extensions of Smooth 2-Groups}

The goal of this section is to prove Theorem \ref{Thm:MainThm}, which classifies the bicategory of central extensions of certain smooth 2-groups. We begin with some elementary results on symmetric monoidal bicategories. Results on general symmetric monoidal bicategories may be found in  \cite{GPS95, KV94, KV94-2, BN96, DS97} and  \cite[Chap. 3]{SchommerPries09}. 

The simplest kinds of symmetric monoidal bicategories arise from 3-term cochain complexes of abelian groups. Let $C^3 \stackrel{d}{\leftarrow} C^2 \stackrel{d}{\leftarrow} C^1$ be a 3-term cochain complex of abelian groups. We may form a strict bicategory $D$ as follows. The objects $D_0$ consist of the elements of the group $C^3$. The 1-morphisms consist of the product $D_1 = C^3 \times C^2$. The source is the projection to $C^3$, the target map is given by $t(c_3, c_2) = c_3 + d(c_2)$, and the strict horizontal composition is given by $(c_3, c_2) \circ (c_3', c_2') = (c_3, c_2 + c_2')$, when $c_3' = c_3 + d(c_2)$. Similarly, the 2-morphisms consist of $C^3 \times C^2 \times C^1$ with source map the projection, target map given by $t(c_3, c_2, c_1) = (c_3, c_2 + d(c_1))$, and vertical composition of composable elements given by addition of the $C^1$ terms. The horizontal composition of composable 2-morphisms is given by $(c_3, c_2, c_1) * (c_3', c_2', c_1') = (c_3, c_2 + c_2', c_1 + c_1')$, which is again a strict operation. 

The bicategory $D$ comes equipped with the structure of a strict symmetric monoidal bicategory. The monoidal structure is induced by the abelian group multiplication in $C^i$, $i = 1,2,3$, and the braiding is trivial. In this way we obtain a number of examples of elementary symmetric monoidal bicategories. For example, an abelian group $M$ may be regarded as a cochain complex concentrated in a single degree. There are three possibilities for 3-term cochain complexes arising in this manner, and hence we obtain three symmetric monoidal bicategories, $M$, $M[1]$, and $M[2]$. The notation $M[2]$ denotes the symmetric monoidal 2-category with one object, one 1-morphism, and $M$ many 2-morphisms, whose compositions are induced from multiplication in M. Similarly $M[1]$  denotes the symmetric monoidal bicategory with one object, $M$ many 1-morphisms, and only identity 2-morphisms, and $M$ without decoration denotes the symmetric monoidal bicategory with objects $M$ and  only identity 1-morphisms and 2-morphisms. The following lemma is presumably well known to experts. 

\begin{lemma} \label{Lma:CatFromChainComplexSplits}
	Let $(C^i, d)$ be a 3-term cochain complex and let $D$ be the resulting symmetric monoidal bicategory. Let $H^* = H^*(C^i, d)$ be the cohomology groups of $(C^i, d)$. Then there is an equivalence,
	\begin{equation*}
		D \simeq H^3 \times H^2 [1] \times H^1 [2],
	\end{equation*} 
of symmetric monoidal bicategories. In general this equivalence is unnatural, but nonetheless there is a natural isomorphism $\pi_0(D) \cong H^3$. 
\end{lemma}

\begin{proof}
	This can be proven in several ways. A global approach is to analyze the k-invariants\footnote{The 1- and 2-morphisms in the bicategroy $D$ arising from a 3-term chain complex are invertible and hence $D$ is a 2-groupoid. Moreover, $\pi_0D$ is a group and so $D$ is a 3-group. Thus its k-invariants are  well understood and coincide with the classical k-invariants of a stable homotopy 2-type.} of Picard symmetric monoidal bicategories as was done for Picard symmetric monoidal categories in \cite[Appendix B.2]{HS05} and for certain braided monoidal categories in \cite{JS93}. Since $D$ is both {\em strict} and {\em symmetric monoidal}, these k-invariants vanish. Thus $D$ splits up to unnatural equivalence as the product of its `homotopy groups'. Notice, however that there is natural map of cochain complexes from $(C^i, d)$ to the complex with a single non-zero group $H^3$ in the top term. This is an isomorphism on third cohomology groups and induces the natural isomorphism $\pi_0(D) \cong H^3$. 
	
	Alternatively, one may simply choose a skeleton of $D$, as in \cite[Lemma 3.4.5 - 3.4.7]{SchommerPries09}. A direct calculation, following the proofs of these lemmas, shows that $D$ splits as in the statement of Lemma \ref{Lma:CatFromChainComplexSplits}. Producing such a splitting usually requires choices.
\end{proof}

Let $G$ be a Lie group and $A$ an abelian Lie group. Let $Z_{SM}(G; A)$ denote the 3-term chain complex
\begin{equation*}
	Z_{SM}^3(G; A) \stackrel{d}{\leftarrow} C^2_{SM}(G; A) \stackrel{d}{\leftarrow} C^1_{SM}(G; A)
\end{equation*}
given by the smooth Segal-Mitchison cohomology of $G$ with values in the trivial $G$-module $A$. By abuse of notation, let $Z_{SM}(G; A)$ also denote the corresponding symmetric monoidal bicategory.

\begin{theorem} \label{ThmMain}
	Let $G$ and $A$ be as above. There is a natural equivalence of symmetric monoidal bicategories $Z_{SM}(G; A) \stackrel{\simeq}{\to} Ext(G; [pt/A])$. Thus, we have a (generally unnatural) equivalence,
	\begin{equation*}
		Ext(G; [pt/A]) \simeq  H^3_\text{SM}( G; A) \times H^2_\text{SM}( G; A)[1] \times H^1_\text{SM}( G; A)[2].
	\end{equation*}
	where $ H^i_\text{SM}( G; A)$ denotes the smooth version Segal-Mitchison topological group cohomology \cite{Segal70}. In particular isomorphism classes of central extensions, 
		\begin{center}
	\begin{tikzpicture}[thick]
		\node (LL) at (0,0) 	{$1$ };

		\node (L) at (2,0) 	{$\downdownarrows$};

		\node (LA) at (2,.5) 	{$A$};
		\node (LB) at (2,-.5) 	{$pt$};
		\node (M) at (4, 0) {$\downdownarrows$};
		\node (MA) at (4,.5) 	{$\Gamma_1$};
		\node (MB) at (4,-.5) 	{$\Gamma_0$};

		\node (R) at (6,0)	{$\downdownarrows$};
		\node (RA) at (6,.5)	{$G$};
		\node (RB) at (6, -.5)	{$G$};

		\node (RR) at (8,0)	{$1$};

		\draw [->] (LL) --  node [left] {$$} (L);
		\draw [->] (L) -- node [above] {$$} (M);
		\draw [->] (M) -- node [right] {$$} (R);
		\draw [->] (R) -- node [below] {$$} (RR);
	\end{tikzpicture}
	\end{center}
	are in natural bijection with $H^3_\text{SM}(G;A)$.
\end{theorem}

\begin{proof}
	The bicategory $Z_{SM}(G; A)$ is covariantly functorial in $G$, contravariantly functorial in $A$, and preserves products. Moreover, just as for $Ext(G;[pt/A])$, the symmetric monoidal structure is induced from the Baer sum. Thus it suffices to produce a natural equivalence of bicategories $Z_{SM}(G; A) \to Ext(G; [pt/A])$. It will automatically be an equivalence of {\em symmetric monoidal} bicategories.  See also \cite[Theorem 3.4.10]{SchommerPries09}.	The remaining statements in the theorem follow from this equivalence and Lemma \ref{Lma:CatFromChainComplexSplits}.

Before getting into the details, which are somewhat computational, let us explain the philosophy behind why this theorem is true. This result is the offspring of two well established ideas. On the one-hand, following \cite{BL04}, there is a direct relationship between smooth functors between Lie groupoids and between certain Lie groupoid cocycles. This link extends to the level of smooth natural transformations, as well. If the multiplication in a smooth 2-group was given by a smooth functor, then we would be able to translate the axioms it must satisfy into certain concrete statements about cocyle data and be able to classify central extensions in terms of this data. See \cite[Theorem 55]{BL04} for an example of a result along these lines. 

On the other hand the multiplication in a smooth 2-group is a bibundle and these come from functors precisely when there exists a global section of the bibundle over the source object space \cite{Lerman08} (see also Proposition \ref{PropSectionIsBundlization}). However, every bibundle admits sections {\em locally} in the sense that for every bibundle $P$ from $G$ to $H$, there exists a cover $f: U \to G_0$, such that the composition of $P$ with the canonical bibundle from $f^*G$ to $G$ admits a global section, see Example \ref{ExamplePullbackGroupoid}.

%We know from the work in \cite{BL04} that there is a direct relationship between Lie groupoid cocycles and the corresponding smooth functors and natural transformations. Moreover, we know from \cite{Lerman08} (see Proposition \ref{PropSectionIsBundlization}) that a bibundle comes from a functor precisely when there exists a section of the bibundle over the source object space. This suggest a strategy for relating structures in $\bibun$ to cocycle data. Every bibundle admits sections {\em locally} in the sense that for every bibundle $P$ from $G$ to $H$, there exists a cover $f: U \to G_0$, such that the composition of $P$ with the canonical bibundle from $f^*G$ to $G$ admit a global section, see Example \ref{ExamplePullbackGroupoid}. 

So while the multiplication bibundle in a 2-group, which is a bibundle $m$ from $G \times G$ to $G$, may not admit global sections, we may choose a cover $f: U_2 \to G_0 \times G_0$ such that the pull-back of $m$ to $f^*(G \times G)$ {\em does} admit global sections. Hence this induced bibundle comes from a functor, and may be described by appropriate classical cocycle data. The associator will have a similar description via cocycle data on the pull-back of $G \times G \times G$ to an appropriately chosen cover of $G_0 \times G_0 \times G_0$. 
In this way we may extract from a smooth 2-group precisely the cocycle data of a smooth Segal-Mitchison cocycle. Conversely, given such cocycle data we may push it forward to bibundle data via the equivalences between, say, $G \times G$ and its pull-back along $U_2 \to G_0 \times G_0$.

We now proceed to prove Theorem \ref{ThmMain}. Let us now note that there is a slight ambiguity in the definition of the cochain complex $Z_{SM}(G;A)$. In defining the cochain complex computing Segal-Mitchison cohomology, we were free to use any good simplicial covering. The resulting cohomology is independent of this choice. However the cochain complex itself clearly depends upon this choice. However, in the course of proving Theorem \ref{ThmMain}, we will show that these choices are irrelevant. More precisely, we will first fix a simplical cover $U$ and construct a functor $Z_{SM, U}(G;A) \to Ext(G; [pt/A]) $ from the chain complex bicategory defined relative to this fixed chosen cover. If the cover is good, this functor is an equivalence of bicategories. Refining a simplical cover $U' \to U$ induces a (strict) functor $Z_{SM, U'}(G;A) \to Z_{SM, U}(G;A)$ which is also an equivalence of bicategories. Our construction is compatible with refinement and since any two covers have a common refinement the choice of cover is irrelevant. Equivalently, we may consider $Z_{SM}(G;A)$ to consist of the directed colimit over all simplicial covers. These considerations produce an equivalence $Z_{SM}(G;A) \to Ext(G; [pt/A])$. 

First we fix a good simplicial cover and construct a canonical central extension from a given cocycle representative $\lambda \in Z^3_{SM}(G;A)$. 
%This construction is quite similar to the construction of a group extension given a 2-cocycle.
A 3-cocycle has three non-trivial parts, which are smooth maps. 
\begin{align*}
	\lambda_3: U_3^{[1]} & \to A \\
	\lambda_2: U_2^{[2]} & \to A \\
	\lambda_1: U_1^{[3]} & \to A
\end{align*} 
We will see that these three data give rise to the three most important structures on a smooth 2-group. $\lambda_1$ will give rise to an $A$-gerbe over $G$, which will be the underlying Lie groupoid of $E^\lambda$. $\lambda_2$ will give rise to the multiplication bibundle for $E^\lambda$, and $\lambda_3$ will give rise to its associator. These three maps $(\lambda_1, \lambda_2, \lambda_3)$ form a cocycle in the double complex $C^{pq} = C^\infty( U_q^{[p+1]}; A)$, which computes the (smooth version of) Segal's group cohomology, thus they satisfy the following relations:
\begin{align*}
	& \delta_h \lambda_1 = 0 \\
	& \delta_v \lambda_1 = \delta_h \lambda_2 \\
	&\delta_v \lambda_2 = \delta_h \lambda_3 \\
	& \delta_v \lambda_3 = 0.
\end{align*} 
The first of these states that $\lambda_1$ is a \v{C}ech cocycle in $\check C^2_{U_1}( G; A)$. 

In Example \ref{ExampleCocycleGerbe} we constructed a $[pt/A]$-principal bundle (a.k.a. an $A$-gerbe) given precisely such a \v{C}ech cocycle. This principal bundle will be the underlying Lie groupoid of our smooth 2-group $E^\lambda$. Recall that the objects of $E^\lambda$ consist the manifold $U_1$ and the morphisms 
consist of the manifold $U_1^{[2]} \times A$, with composition being given by the formula,
\begin{align*}
	E^\lambda_1 \times_{U_1} E^\lambda_1 &\to E^\lambda_1 \\
	(u_0, u_1, a) \times (u_1, u_2, b) & \mapsto (u_0, u_2, a + b + \lambda_1(u_0, u_1, u_2)).
\end{align*}
There are several associated objects we may build out of the cocycle $\lambda$. The function $(d_0^*\lambda_1, d_2^*\lambda_1)$ from $U_2^{[3]}$ to $A \times A$ defines a \v{C}ech cocycle in $\check C^2_{U_2}(G \times G; A \times A)$ and hence gives rise to a $[pt/ A \times A]$-principal bundle $F^\lambda$ over $G \times G$. Here $d_i$ is the simplicial map in the simplicial manifold $U_\bullet$. There is a functor $(d_0, d_2): F^\lambda \to E^\lambda \times E^\lambda$, which is given on morphisms by the map,
\begin{align*}
	U_2^{[2]} \times A^2 &\to U_1^{[2]} \times A \times U_1^{[2]} \times A \\
	(v_0, v_1, a, b) & \mapsto (d_0(v_0), d_0(v_1), a) \times (d_2(v_1), d_2(v_1), b).
\end{align*}
This realizes $F^\lambda$ as a pull-back of the groupoid $E^\lambda \times E^\lambda$ and so becomes an equivalence upon bundlization.  

Similarly, the function $(d_0^* d_0^* \lambda_1, d_0^* d_2^* \lambda_1, d^*_2 d^*_2 \lambda_1)$ defines a \v{C}ech cocycle in $\check C^2_{U_3}(G^3; A^3)$ and hence a $[pt/A^3]$-principal bundle $H^\lambda$ over $G^3$ (the three maps $d_0d_0$, $d_2d_0$, and $d_2 d_2$ are the simplicial maps living over the three projections from $G^3$ to $G$). These Lie groupoids fit into a diagram of smooth functors. 
\begin{center}
\begin{tikzpicture}
	\node (LT) at (-2.5, 3) {$E^\lambda \times E^\lambda$};
	\node (MT) at (3, 3) {$F^\lambda$};
	\node (RT) at (6, 3) {$E^\lambda$};

	\node (LM) at (-2.5, 2) {$E^\lambda \times E^\lambda \times E^\lambda$};
	\node (MM) at (3, 2) {$H^\lambda$};
	\node (RM) at (6, 2) {$E^\lambda$};
	\node (LB) at (-2.5, 1) {$E^\lambda \times E^\lambda \times E^\lambda$};
	\node (MB) at (3, 1) {$H^\lambda$};
	\node (RB) at (6, 1) {$E^\lambda$};

	\draw [->] (MT) -- node [above] {$g = (d_0, d_2)$} (LT);
	\draw [->] (MT) -- node [above] {$\mu$} (RT);
	\draw [->] (MM) -- node [above] {$h = (d_0d_0, d_2d_0, d_1 d_2)$} (LM);
	\draw [->] (MM) -- node [above] {$f_1$} (RM);
	\draw [->] (MB) -- node [above] {$h = (d_0d_0, d_2d_0, d_1 d_2)$} (LB);
	\draw [->] (MB) -- node [above] {$f_2$} (RB);
	
\end{tikzpicture}
\end{center}
The left-pointing functors %are the obvious ones, and 
become equivalences after bundlization. The right-pointing functors are given explicitly by the following formulas: 
\begin{align*}
	\mu: F^\lambda_1= U_2^{[2]} \times A^2 & \to E^\lambda_1 = U_1^{[2]} \times A\\
	(v_0, v_1, a, b) & \mapsto (d_1(v_0), d_1(v_1), a + b + \lambda_2(v_0, v_1) \\
	f_1: H^\lambda_1 = U_3^{[2]} \times A^3 & \to E^\lambda_1 = U_1^{[2]} \times A \\
	(w_0, w_1, a, b, c) & \mapsto (d_1d_1(w_0), d_1d_1(w_1), a + b + c + \\
			& \quad \quad +  d^*_2 \lambda_2(w_0, w_1) + d_0^*\lambda_2(w_0, w_1)) \\
	f_2: H^\lambda_1 = U_3^{[2]} \times A^3 & \to E^\lambda_1 = U_1^{[2]} \times A \\
	(w_0, w_1, a, b, c) & \mapsto (d_1d_1(w_0), d_1d_1(w_1), a + b + c + \\
		& \quad \quad + d^*_1 \lambda_2(w_0, w_1) + d_3^*\lambda_2(w_0, w_1) )
\end{align*}
These are functors because the identity $\delta_v \lambda_1 = \delta_h \lambda_2$ holds. Moreover the identity $\delta_h \lambda_3 = \delta_v \lambda_2$ implies that $\lambda_3: U_3 \to A$ gives the components of a smooth natural transformation $a$ from $f_1$ to $f_2$. Turning these into bibundles and inverting $g$ will give us bibundle $M$ from $E^\lambda \times E^\lambda$ to $E^\lambda$. Inverting $h$ and composing with $f_1$ and $f_2$ yields two bibundles from $E^\lambda \times E^\lambda \times E^\lambda$ to $E^\lambda$. These are canonically identified with $M \circ (M \times 1)$ and $M \circ (1 \times M)$, respectively. The natural transformation $a$ induces a natural isomorphism $\alpha: M \circ (M \times 1) \to M \circ (1 \times M)$. The equation $\delta_v \lambda_3 = 0$ ensures that this associator satisfies the pentagon identity. 

More concretely, consider the composition $m: U_1 \times U_1 \to G \times G \stackrel{m}{\to} G$, and the induced fiber product $U_1 \times_G^m( U_1 \times U_1)$. This space admits a covering by the space $V = U_1 \times_G^{d_1} U_2 {}^{(d_0, d_2)}\times_{G \times G} (U_1 \times U_1)$. 
The data $(\lambda_1, \lambda_2)$ defines a function $\phi$ given by the following formula. 
\begin{align*}
	\phi: U_1 \times_G U^{[2]}_2 \times_{G \times G} (U_1 \times U_1) & \to A \\
	(u_0, v_0, v_1, u_2, u_3) & \mapsto \lambda_2(v_0, v_1) - \lambda_1( u_0, d_1 v_0, d_1 v_1) \\
	&\quad  - \lambda_1(d_0 v_0, d_0 v_1, u_1) \\
	& \quad  - \lambda_1( d_2 v_0, d_2 v_1, u_2) 
\end{align*}
This function defines a \v{C}ech cocycle $\check C^1_{V}( U_1 \times_G (U_1 \times U_1); A)$ and hence a corresponding $A$-bundle $M$ over $U_1 \times_G (U_1 \times U_1)$. This is the total space of the bibundle $M$ above. The necessary groupoid actions are easily constructed from this description. A compatible unit is straight forward to define and is determined up a contractible category of choices.  

A direct calculation shows that the sequence of homomorphisms 
\begin{equation*}
	[pt/A] \to E^\lambda \to G
\end{equation*}
is a central extension of smooth 2-groups, and thus provides a construction of a central extension from a cocycle $\lambda \in Z^3_{SM}(G;A)$. It remains to show that this construction can be extended to the entire cochain bicategory  $Z_{SM}(G;A)$. 

Let $\lambda, \lambda' \in Z^3_{SM}(G;A)$ be two cocycles. A 1-morphism in $Z_{SM}(G;A)$ from $\lambda$ to $\lambda'$ is precisely a cochain $\theta \in C_{SM}^2(G;A)$ such that $\delta \theta = \lambda - \lambda'$. This cochain has components $\theta_1: U_1^{[2]} \to A$ and $\theta_2: U_2 \to A$. In components the equation  $\delta \theta = \lambda - \lambda'$ becomes, 
\begin{align*}
	\delta_h \theta_1 & = \lambda_1 - \lambda_1' \\
	\delta_v \theta_1 + \delta_h \theta_2 & = \lambda_2 - \lambda_2' \\
	\delta_v \theta_2 & = \lambda_3 - \lambda_3'.
\end{align*}
This data gives rise to three functors
\begin{align*}
	p_E: & E^\lambda \to E^{\lambda'}\\
	p_F: & F^\lambda \to F^{\lambda'}\\
	p_H: & H^\lambda \to H^{\lambda'}\\
\end{align*}
and a natural isomorphism of functors $b: \mu' \circ p_F \to p_E \circ \mu$,  such that the following diagrams commute strictly,
\begin{center}
\begin{tikzpicture}
	\node (LT) at (0, 1.5) {$E^\lambda \times E^\lambda $};
	\node (LB) at (0, 0) {$E^{\lambda'} \times E^{\lambda'} $};
	\node (RT) at (2, 1.5) {$F^\lambda$};
	\node (RB) at (2, 0) {$F^{\lambda'}$};
	\draw [->] (LT) -- node [left] {$p_E \times p_E$} (LB);
	\draw [<-] (LT) -- node [above] {$g$} (RT);
	\draw [->] (RT) -- node [right] {$p_F$} (RB);
	\draw [<-] (LB) -- node [below] {$g'$} (RB);
	%\node at (0.5, 1) {$\ulcorner$};
	%\node at (1.5, 0.5) {$\lrcorner$};
	
	\node (LLT) at (6, 1.5) {$E^\lambda \times E^\lambda \times E^\lambda$};
	\node (LLB) at (6, 0) {$E^{\lambda'} \times E^{\lambda'} \times E^{\lambda'}$};
	\node (RRT) at (9, 1.5) {$H^\lambda$};
	\node (RRB) at (9, 0) {$H^{\lambda'}$};
	\draw [->] (LLT) -- node [left] {$p_E \times p_E \times p_E$} (LLB);
	\draw [<-] (LLT) -- node [above] {$h$} (RRT);
	\draw [->] (RRT) -- node [right] {$p_H$} (RRB);
	\draw [<-] (LLB) -- node [below] {$h'$} (RRB);
	%\node at (0.5, 1) {$\ulcorner$};
	%\node at (1.5, 0.5) {$\lrcorner$};
\end{tikzpicture}.
\end{center}

Explicitly these functors are defined as follows. Each of $p_E$, $p_F$, and $p_H$ is the identity on objects, and on 1-morphisms they are given by the formulas:
\begin{align*}
	p_{E,1}: E^\lambda_1 = U_1^{[2]} \times A & \to U_1^{[2]} \times A = E^{\lambda'}_1 \\
	(u_0, u_1, a) & \mapsto (u_0, u_1, a + \theta_1(u_0, u_1)) \\
	p_{F,1}: F^\lambda_1 = U_2^{[2]} \times A \times A & \to U_2^{[2]} \times A \times A = F^{\lambda'}_1 \\
	(v_0, v_1, a, b) & \mapsto (v_0, v_1, a + d_0^*\theta_1(v_0, v_1), b + d_2^* \theta_1(v_0, v_1)) \\
	p_{H,1}: H^\lambda_1 = U_1^{[2]} \times A \times A \times A & \to U_1^{[2]} \times A \times A \times A = H^{\lambda'}_1 \\
	(w_0, w_1, a, b, c) & \mapsto (w_0, w_1, a + d_0^* d_0^*\theta_1(w_0, w_1), \\
	& \qquad \qquad b + d_0^* d_2^*\theta_1(w_0, w_1), c + d_2^* d_2^*\theta_1(w_0, w_1)).
\end{align*}
The equation $\delta_h \theta_1 = \lambda_1 - \lambda_1'$ is equivalent to the statement that these formulas define functors. The natural transformation $b: \mu' \circ p_F \to p_E \circ \mu$ is given by $b = (\Delta \circ d_1, \theta_2): U_2 \to U_1^{[2]} \times A$. The equation $\delta_v \theta_1 + \delta_h \theta_2 = \lambda_2 - \lambda_2'$ is equivalent to the naturality of this natural transformation. 

We now turn each of these functors into bibundles. The functors $p_E$, $p_E$, and $p_H$ become equivalence bibundles, which are induced from the single bibundle $P: E^\lambda \to E^{\lambda'}$. The natural transformation $b$ becomes a natural isomorphism of bibundles $\beta: M' \circ (P \times P) \to  P \circ M$ from $E^\lambda \times E^\lambda$ to $E^{\lambda'}$. The final equation $\delta_v \theta_2 = \lambda_3 - \lambda_3'$ is equivalent to the commutativity of the diagram in Figure \ref{fig:2grpHom1}, which says that $P$ and $\beta$ are components of a 1-homomorphism from $E^\lambda$ to $E^{\lambda'}$. 

The remaining components of the 1-homomorphism $(P, \beta)$, namely those involving units of $E^\lambda$ and $E^{\lambda'}$, exist and are uniquely determined by the requirement that this be a homomorphism. The homomorphism $P: E^\lambda \to E^{\lambda'}$ is canonically a homomorphism over $G$ (indeed any two 1-homomorphisms from $E^\lambda \to G$ which are isomorphic are uniquely isomorphic). Moreover there is a unique 2-homomorphism $i' \cong P \circ i$ making $(P, \beta)$ into a morphism of central extensions, where $i: [pt/A] \to E^\lambda$ and $i': [pt/A] \to E^{\lambda'}$ are the previously constructed inclusions. In this way we obtain from each 2-cochain $\theta \in Z_{SM}(G;A)$ a homomorphism of central extensions which we denote $P_\theta$. 

If $\lambda$, $\lambda'$, and $\lambda''$ are three cocycles in $Z^3_{SM}(G;A)$ and $\theta$ and $\theta'$ are two cochains in $C^2_{SM}(G;A)$ which represent 1-morphisms from $\lambda$ to $\lambda'$ and from $\lambda'$ to $\lambda''$, respectively, then their composite in $Z_{SM}(G;A)$ is given by the sum $\theta' + \theta$. A simple calculation shows that the construction of the functors $p_E, p_F$ and $p_H$ preserves this composition strictly. For example $p_E^{\theta'} \circ p_E^{\theta} = p_E^{\theta + \theta'}$, on the nose. The natural isomorphisms $b^\theta$ and $b^{\theta'}$ also obey a strict composition identity:
\begin{equation*}
	b^{\theta' + \theta} = (p_E^{\theta'} * b^\theta) \circ (b^{\theta'} * p_F^\theta).
\end{equation*}
After bundlization, these strict equalities become the natural isomorphisms $P^{\theta'} \circ P^\theta \cong P^{\theta' + \theta}$ of homomorphisms of central extensions. The natural isomorphisms induced by $b^\theta$ and $b^{\theta'}$ also obey the expected composition law. A  similar calculation gives natural isomorphisms $P^{0} \cong id_{E^\lambda}$ for any cocycle $\lambda \in Z^3_{SM}(G;A)$. These natural isomorphisms are part of the data of the functor $Z_{SM}(G;A) \to Ext(G;[pt/A])$.

% 2-homomorphisms coming from 1-cochains. 

The rest of the data of this functor concerns 2-morphisms in $Z_{SM}(G;A)$. Let $\lambda, \lambda' \in Z^3_{SM}(G;A)$ be objects in $Z_{SM}(G;A)$ and let $\theta, \theta' \in C^2_{SM}(G;A)$ represent 1-morphisms from $\lambda$ to $\lambda'$. Thus $\delta \theta = \delta \theta' = \lambda - \lambda'$. A 2-morphism from $\theta$ to $\theta'$ is represented by a 1-cochain $\omega \in C^1_{SM}(G;A)$ such that $\delta \omega = \theta - \theta'$. Such a 1-cochain consists of a single function $\omega: U_1 \to A$ such that $-\delta_h \omega = \theta_1 - \theta_1'$ and $\delta_v \omega = \theta_2 - \theta_2'$. This gives rise to a natural isomorphism of functors $\eta: p_E^{\theta} \to p_E^{\theta'}$ whose components are $\eta(u) = (u, u, \omega(u))$. That this formula defines a natural isomorphism is equivalent to the equation  $-\delta_h \omega = \theta_1 - \theta_1'$. This induces a natural isomorphism of homomorphisms $\eta: P^\theta \to P^{\theta'}$. The second equation, $\delta_v \omega = \theta_2 - \theta_2'$ is equivalent to the commutativity of the first diagram in Figure \ref{fig:2grp2homs}. The commutativity of the second diagram in that figure is automatic in this case. The 2-homomorphism $\eta^\omega$ is clearly compatible with the projection to $G$ and is also compatible with the inclusion of $[pt/A]$ into $E^\lambda$ and $E^{\lambda'}$. Thus it defines a 2-morphism of central extensions. A similar calculation to before shows that at the level of natural transformations of functors, horizontal and vertical composition in $Z_{SM}(G;A)$ is preserves strictly. After bibundlization this provides the remaining natural isomorphisms and equalities which show that our assignment $Z_{SM}(G;A) \to Ext(G; [pt/A])$ is a functor between bicategories. 

% We're using the convention that  d = d_v + (-1)^{deg} d_h

If the simplicial cover used to define the Segal-Mitchison cohomology is {\em good}, then this is an equivalence of bicategories, which is equivalent to showing that it is essentially surjective on objects, essentially full on 1-morphisms and fully-faithful on 2-morphisms. To see this, recall that for any extension $E$, we know that there exists a sufficiently fine cover of $G$ over which the principal bundle $E \to G$ may be trivialized, a sufficiently fine cover of $G \times G$ over which the principal bundle $E \times E$ and the bibundle $M$ can both be trivialized, and a sufficiently fine cover over which the associator may be trivialized. In particular these may be trivialized over any good covering. Choosing explicit trivializations of these principal bundles and bibundles reproduces exactly the components of a Segal-Mitchison cocycle and applying the above construction reproduces (up to equivalence) the original extension. Similarly, any homomorphism $P: E^\lambda \to E^{\lambda'}$ between extensions arising from cocycles may be trivialized over a good cover and such a trivialization gives rise to an explicit cochain $\theta \in C^2_{SM}(G;A)$ representing a morphism between the corresponding cocycles. Applying our previous construction yields a homomorphism of extensions isomorphic to the original $P$. Finally if $E^\lambda$ and $E^{\lambda'}$ are extension arising from cocycles, $P^\theta, P^{\theta'}: E^\lambda \to E^{\lambda'}$ are morphisms of extensions arising from 2-cochains, then any 2-morphism $\omega: P^\theta \to P^{\theta'}$ arises from a unique 1-cochain $\omega$. This last statement is essentially equivalent to the fact that an isomorphism between trivialized principal $A$-bundles is given by a unique function function on the base. This completes our proof of Theorem \ref{ThmMain}.

% Question: Do we need to add a few remarks about refining covers?

 %Analogous but easier calculations show that a 2-cochain $\omega \in C^2_{SM}(G;A)$ gives rise to a 1-isomorphism between extensions corresponding to the 3-cocycles $\lambda$ and $\lambda + d(\omega)$, that a 1-cochain $\alpha$ gives rise to a 2-isomorphism between the 1-isomorphisms $\omega$ and $\omega + d(\alpha)$, and that these assemble into a functor $Z_{SM}(G;A) \to Ext(G;[pt/A])$. We leave the details to the interested reader. 
%Conversely, any 2-isomorphism, 1-isomorphism, or central extension is equivalent to one arising from this construction. In particular given an extension $E$, we know that there exists a sufficiently fine cover of $G$ over with the gerbe $E \to G$ may be trivialized, a sufficiently fine cover of $G \times G$ over which the gerbe $E \times E$ and the bibundle $M$ can both be trivialized, and a sufficiently fine cover over which the associator may be trivialized. Choosing our covers carefully permits us to assemble these trivializations into a smooth Segal-Mitchison cocycle $\lambda$ equipped with an equivalence $E \simeq E^\lambda$ of central extensions of 2-groups. Similar choices permit us  to represent 1- and 2-isomorphism as appropriate Segal-Mitchison cochains. 
\end{proof}

%\begin{remark}
%A similar construction works when the action of $G$ on $A$ is non-trivial. We get an extension of smooth 2-groups, but in this case it will not be a central extension. 
%\end{remark}

\subsection{String(n) as an Extension of Smooth 2-Groups}

The model of Segal-Mitchison topological group cohomology that we used Theorem \ref{ThmMain} computes this cohomology from the total complex associated to certain a double complex, and consequently  gives us several calculational tools. In particular there are the edge homomorphisms 
\begin{align*}
	H^i_\text{smooth}(G;A) & \to H^i_{SM}(G;A) \\
	H^{i+1}_{SM}(G;A) & \to \check H^i(G; \cO_A)
\end{align*}
where $\cO_A$ is the sheaf of smooth $A$-valued functions on the space $G$, and $H^i_\text{smooth}(G;A)$ denotes naive group cohomology with smooth cocycles. %The degree shift in this last map is due to the fact that the bottom row of the double complex is trivial. 
The construction in the previous section shows that in a central extension of smooth 2-groups,
\begin{equation*}
		[pt/A] \to E^\lambda \to G
\end{equation*}
coming from a cocycle $\lambda \in C^3(G;A)$, the underlying $[pt/A]$-principal bundle of $E^\lambda$ is classified by the image of $[\lambda_1]$ in $\check H^2(G; \cO_A)$. This allows us to identify the homotopy type of the geometric realization of $E^\lambda$. 

In fact we can realize the component $\lambda_1$ simplicially as a map of simplicial spaces,
\begin{equation*}
	\lambda_1: (G_{U_1})_\bullet \to K(A[2])_\bullet
\end{equation*}
Here $K(A[2])_\bullet$ is the simplicial topological abelian group associated to the chain complex with no differentials and with $A$ concentrated in degree two, and $(G_{U_1})_\bullet$ is the \v{C}ech simplicial manifold associated with the cover $U_1 \to G$. A direct calculation shows that the simplicial nerve of $E^\lambda$ is the pull-back in simplicial spaces,
\begin{center}
\begin{tikzpicture}[thick]
	\node (LT) at (0,1.5) 	{$(E^\lambda)_\bullet$ };
	\node (LB) at (0,0) 	{$(G_{U_1})_\bullet$};
	\node (RT) at (2,1.5) 	{$E(A[2])_\bullet$};
	\node (RB) at (2,0)	{$K(A[2])_\bullet$};
	\node at (.35, 1) {$\ulcorner$};
	\draw [->] (LT) --  node [left] {$$} (LB);
	\draw [->] (LT) -- node [above] {$$} (RT);
	\draw [->] (RT) -- node [right] {$$} (RB);
	\draw [->] (LB) -- node [below] {$\lambda_1$} (RB);
\end{tikzpicture}
\end{center}
where $E(A[2])_\bullet$ is the simplicial topological abelian group corresponding to the two term chain complex of topological abelian groups with $A$ in degrees one and two, and with differential the identity. This is a contractible chain complex and hence the geometric realization of the corresponding simplicial space is contractible. If fact it is the universal $K(A,1)$-bundle over the geometric realization $| K( A[2])_\bullet| \simeq K( A, 2)$. Since geometric realization of simplicial spaces commutes with fiber products, we have that
\begin{equation*}
	G \stackrel{\simeq}{\to} |G_U| \stackrel{|\lambda_1|}{\to} K(A, 2)
\end{equation*}
is the classifying map in the long exact fibration sequence. This identifies the homotopy type of the space $|E^\lambda|$.

\begin{theorem} \label{ThmConstructionOfString}
	Let $n \geq 5$. Then $H^3_{SM}(Spin(n); S^1) \cong H^4(BSpin(n)) \cong \Z$ and the central extension of smooth 2-groups corresponding to a generator gives a model for $String(n)$ as a smooth 2-group. 
\end{theorem}

\begin{proof}
It is well known that $ H^4(BSpin(n)) \cong \Z$ for $n \geq 5$, and so we know by Corollary \ref{CorCptG} that Segal cohomology $H^3_{SM}(Spin(n); S^1) \cong \Z$ and by Theorem \ref{ThmMain} that this classifies central extensions of smooth 2-groups, $[pt/S^1] \to E \to Spin(n)$. If $[\lambda] \in H^3_{SM}(Spin(n), S^1)$ is a class associated to a given central extensions, then by the above considerations we know that the topology of $|E|$ is determined by the image of $[\lambda]$ under the edge homomorphisms $H^3_{SM}(Spin(n); S^1) \to \check H^2(Spin(n); S^1) \cong \Z$, and moreover $|E^\lambda|$ will have the correct homotopy type precisely if $[\lambda]$ is mapped to a generator of $\check H^2(Spin(n); S^1)$. 

There are several ways to deduce that this edge homomorphism is surjective, and hence an isomorphism. For example, using the short exact sequence of smooth Lie groups $\Z \to \R \to \S^1$ and the induced long exact sequence in Segal-Mitchison cohomology, we see that the edge homomorphism is the same as the one for integral coefficients,
\begin{equation*}
	H^4(BG) \cong H^4_{SM}(G; \Z) \to \check H^3(G; \Z) \cong H^3(G).
\end{equation*}
Segal \cite{Segal70} identifies this map with the transfer map from the Serre spectral sequence, which is well known to be an isomorphism in these degrees for simply connected Lie groups like
 $G=Spin(n)$. Alternatively, we may re-examine the double complex which computes Segal-Mitchison cohomology, and use it to extract slightly more information. Associated to this double complex is a spectral sequence with $E_1$-term,
\begin{equation*}
	E_1^{p,q} = \check H^q( G^p ; \cO_A) \Rightarrow H^{p+q}_{SM}(G; A).
\end{equation*}
In the case that $G= Spin(n)$ and $A = S^1$, the $E_1$-term looks as follows
\begin{center}
\begin{tikzpicture}[thick]
%	\node () at (0,5) 	{$$ };
%	\node () at (0,4) 	{$$ };
	\node (H3) at (0,3)  {$0$};	
	\node (WA) at (0,2)  {$0$};	
	\node (A2) at (0,1) {$0$};	
	\node () at (0,0) 	{$S^1$ };
	%\node () at (0,-1) 	{$0$ };
	
%	\node () at (2,5) 	{$$ };
%	\node () at (2,4) 	{$$ };
	\node () at (2,3) 	{$\vdots $ };
	\node () at (2,2) 	{$\check H^2(G; S^1) $ };
	\node () at (2,1) 	{$0 $};
	\node () at (2,0) 	{$C^\infty(G; S^1) $ };
	
%	\node () at (3,5) 	{$$ };
%	\node () at (3,4) 	{$$ };
	\node () at (4.5,3) 	{$\vdots $ };
	\node () at (4.5,2) 	{$\check H^2(G^2; S^1) $ };
	\node (AA) at (4.5,1) 	{$0$};
	\node (A0) at (4.5,0) 	{$C^\infty(G^2; S^1)$ };
	
%	\node () at (4.5,5) 	{$$ };
%	\node () at (4.5,4) 	{$$ };
%	\node () at (4.5,3) 	{$$ };
	\node () at (7,2) 	{$\cdots$ }; %\check H^2(G^3; S^1)
	\node (T) at (7,1) 	{$0$};
	\node () at (7,0) 	{$C^\infty(G^3; S^1)$ };
	
%	\node () at (6,5) 	{$$ };
%	\node () at (6,4) 	{$$ };
%	\node () at (6,3) 	{$$ };
%	\node () at (9,2) 	{$$ };
	\node () at (9,1) 	{$\cdots$};
	\node () at (9,0) 	{$\cdots$ };

	\draw [->] (-2,-.5) --  node [below] {$p$} (9,-.5);
	\draw [->] (-1.5,-1) -- node [left] {$q$} (-1.5, 4);
	
%	\draw [->] (A0) --  node [fill=white] {$d_2$} (A2);
%	\draw [->] (H4) --  node [fill=white] {$d_2$} (AA);
%	\draw [->] (AA) --  node [fill=white] {$d_2$} (WA);
%	\draw [->] (T) --  node [fill=white] {$d_3$} (H3);
	
%	\draw [->] (H4) to [out=90, in = 30]  node [fill=white] {$d_4$} (H3);
	
\end{tikzpicture}
\end{center}

The cohomology of the first row under the $d^1$-differential is precisely the smooth group cohomology of $G$, i.e. the cohomology computed using smooth $S^1$-valued group cocycles. Since $G$ is compact and 1-connected, this is trivial in degrees larger then zero \cite{Hu52-1, Hu52-2, vanEst53, vanEst55, Mostow62}. This yields the following exact sequence,
\begin{equation*}
0 \to H^3_{SM}(Spin(n); S^1) \to \check H^2(G; S^1) \stackrel{d^1}{\to} \check H^2(G^2; S^1). 
\end{equation*}
The kernel of $d^1$ consists of the {\em primitive elements}. It is well known that for $G= Spin(n)$ every element of $\check H^2(G; S^1)$ is primitive in this sense. 
\end{proof}

%%%% The above shows that |E^\lambda| has the correct homotopy type and admits an A-\infty map to $G$.

\begin{remark}
There are two generators of this cohomology group and hence there are two associated central extensions. The corresponding smooth 2-groups which model String(n) are equivalent and this equivalence is a map of extensions  over $G$ which induces the order two automorphism of $[pt/S^1]$. 
\end{remark}

The above considerations also give a new conceptual re-interpretation of the notion of {\em multiplicative bundle gerbe}. In \cite{CJMSW05} multiplicative $S^1$-bundle gerbes over $G$ were introduced and shown to correspond to elements of $H^4(BG)$. We see from the above that a multiplicative bundle gerbe over $G$ may instead be viewed as a central extension of smooth 2-groups. 

\subsection{Concluding Remarks} 

Theorem \ref{ThmMain} provides a construction which produces a central extension of smooth 2-groups from a given smooth Segal-Mitchison cocycle and Theorem \ref{ThmConstructionOfString} shows that for any choice of generator of $H^3_{SM}(Spin(n); S^1)$ the corresponding central extension gives a model for String$(n)$. But in what sense is this a construction of String$(n)$? One may worry about the choices involved in this construction. However the choices don't matter. Theorem \ref{ThmMain} shows that the isomorphism classes of homomorphisms between any two extensions are in bijection with second Segal-Mitchison cohomology $H^2_{SM}(G; A)$, and the 2-homomorphisms between any two such homomorphisms form a torsor for $H^1_{SM}(G;A)$. In the case of String$(n)$, where $G = Spin(n)$ and $A = S^1$, both of these groups vanish, so that the bicategory of String$(n)$-extensions forms a contractible bicategory, i.e. any two extensions are equivalent, and any two homomorphisms realizing this equivalence are isomorphic via a unique 2-isomorphism.  This is the strongest possible uniqueness result one could hope for, and shows that the String$(n)$ 2-group extension is unique in precise analogy with the unique Spin$(n)$ extension of $SO(n)$. 

This brings us to the matter of other extensions. While it was sufficient to construct a model of String$(n)$, Theorem \ref{ThmMain} only classifies central extensions of smooth 2-groups of the particular form,
\begin{equation*}
	[pt/A] \to E \to G
\end{equation*}
where $G$ is an ordinary Lie group, and $A$ an ordinary abelian Lie group, viewed as a trivial $G$-module. This can be generalized, and the construction presented here works with negligible modification when the action of $G$ on $A$ is non-trivial. In this case we get an extension of smooth 2-groups, but it will not be a {\em central} extension. 

More generally, we would like to understand the bicategory of extensions of arbitrary smooth 2-groups $Ext(\G; \A)$, where $\G$ and $\A$ do not necessarily come from ordinary Lie groups. One could hope for some sort of cohomology theory which classifies theses extensions and which reduces to Segal-Mitchison cohomology when $\G =G$ is an ordinary Lie group and $\A = [pt/A]$ for $A$ an abelian Lie group. 

This hypothetical cohomology should take short exact sequences of smooth abelian 2-groups to long exact sequences and have other nice homological properties. Indeed such a cohomology does in fact exist, and  we may identify $Ext(\G; \A) \simeq H^2(\G; \A)$. Specializing to the case $\G = G$ and $\A = [pt/A]$, and using the short exact sequence $[pt/A] \to 0 \to A$ from Example \ref{Example:FavExtension}, we have isomorphisms,
\begin{equation*}
	H^2(G; [pt/A]) \cong H^3(G;A) = H^3_{SM}(G;A).
\end{equation*}
However the proper way to define this cohomology theory and deduce its properties requires developing the machinery of bicategorical homological algebra, in particular in a form that can be applied to the smooth setting. This would take us too far afield of current goals, but is a topic we take up in \cite{SchommerPries}.

\bibliographystyle{alpha}
\bibliography{StringGroup}

\end{document}